%% file: main.tex
\documentclass[10pt,a4paper]{article}

\usepackage[utf8]{inputenc}
\usepackage[british]{babel}

\usepackage[left=2cm,right=2cm,top=2cm,bottom=2cm]{geometry}

\usepackage{cancel}
\usepackage{amstext}
\usepackage{amssymb}
\usepackage{amsmath}
\usepackage{amsfonts}
\usepackage{mathrsfs}
\usepackage{bm}
\usepackage{nicefrac, xfrac}
\usepackage{upgreek}
\usepackage{booktabs}
\usepackage{array}
\usepackage{xfrac}

\usepackage{xcolor}
\usepackage{color}

\usepackage{graphicx}
\usepackage{pgf,tikz,pgfplots}
\pgfplotsset{compat=1.15}
\usetikzlibrary{arrows, arrows.meta, calc, math, shapes, positioning, fit, backgrounds}

\usepackage{wrapfig}

\usepackage[colorlinks=true]{hyperref}

\usepackage{ifpdf}

\usepackage[babel=true]{csquotes}
\MakeOuterQuote{"}
\usepackage{textcomp}

\usepackage{graphicx}
\usepackage{longtable}
\usepackage{caption}
\usepackage{subcaption}

\usepackage{svg}

\usepackage{url}
\usepackage{float}
\usepackage{multicol}

\usepackage[framed, amsthm, amsmath, thmmarks]{ntheorem}

\definecolor{mygray}{gray}{0.95}

\newcommand{\theoremcolorbox}[2]{\def\theoremframecommand{\fcolorbox{#1}{#2}}}
\theoremcolorbox{black}{mygray}
\newtheorem{theorem}{Theorem}[section]
\newtheorem{definition}[theorem]{Definition}
\newtheorem{corollary}[theorem]{Corollary}
\newtheorem{lemma}[theorem]{Lemma}
\newtheorem{proposition}[theorem]{Prop.}

\newtheorem{property}[theorem]{Property}

\theoremstyle{remark}
\newtheorem{remark}{Remark}
\newtheorem{example}{Example}

\makeatletter
\NewDocumentEnvironment{assum}{m O{}}{%
  \par\addvspace{\topsep}%
  \noindent
  \def\@currentlabel{#1}%
  {\normalfont\bfseries Assumption #1%
    \if\relax\detokenize{#2}\relax\else
      \ \normalfont(#2)%
    \fi
    \normalfont\bfseries.}%
  \nobreak\hspace{.5em}\ignorespaces
}{%
  \par\addvspace{\topsep}%
}
\NewDocumentEnvironment{cond}{m O{}}{%
  \par\addvspace{\topsep}%
  \noindent
  \def\@currentlabel{#1}%
  {\normalfont\bfseries Condition #1%
    \if\relax\detokenize{#2}\relax\else
      \ \normalfont(#2)%
    \fi
    \normalfont\bfseries.}%
  \nobreak\hspace{.5em}\ignorespaces
}{%
  \par\addvspace{\topsep}%
}
\makeatother

\makeatletter
\NewDocumentEnvironment{update}{m O{}}{%
  \par\addvspace{\topsep}%
  \noindent
  \def\@currentlabel{#1}%
  {\normalfont\bfseries Update #1%
    \if\relax\detokenize{#2}\relax\else
      \ \normalfont(#2)%
    \fi
    \normalfont\bfseries.}%
  \nobreak\hspace{.5em}\ignorespaces
}{%
  \par\addvspace{\topsep}%
}
\makeatother

\usepackage{enumerate}
\usepackage[shortlabels]{enumitem}

\numberwithin{equation}{section}

\usepackage[ruled,vlined]{algorithm2e}
\allowdisplaybreaks
\usepackage{algorithmicx, algpseudocode}
\usepackage[capitalize]{cleveref}

\input{macros.tex}

\newcommand*\Rp{\reals^p}

\newcommand*\mykappa[1]{ \kappa_{\scriptscriptstyle \text{#1}} }

\providecommand{\defeq}{\mathrel{:=}}
\providecommand{\kbar}{\overline{\kappa}}
\providecommand{\calI}{\mathcal{I}}

\providecommand{\calA}{\mathcal{A}}

\newcommand{\Rdelta}{\textsf{R-delta}}
\newcommand{\Rstep}{\textsf{R-step}}
\newcommand{\Rdfo}{\textsf{R-DFO}}
\newcommand{\Rgrad}{\textsf{R-grad}}
\newcommand{\Rrtr}{\textsf{R-RTR}}
\newcommand{\Rgrtr}{\textsf{R-gRTR}}

\newcommand{\PartI}{Part~I~\cite{Jeremy2026-I}}
\newcommand{\PartIII}{Part~III~\cite{Jeremy2026-III}}

\newcommand{\gone}{\gamma_1}
\newcommand{\gtwo}{\gamma_2}
\newcommand{\gthree}{\gamma_3}

\newif\ifnotes\notestrue
\notesfalse             %%  Uncomment this line to hide footnotes.  <----

\def\hfabian#1{}

\def\hjeremy#1{}

\usepackage[square,numbers]{natbib}
\setcitestyle{semicolon, aysep={,}}

\title{{\bf Structural conditions on trust-region radius updates.  \\
            Part II: second-order convergence and asymptotic radius behaviour} }

\author{{\bf J\'{e}r\'{e}my Rieussec} \quad {\bf Fabian Bastin}\\
	University of Montreal, Dept. Computing Science and Operational Research\\
	Email: jeremy.rieussec@umontreal.ca, bastin@iro.umontreal.ca \\
}

\date{\today}

\notesfalse  % review markup off for the public build
\begin{document}

\maketitle

\begin{abstract}
    Continuing the first-order framework of Part~I~\cite{Jeremy2026-I}, we ask not which structural
    conditions on $\{\Delta_k\}$ suffice for global convergence but what the radius does, and the
    answers separate the mechanisms according to whether the radius is \emph{tied to a criticality
    measure}.

    We first identify three asymptotic regimes: the criticality-anchored rules give
    $\sum_k \Delta_k^2/M_k < +\infty$, hence $\liminf_k \Delta_k = 0$, and
    $\Delta_k \to 0$ once the model Hessians are uniformly bounded, the step-driven rules inherit the
    stronger summability of the step series in the local regime, hence
    $\sum_k \Delta_k < +\infty$, and the fixed-factor rules keep
    $\Delta_k$ bounded away from zero. Near a first-order critical point with positive definite
    Hessian, convergence of a subsequence forces convergence of the full sequence, and for the
    criticality-anchored rules this holds with no hypothesis on the model Hessian.

    We then extend the framework to second-order criticality through
    $\tau_k = \max\{\norm{g_k}, -\lambda_k^-\}$, obtaining convergence to second-order critical
    points together with the worst-case complexity
    $\calO(\max\{\varepsilon_g^{-2}\varepsilon_H^{-1},\, \varepsilon_H^{-3}\})$, which reduces to
    $\calO(\varepsilon^{-3})$ for a common tolerance and is sharp for trust-region methods with
    quadratic models. 

    Finally we establish eventual inactivity of the trust-region constraint and superlinear
    convergence, isolating two properties of the radius rule through which alone the mechanism
    enters, the rate itself being mechanism-independent.
    The criticality-anchored and capped gradient-scaled rules reach inactivity only under
    a side condition, $\zeta > \kbar$ or $\overline{\mu} > \kbar$, in which $\kbar$ is a constant of
    the solution and cannot be checked in advance. We show that these conditions are necessary
    rather than artefacts of the analysis: on the single function $f(x) = \sin(x)/x$, a
    $C^\infty$ function with bounded Hessian and Lipschitz gradient, $\kbar$ is unbounded over the
    minimisers, so no parameter value meets the condition at all of them, and below the threshold
    the constraint is active at every sufficiently large iteration. Among the rules tied to a
    criticality measure, only those with an unbounded multiplier, the gradient-scaled rule and its
    retrospective variant, escape such a constant.
\end{abstract}

Unconstrained optimisation \textperiodcentered\ Trust-region methods \textperiodcentered\
Radius update rules \textperiodcentered\ Convergence analysis \textperiodcentered\
Second-order criticality \textperiodcentered\ Superlinear convergence \textperiodcentered\
Worst-case complexity

%%%%%%%%%%%%%%%%%%%%%%%%%%%%%%%%%%%%%%%%%%%%%%%%%%%%%%%%%%%%%%%%%%%%%%%%%%%%%%%
%%  Replacement for \section{Introduction}, down to and excluding
%%  \noindent\textbf{Notation.}  The two notation paragraphs are unchanged and
%%  stay where they are.
%%
%%  Structured after Part~I: run-in headings for the problem, the contributions,
%%  the relation to existing work, the outline and the notation.
%%
%%  CORRECTIONS CARRIED IN
%%
%%  1.  The overclaim.  "the one mechanism in the survey whose eventual
%%      inactivity requires no hypothesis involving a constant that depends on
%%      the solution" is false as stated, $\Rdelta$ and $\Rstep$ requiring none
%%      either.  The claim now carries its qualifier.
%%  2.  $\kbar = 4/m$, not $4/\lambda^*_{\min}$.
%%  3.  The sinc example is one-dimensional, $f(x) = \sin(x)/x$.
%%  4.  The \fabian note and every \mfabian wrapper are gone.
%%%%%%%%%%%%%%%%%%%%%%%%%%%%%%%%%%%%%%%%%%%%%%%%%%%%%%%%%%%%%%%%%%%%%%%%%%%%%%%

\section{Introduction} \label{sec:introduction}

Trust-region methods for the unconstrained problem $\min_{x \in \Rp} f(x)$ with $f$ smooth are
built from four components: a model of the objective, a rule for solving the subproblem restricted to a ball, a test deciding whether the trial step is accepted, and a rule for updating the radius of that ball. For the first two there are standard choices and a settled theory~\cite{ConnGoulToin00, NoceWrig06}. The last two are coupled, both reading the ratio $\rho_k$ of actual to predicted reduction, and it is the fourth that this paper is about.

\noindent\textbf{The radius update problem.}
Radius update mechanisms are documented individually and in detail~\citep[ch.~17]{ConnGoulToin00}, \citep{YuanFan01, Toint13, GrapigliaYuanYuan16}. What is missing is comparative treatment. Each rule is introduced with its own convergence proof, its own constants and its own standing assumptions, so it is hard to say which properties are intrinsic to the trust-region framework, which belong to a particular rule, and what is lost or gained by exchanging one for another.
\PartI{} argued that the variety conceals a common structure. Three
abstract conditions on the sequence $\{\Delta_k\}$, a lower bound relating the radius to the
smallest gradient norm seen so far, a contraction on unsuccessful iterations and a bounded
expansion on successful ones, suffice for $\lim_{k\to\infty}\norm{g_k} = 0$ together with the
worst-case evaluation complexity $\calO(\varepsilon^{-2})$. Five mechanisms drawn from the
literature satisfy them: 
the fixed-factor rule (\ref{update:multiplicative factor}),
the step-driven rule (\ref{update:step-dependent update}),
the retrospective rule of~\citet{BastMalmMoufToinToma10} (\ref{update:retrospective}),
the criticality-anchored rule of derivative-free optimisation~\cite{ConnScheVice09_introduction} (\ref{update: rdfo generic}),
and the gradient-scaled rule originating with~\citet{YuanFan01} (\ref{update: rgrad generic}). 

First-order global convergence is a weak requirement, met by mechanisms whose radii behave in qualitatively different ways: some drive $\Delta_k$ to zero, others keep it bounded away from zero, and the difference is invisible to the
first-order theory. 
It becomes visible in three places. 
In derivative-free optimisation the radius controls model accuracy, so $\Delta_k \to 0$ is necessary and its rate governs how quickly the model becomes reliable. 
In the stochastic setting the radius plays the same role for sampling accuracy, and analyses through martingales and renewal processes~\cite{BandScheVice14, ChenMeniSche18, BlanCartMeniSche19, CartSche18, CurtSche20, WangYuan2022, PaquSche2020, ShashaaniHashemiPasupathy2018, JinScheinbergXie2024, JinScheinbergXie2025} tie convergence rates
to lower bounds on the stepsize parameter. 
And near a minimiser the constraint must eventually stop binding if the method is to take full quasi-Newton steps and recover a fast local rate.

\noindent\textbf{Contributions and outline.}
Part~I leaves open what the radius does asymptotically, whether the structural framework extends
to second-order criticality, and under what conditions the constraint becomes inactive. We answer
all three, and the answers separate the mechanisms along a single axis, namely whether the radius
is \emph{tied to a criticality measure}.

Section~\ref{sec:asymptotic radius} identifies three asymptotic regimes. The criticality-anchored
rules give $\sum_k \Delta_k^2/M_k < +\infty$, hence $\liminf_k \Delta_k = 0$, and
$\Delta_k \to 0$ once the model Hessians are uniformly bounded, the step-driven rules inherit
the stronger summability $\sum_k \Delta_k < +\infty$ in the local regime by way of Powell's
theorem on the step series \cite{Powe75}, and the fixed-factor rules keep $\Delta_k$ bounded away from zero. 
The example $f(x) = e^{-x}$ separates the two groups: the criticality-anchored rules can drive the radius to zero, while the step-driven and the fixed-factor rules keep it bounded away from zero. Further properties follow for convergence to a first-order critical point $x^*$ with positive definite Hessian, Section~\ref{sec:local dynamics}. Convergence of a subsequence forces convergence of the full sequence as soon as the step is controlled by the gradient, see Condition~\ref{cond: s_k bounded by g_k}, independently of the radius rule (Theorem~\ref{thm: step bounded by gradient norm implies convergence}).

Sections~\ref{sec:second-order convergence} and~\ref{sec:second-order mechanism survey} extend the
structural framework to second-order criticality. Replacing the criticality measure $\norm{g_k}$
by $\tau_k = \max\{\norm{g_k}, -\lambda_k^-\}$ yields a second-order lower-bound condition that
underwrites convergence to second-order critical points together with the worst-case bound
of~\citet{CartGoulToin12}, and the mechanisms are verified against it one by one. The upgrade of
the measure is indispensable rather than cosmetic: with $\norm{g_k}$ retained in its place the
gradient-scaled rule could halt outright at an exact saddle and the criticality-anchored rule contracts
geometrically to zero, so both fail at precisely the points the second-order theory exists to
escape.

Section~\ref{sec:convergence rate} establishes eventual inactivity of the trust-region constraint
and the superlinear rate that follows. The analysis isolates two independent properties of the
radius rule, and shows that the mechanism enters the argument only through these
two inequalities. The gradient-scaled rule satisfies neither, yet admits a stronger argument of
its own. Its inactivity reduces to whether the multiplier $\mu_k$ crosses a threshold, which is dependent on the solution. Among the rules whose radius is tied to a criticality measure, it and its retrospective variant, both taken with an unbounded multiplier, are the only ones whose eventual inactivity requires no hypothesis involving a constant that depends on the solution. The criticality-anchored and capped gradient-scaled rules instead require $\zeta > \kbar$ and $\overline{\mu} > \kbar$, which cannot be arranged in advance.
Section~\ref{subsec:discussion on the convergence rate} shows that these side conditions
are necessary, and not an artefact of the route taken to them. On the single $C^\infty$ function
$f(x) = \sin(x)/x$, with bounded Hessian and Lipschitz gradient, $\kbar$ is unbounded over the
minimisers, so no finite $\zeta$ or $\overline{\mu}$ satisfies the condition at all of them; and at
a minimiser where it fails, the constraint is active at every sufficiently large iteration. The
existing analyses of these rules establish that they work; what is shown here is that the others
cannot dispense with a constant of the solution.
Once inactivity is established the rate theorem no longer refers to the mechanism at all. Along the way we show that the three formulations of the quasi-Newton consistency condition used across the literature~\cite{DennMore74, DennMore77, Stei83} are equivalent under the inexactness hypothesis on the subproblem solver.

The two halves of the paper are in tension on one point. The asymptotic
statements of Section~\ref{sec:asymptotic radius} for the gradient-scaled family assume a bounded
multiplier, whereas the variant that emerges best from Section~\ref{sec:convergence rate} is the
unbounded one, for which $\Delta_k = \mu_k\norm{g_k}$ is a product of a divergent and a vanishing
sequence. That rule is robust for the local question and not covered by the global asymptotics.
Section~\ref{sec:conclusion} returns to it.%

\noindent\textbf{What is new here.}
A comparative treatment earns its place by what it reorganises, and it is worth saying at the outset which
statements are ours and which are restated. The asymptotic regime of
Section~\ref{sec:asymptotic radius} is new: the summability
$\sum_k \Delta_k^2/M_k < +\infty$ for the criticality-anchored rules holds with no hypothesis
whatever on the iterates, where the existing analyses of those rules work in a neighbourhood of a
limit point. Condition~\ref{condition: lower bound Delta (2nd order)} and the five verifications of
Section~\ref{sec:second-order mechanism survey} are new for the step-driven, criticality-anchored
and retrospective families, and restate known results for the other two. The two properties of
Section~\ref{subsec: mechanism-free elements}, and with them the mechanism-independence of the rate,
are new. So are the necessity results of
Section~\ref{subsec:discussion on the convergence rate}, which have no counterpart we are aware of.
The local theory of Section~\ref{sec:local dynamics}, the second-order convergence and complexity
of Section~\ref{sec:second-order convergence}, and the ceiling on the multiplier of
Section~\ref{subsec:second-order gradient-scaled inactivity} are restated from the literature, with
attribution at each point; they are here because the comparison needs them in a common frame, not
because they are ours.

\noindent\textbf{Relation to existing work.}
The closest predecessor remains the survey of~\citet{Yuan15}.
Its Section~2 covers the present subject directly: it gives the step-driven rule
of~\citet{Hei03}, the gradient-scaled rule of~\citet{YuanFan01} together with the variant
$\Delta_k = \mu_k\norm{g_k}^\alpha$, and the retrospective rule
of~\citet{BastMalmMoufToinToma10}, and it observes that a radius bounded away from zero is
sufficient but not necessary for the superlinear rate, so that rules driving $\Delta_k$ to zero are
worth studying. That is where we part company. Organising the mechanisms by whether the radius
vanishes does not separate them: the step-driven rule has a vanishing radius in the local regime
and pays neither of the two prices the criticality-anchored rules pay. The separation is by whether
the radius is tied to a criticality measure, and this paper is an argument for that axis rather
than a complement to the earlier one.

The local theory we build on is classical. \citet{Stei83} and~\citet[ch.~6]{ConnGoulToin00}
establish convergence of the whole sequence from a convergent subsequence, and the superlinear
rate rests on the consistency conditions of~\citet{DennMore74, DennMore77}. These analyses fix a
radius rule. We instead isolate the two inequalities a rule must supply and prove the rate theorem
without reference to any rule, which is what allows the mechanisms to be compared rather than
merely listed. The same holds at second order, where~\citet{CartGoulToin12} obtain the complexity
bound for a particular update and we obtain it for any mechanism satisfying a condition, and in
the nonlinear stepsize control framework of~\citet{Toint13, GrapigliaYuanYuan15,
GrapigliaYuanYuan16}, where the radius follows a prescribed function of two quantities while our
conditions leave the rule free.

The gradient-scaled rule $\Delta_k = \mu_k\norm{g_k}$ is due to~\citet{YuanFan01}, with the
retrospective variant of~\citet{FanPan16}, and~\citet{Fan06, FanLu15, FanPan11} show
that a radius vanishing at the rate $\calO(\norm{x_k - x^*})$ is compatible with both global and
superlinear convergence. Their analyses establish the rate for that rule. We place it inside the
classification and identify what makes it exceptional, namely that its multiplier is the
radius-to-criticality ratio itself, which is why its inactivity result carries no side condition
while the other criticality-anchored rules do.%

%=======================================================================
%   SECTION 2: BACKGROUND AND STANDING ASSUMPTIONS
%=======================================================================
\section{Background and standing assumptions} \label{sec: Background}

This paper builds directly on Part~I~\cite{Jeremy2026-I}, from which we inherit the generic trust-region framework (Algorithm~1 of~\cite{Jeremy2026-I}, recalled in Appendix~\ref{app: BTR algorithm}), the notation, and a set of standing assumptions used throughout without further comment: boundedness below of $f$ (\ref{as: Bounded below}), Lipschitz continuity of the gradient with constant $L$ (\ref{as: Lipschitz continuous gradient}), first-order model agreement $m_k(0) = f(x_k)$ and $\nabla m_k(0) = \nabla f(x_k)$ (\ref{as: model agreement with function}), and the Cauchy decrease condition with constant $\mykappa{mdc}$ (\ref{as: Cauchy decrease}). For convenience, these are reproduced in Appendix~\ref{app: assumptions Part I}.

We also draw on the convergence theory of Part~I. Under the uniform Hessian bound \ref{as: model hessian} and weak admissibility (Conditions~\ref{cond: PartI C1} and~\ref{cond: PartI C2} of Part~I, recalled in Appendix~\ref{app: assumptions Part I}), the iterates satisfy $\lim_{k\to\infty} \|g_k\| = 0$: Theorem~3.2 of~\cite{Jeremy2026-I} yields $\liminf_{k\to\infty}\|g_k\| = 0$, which Theorem~2.4 of~\cite{Jeremy2026-I} upgrades to the full limit. Under the linear-growth assumption~\ref{as: Linear growth bound model Hessian}, strong admissibility (Conditions~\ref{cond: PartI C1} and~\ref{cond: PartI C3}) still guarantees $\liminf_{k\to\infty}\|g_k\| = 0$ (Theorem~3.4 of~\cite{Jeremy2026-I}), and yields the optimal worst-case complexity $\calO(\varepsilon^{-2})$ (Theorem~4.1 of~\cite{Jeremy2026-I}). Throughout, the radius obeys the lower bound
\begin{equation}
    \Delta_k \geq \frac{\mykappa{lbd}}{M_k} \min_{0 \leq i \leq k} \norm{g_i},
    \label{eqn:radius lower bound Part I}
\end{equation}
where $\{M_k\}$ is a non-decreasing sequence upper-bounding the model Hessian norm, typically $M_k = L + \max_{0 \leq i \leq k} \norm{H_i}$ (equation~(3.2) of~\cite{Jeremy2026-I}). 
Finally, we make repeated use of the five update mechanisms shown admissible in Part~I. Four are verified in its Section~6, the fixed-factor rule (R-delta), the step-driven rule (R-step), the criticality-anchored or DFO-like rule (R-DFO) and the gradient-scaled rule (R-grad), and its Section~5 analyses the retrospective update (R-RTR).
    
\subsection{Inexactness condition on the subproblem solver} \label{subsec:inexactness}

    Several of our results rely on a standard inexactness hypothesis for the trust-region subproblem. We state it here once and refer back throughout.

\begin{cond}{C.SubIn}[subproblem inexactness]
\label{cond:inexactness}
    
    Whenever the trust-region constraint is inactive, the trial step $s_k$ satisfies
    \begin{equation}
        \|H_k s_k + g_k\| \;\leq\; \xi_k\, \|g_k\|, \qquad \xi_k \in [0,1), \quad \xi_k \to 0 \text{ as } k \to \infty.
        \label{eqn:inexactness}
    \end{equation}
\end{cond}
Condition~\ref{cond:inexactness} is the trust-region counterpart of the inexact-Newton residual condition of~\citet{DembEiseStei82}, in which the forcing sequence $\xi_k$ controls the rate of convergence. When the Hessian is Lipschitz near the solution, a vanishing $\xi_k$ yields a superlinear rate, and $\xi_k = O(\|g_k\|)$ a quadratic one. We use it here to establish eventual inactivity of the trust-region constraint and superlinear convergence of the iterates. The condition enters in Section~\ref{sec: Background-qnewton}, for the equivalence of the three quasi-Newton consistency conditions, and in Section~\ref{sec:convergence rate}, for superlinear convergence. It is implied in practice by the truncated conjugate-gradient (Steihaug--Toint) method with a vanishing forcing sequence; see~\cite{Stei83, NoceWrig06}. The condition measures how accurately $s_k$ solves the Newton system $H_k s = -g_k$ relative to the current criticality $\|g_k\|$. The limiting case $\xi_k = 0$ corresponds to the exact (quasi-)Newton step $s_k = -H_k^{-1} g_k$, available when $H_k \succ 0$ and the step is feasible; when the model is indefinite or the constraint is active, $s_k$ is instead the corresponding constrained subproblem solution.

\subsection{Quasi-Newton consistency conditions} \label{sec: Background-qnewton}

The following assumption is a standard consistency condition for convergence of quasi-Newton methods~\cite{ConnGoulToin00}. 
\begin{assum}{H.OnCV}[operator-norm convergence]
\label{as: naive model Hessian CV}
    \[
        \lim_{k \to \infty} \norm{\nabla^2 m_k(x_k) - \nabla^2 f(x_k)} = 0 \quad \text{whenever} \quad \lim_{k \to \infty} \norm{g_k} = 0.
    \]
    
\end{assum}
It requires that the model Hessian $H_k$ approximates the true Hessian $\nabla^2 f(x_k)$ in operator norm whenever $\norm{g_k} \to 0$. This is a strong condition. It requires that $H_k$ approximates $\nabla^2 f(x_k)$ well in all directions.
As explained in~\citet{DennMore74}, this assumption supports $Q$-superlinear convergence by ensuring $\{H_k\} \to \nabla^2 f(x^*)$. 
However, this consistency condition is sufficient but not necessary. Superlinear convergence requires only that $H_k$ approximates $\nabla^2 f(x^*)$ accurately along the directions $s_k$. 
The following conditions on the model Hessian $H_k$ are standard weaker formulations in the literature on superlinear convergence~\cite{DennMore74, DennMore77, DembEiseStei82, Toin79, CartGoulToin12, FanPan16}. They are stated separately because they appear in different works under different forms. We show in Theorem~\ref{thm: equivalence quasi-Newton conditions} that they are equivalent whenever $\{x_k\} \to x^*$ with $\nabla^2 f(x^*)$ positive definite and Condition~\ref{cond:inexactness} holds.

\begin{assum}{H.Gr}[gradient residual condition]
\label{as: Gradient residual condition}
    
    \begin{equation}
        \lim_{k \to \infty} \frac{\|g_k + \nabla^2 f(x_k)\, s_k\|}{\|s_k\|} = 0
        \quad \text{whenever} \quad \lim_{k \to \infty} \|g_k\| = 0.
        \label{eqn: Gradient residual condition}
    \end{equation}
\end{assum}

\begin{assum}{H.Hs}[Hessian--step residual condition]
\label{as: Hessian-step residual condition}
    
    \begin{equation}
        \lim_{k \to \infty} \frac{\|(H_k - \nabla^2 f(x_k))\, s_k\|}{\|s_k\|} = 0
        \quad \text{whenever} \quad \lim_{k \to \infty} \|g_k\| = 0.
        \label{eqn: Hessian-step residual condition}
    \end{equation}
\end{assum}

\begin{assum}{H.Fd}[finite-difference consistency]
\label{as: Finite-difference Hessian consistency}
    
    \begin{equation}
        \lim_{k \to \infty} \frac{\|g(x_k + s_k) - g_k - H_k s_k\|}{\|s_k\|} = 0
        \quad \text{whenever} \quad \lim_{k \to \infty} \|g_k\| = 0.
        \label{eqn:Finite-difference Hessian consistency}
    \end{equation}
\end{assum}

The three are equivalent whenever $\{x_k\} \to x^*$ with $\nabla^2 f(x^*)$ positive definite,
$\norm{s_k} \to 0$ and Condition~\ref{cond:inexactness} holds; we use them interchangeably below
on that basis. The statement and its proof are deferred to
Appendix~\ref{app: quasi-Newton equivalence}, since nothing in the argument concerns the radius
rule. \ref{as: Gradient residual condition} is the formulation used
by~\citet{DennMore74,DennMore77} for quasi-Newton methods, \ref{as: Finite-difference Hessian consistency}
that of~\citet{Stei83}, and all three are weaker than the condition
$\|H_k - \nabla^2 f(x^*)\| \to 0$ used in~\cite{ConnGoulToin00}: superlinear convergence requires
only that $H_k$ approximate $\nabla^2 f(x^*)$ accurately \emph{along the search directions
$s_k$}, not in operator norm.

% \begin{remark}
%     \label{rem: AH B C does not need inexactness}
%     The Theorem and proof are deferred to Appendix~\ref{app: quasi-Newton equivalence} because they do not concern the radius rule.
%     The equivalence \ref{as: Hessian-step residual condition} $\Leftrightarrow$ \ref{as: Finite-difference Hessian consistency} of Step~1 does \emph{not} depend on the inexactness Condition~\ref{cond:inexactness} nor on the quasi-Newton structure of $s_k$. It holds for any step $s_k$ along which $\nabla^2 f$ is continuous and $\norm{s_k} \to 0$. The equivalence \ref{as: Hessian-step residual condition} $\Leftrightarrow$ \ref{as: Gradient residual condition} of Step~2, by contrast, requires control on the inexactness residual $H_k s_k + g_k$; the quasi-Newton identity $H_k s_k = -g_k$ used in~\cite{DennMore74,DennMore77} is the special case $\xi_k = 0$.
% \end{remark}

%=======================================================================
%   SECTION: ASYMPTOTIC BEHAVIOUR OF THE RADIUS
%=======================================================================
\section{Asymptotic behaviour of the trust-region radius: weak assumption} \label{sec:asymptotic radius}

    Part~I settled the question of \emph{whether} the gradient norm converges: under the structural conditions \ref{cond: PartI C1}, \ref{cond: PartI C2} and~\ref{cond: PartI C3} on the radius update rule (recalled in Appendix~\ref{app: assumptions Part I}), every admissible mechanism drives $\norm{g_k}$ to zero. It left open the complementary question, on which the local theory of Section~\ref{sec:local dynamics} depends, of what becomes of the radius itself. Two mechanisms may share the same first-order guarantee and yet differ entirely in their asymptotic regime, one letting $\Delta_k$ vanish, the other holding it away from zero, with opposite consequences for the local rate.

    Without assumptions on the iterates $\{x_k\}$, no asymptotic property of $\{\Delta_k\}$ follows from the trust-region framework alone. Each follows, when it follows at all, from the update mechanism, and the mechanisms divide accordingly into two groups.

    The first group is the criticality-anchored one: R-DFO ties the radius to the criticality measure through the test $\Delta_k \lessgtr \zeta \norm{g_k}$, and R-grad enforces the proportionality $\Delta_k = \mu_k \norm{g_k}$ outright. For these two rules, Propositions~\ref{prop: behaviour DFO-like radius} and~\ref{prop: CV rate gradient-scaled radius} establish
    \[
        \sum_{k=0}^{+\infty} \frac{\Delta_k^2}{M_k} < +\infty,
        \qquad\text{whence}\qquad
        \liminf_{k \to +\infty} \Delta_k = 0,
    \]
    and $\Delta_k \to 0$ when the model Hessians are uniformly bounded. These conclusions require \emph{no} assumption on the iterates: no boundedness of $\{x_k\}$, no existence of accumulation points.
    R-grad carries this one step further: since the radius and the gradient norm are proportional by construction, the summability travels in both directions, and Corollary~\ref{cor:gradient-scaled-liminf} recovers $\lim_k \norm{g_k} = 0$ as a by-product, {a first-order guarantee obtained, when the model Hessians are uniformly bounded, from the mechanism's own structure rather than from the admissibility machinery of Part~I.}

    The second group, the rules whose radius is driven by the previous radius (R-delta) or the realised step (R-step), admits no such conclusion at this level of generality. Their update carries no information about criticality, so nothing in the mechanism prevents $\Delta_k$ from remaining bounded away from zero; and, as the next subsections make precise, nothing forces it to. Establishing an asymptotic regime for them requires assumptions on the iterates, which is the subject of Section~\ref{sec:local dynamics}.

    \subsection{Criticality-anchored update (DFO-like)} \label{subsec:asymptotic DFO-like}

        The R-DFO rule (Appendix~\ref{app: R-DFO}) enforces the criticality anchoring explicitly:
        the test $\Delta_k \leq \zeta \norm{g_k}$ decides whether the radius is expanded or
        contracted, so a radius large relative to the current criticality is corrected at once.
        \citet[Lem.~2.9]{GratRoyeViceZhan17} prove square-summability for a derivative-free
        algorithm carrying the same test in the form $\norm{g_k} \geq \alpha \Delta_k$, under a
        uniform bound $\norm{H_k} \leq \mykappa{umh}$ and $\zeta = \alpha^{-1}$, and with an explicit constant. Their
        statement is part~(ii) below with $M_k$ constant, and the two proofs partition the iterations the same way.  
        
        The following proposition shows that, under \ref{as: Linear growth bound model Hessian} the summability survives in the form $\sum_k \Delta_k^2 / M_k < +\infty$, with no uniform bound on the model Hessian and no assumption on the iterates.

    \begin{proposition}
    \label{prop: behaviour DFO-like radius}
    Let \ref{as: Bounded below}, \ref{as: Lipschitz continuous gradient},
    \ref{as: model agreement with function}, \ref{as: Linear growth bound model Hessian}
    and \ref{as: Cauchy decrease} hold. Then the update R-DFO satisfies
    \begin{equation}
        \sum_{k=0}^{+\infty} \frac{\Delta_k^2}{M_k} \;<\; +\infty .
        \label{eqn: dfo radius summability}
    \end{equation}
    Consequently,
    \begin{enumerate}[(i)]
        \item $\displaystyle \liminf_{k \to +\infty} \Delta_k = 0$;
        \item if in addition $\norm{H_k} \leq \mykappa{umh}$ for every $k$
              (\ref{as: model hessian}), then
              $\displaystyle \sum_{k=0}^{+\infty} \Delta_k^2 < +\infty$ and
              $\displaystyle \lim_{k \to +\infty} \Delta_k = 0$.
    \end{enumerate}
\end{proposition}

\begin{proof}
    Partition the iterations into
    \begin{equation*}
        \calA = \big\{ k \geq 0 \;:\; \rho_k \geq \eta_1 \ \text{ and } \ \Delta_k \leq \zeta \norm{g_k} \big\}
        \qquad \text{and} \qquad
        \calR = \NN \setminus \calA ,
    \end{equation*}
    so that $\calA$ collects the iterations on which the R-DFO rule expands the radius. On $\calR$
    the rule contracts it: if $\rho_k < \eta_1$ then $\Delta_{k+1} \leq \gamma_2 \Delta_k$,
    and if $\rho_k \geq \eta_1$ with $\Delta_k > \zeta \norm{g_k}$ then
    $\Delta_{k+1} = \gamma_2 \Delta_k$, the criticality step of
    \ref{update: rdfo generic} being deterministic.
    In either case
    \begin{equation}
        \Delta_{k+1} \leq \gamma_2 \Delta_k , \qquad k \in \calR .
        \label{eqn: dfo contraction on R}
    \end{equation}

    Suppose first that $\calA$ is finite, and let $K$ be such that $k \in \calR$ for every
    $k \geq K$. Iterating~\eqref{eqn: dfo contraction on R} gives
    $\Delta_{K+j} \leq \gamma_2^{\,j} \Delta_K$ for $j \geq 0$, so that $\Delta_k \to 0$. As
    $\{M_k\}$ is positive and non-decreasing,
    \begin{equation*}
        \sum_{k=0}^{+\infty} \frac{\Delta_k^2}{M_k}
        \;\leq\; \sum_{k=0}^{K-1} \frac{\Delta_k^2}{M_k}
                 + \frac{\Delta_K^2}{M_K} \sum_{j=0}^{+\infty} \gamma_2^{\,2j}
        \;=\; \sum_{k=0}^{K-1} \frac{\Delta_k^2}{M_k}
              + \frac{1}{1 - \gamma_2^2} \, \frac{\Delta_K^2}{M_K}
        \;<\; +\infty ,
    \end{equation*}
    which is~\eqref{eqn: dfo radius summability}.

        Suppose now that $\calA$ is infinite, and let $k \in \calA$, so that
    $\norm{g_k} \geq \Delta_k/\zeta$ and $\rho_k \geq \eta_1$. Using the first of these twice in
    \ref{as: Cauchy decrease}, with $\norm{H_k} \leq M_k$,
    \begin{equation}
        f(x_k) - f(x_{k+1})
        \;\geq\; \eta_1 \frac{\mykappa{mdc}}{2} \norm{g_k}
                 \min\!\left( \Delta_k, \frac{\norm{g_k}}{M_k} \right)
        \;\geq\; \eta_1 \frac{\mykappa{mdc}}{2 \zeta^2}
                 \min\!\left( \zeta, \frac{1}{M_k} \right) \Delta_k^2
        \;\geq\; c \, \frac{\Delta_k^2}{M_k} , \qquad k \in \calA ,
        \label{eqn: dfo decrease on A}
    \end{equation}
    with $c \defeq \eta_1 \mykappa{mdc} \min\{1, \zeta M_0\} / (2 \zeta^2) > 0$. The last step is
    $\min\{\zeta, 1/M_k\} \geq \min\{1, \zeta M_0\}/M_k$. For $1/M_k \leq \zeta$ it reads
    $\min\{1, \zeta M_0\} \leq 1$, and otherwise it follows from $\zeta \geq \zeta M_0/M_k$,
    the sequence $\{M_k\}$ being non-decreasing.
    The sequence $\{f(x_k)\}$ is non-increasing and bounded below by $\mykappa{lbf}$ under
    \ref{as: Bounded below}, whence
    \begin{equation}
        c \sum_{k \in \calA} \frac{\Delta_k^2}{M_k}
        \;\leq\; \sum_{k \in \calA} \big( f(x_k) - f(x_{k+1}) \big)
        \;\leq\; f(x_0) - \mykappa{lbf}
        \;<\; +\infty .
        \label{eqn: dfo summable on A}
    \end{equation}

    It remains to account for the iterations of $\calR$ lying between two consecutive elements of
    $\calA$. Let $k_0 \in \calA$ and put $j^{+} = \min\{ j \geq 1 : k_0 + j \in \calA \}$, which is
    finite because $\calA$ is infinite. As $k_0 \in \calA$ the rule expands, so
    $\Delta_{k_0+1} \leq \gamma_3 \Delta_{k_0}$; and for $1 \leq j \leq j^{+} - 1$ we have
    $k_0 + j \in \calR$, so that~\eqref{eqn: dfo contraction on R} gives
    \begin{equation*}
        \Delta_{k_0+j} \;\leq\; \gamma_2^{\,j-1} \Delta_{k_0+1}
        \;\leq\; \gamma_2^{\,j-1} \gamma_3 \Delta_{k_0} .
    \end{equation*}
    Using once more that $\{M_k\}$ is non-decreasing, and with the convention
    $\sum_{j=1}^{0} a_j = 0$,
    \begin{equation}
        \sum_{j=0}^{j^{+}-1} \frac{\Delta_{k_0+j}^2}{M_{k_0+j}}
        \;\leq\; \frac{\Delta_{k_0}^2}{M_{k_0}}
                 \left( 1 + \gamma_3^2 \sum_{j=1}^{j^{+}-1} \gamma_2^{\,2(j-1)} \right)
        \;\leq\; \left( 1 + \frac{\gamma_3^2}{1 - \gamma_2^2} \right) \frac{\Delta_{k_0}^2}{M_{k_0}} .
        \label{eqn: dfo block bound}
    \end{equation}
    The blocks $\{ k_0, \dots, k_0 + j^{+} - 1 \}$, indexed by $k_0 \in \calA$, are pairwise
    disjoint and cover $\NN$ up to the finitely many indices preceding the first element of $\calA$.
    Summing~\eqref{eqn: dfo block bound} over $k_0 \in \calA$ and
    invoking~\eqref{eqn: dfo summable on A} therefore gives
    \begin{equation*}
        \sum_{k=0}^{+\infty} \frac{\Delta_k^2}{M_k}
        \;\leq\; \sum_{k=0}^{k_1-1} \frac{\Delta_k^2}{M_k}
                 + \left( 1 + \frac{\gamma_3^2}{1 - \gamma_2^2} \right)
                   \sum_{k \in \calA} \frac{\Delta_k^2}{M_k}
        \;<\; +\infty ,
    \end{equation*}
    where $k_1 = \min \calA$. This proves~\eqref{eqn: dfo radius summability} in both cases.

    To prove~(i), assume $\liminf_{k \to +\infty} \Delta_k = c_0 > 0$. There is then $K$ with
    $\Delta_k \geq c_0/2$ for every $k \geq K$, so that
    \begin{equation*}
        \sum_{k \geq K} \frac{\Delta_k^2}{M_k}
        \;\geq\; \frac{c_0^2}{4} \sum_{k \geq K} \frac{1}{M_k} \;=\; +\infty
    \end{equation*}
    by \ref{as: Linear growth bound model Hessian}, contradicting
    \eqref{eqn: dfo radius summability}.

    Finally, let $\norm{H_k} \leq \mykappa{umh}$ for every $k$. Then $M_k \leq L + \mykappa{umh}$
    for every $k$, and~\eqref{eqn: dfo radius summability} gives
    \begin{equation*}
        \sum_{k=0}^{+\infty} \Delta_k^2
        \;\leq\; \big( L + \mykappa{umh} \big) \sum_{k=0}^{+\infty} \frac{\Delta_k^2}{M_k}
        \;<\; +\infty ,
    \end{equation*}
    so that $\Delta_k \to 0$. This proves~(ii) and completes the proof.
\end{proof}

    \subsection{Gradient-scaled update} \label{subsec:asymptotic gradient-scaled}

    The R-grad rule (Appendix~\ref{app: R-grad}) anchors the radius to criticality in the strongest possible sense. Rather than testing $\Delta_k$ against $\zeta\norm{g_k}$, it sets $\Delta_k = \mu_k \norm{g_k}$ identically and updates the scaling coefficient $\mu_k$ in its place. 
    The next proposition exploits the connection in both directions.  Over the successful
    iterations it bounds $\sum_{k \in \calS} \norm{g_k}^2/M_k$ whether or not $\mu_k$ is bounded
    above, which Proposition~\ref{prop: behaviour DFO-like radius} does not deliver; under a cap on
    the multiplier it then recovers the conclusions of that proposition for
    $\sum_k \Delta_k^2/M_k$.

    \begin{proposition}
        \label{prop: CV rate gradient-scaled radius}
        Let \ref{as: Bounded below}, \ref{as: Lipschitz continuous gradient},
        \ref{as: model agreement with function}, \ref{as: Linear growth bound model Hessian}
        and \ref{as: Cauchy decrease} hold, and let $\eta > 0$. Then the gradient-scaled update
        R\nobreakdash-grad satisfies
        \begin{equation}
            \sum_{k \in \calS} \frac{\norm{g_k}^2}{M_k} < +\infty ,
            \label{eqn: grad successful summability}
        \end{equation}
        where $\calS = \{k \geq 0 : \rho_k \geq \eta\}$. If, in addition, there exists
        $\overline{\mu} > 0$ such that $\mu_k \leq \overline{\mu}$ for every $k$, then
        \begin{enumerate}[(i)]
            \item $\displaystyle \sum_{k=0}^{+\infty} \frac{\Delta_k^2}{M_k} < +\infty$ and $\displaystyle \liminf_{k \to +\infty} \Delta_k = 0$; 
            \item if moreover $\norm{H_k} \leq \mykappa{umh}$ for every $k$, then
                $\displaystyle \sum_{k=0}^{+\infty} \Delta_k^2 < +\infty$ and
                $\displaystyle \lim_{k \to +\infty} \Delta_k = 0$.
        \end{enumerate}
    \end{proposition}
    \begin{proof}
    Proposition~6.6 of Part~I provides a constant $\mykappa{lbd} > 0$ with
    $\mu_k \geq \mykappa{lbd}/M_k$ for every $k$, so that
    $\Delta_k = \mu_k \norm{g_k} \geq (\mykappa{lbd}/M_k) \norm{g_k}$. Let $k \in \calS$. Since
    $\norm{H_k} \leq M_k$, assumption~\ref{as: Cauchy decrease} yields
    \begin{equation*}
        f(x_k) - f(x_{k+1})
        \;\geq\; \eta \, \frac{\mykappa{mdc}}{2} \, \norm{g_k}
                 \min\!\left(\frac{\norm{g_k}}{\norm{H_k}}, \Delta_k\right)
        \;\geq\; \eta \, \frac{\mykappa{mdc} \, \mykappa{lbd}}{2} \, \frac{\norm{g_k}^2}{M_k} .
    \end{equation*}
    The sequence $\{f(x_k)\}$ is non-increasing and bounded below by $\mykappa{lbf}$ under
    \ref{as: Bounded below}, whence
    \begin{equation*}
        \eta \, \frac{\mykappa{mdc} \, \mykappa{lbd}}{2} \sum_{k \in \calS} \frac{\norm{g_k}^2}{M_k}
        \;\leq\; \sum_{k \in \calS} \big( f(x_k) - f(x_{k+1}) \big)
        \;\leq\; f(x_0) - \mykappa{lbf} \;<\; +\infty ,
    \end{equation*}
    which is~\eqref{eqn: grad successful summability}.

    Assume henceforth $\mu_k \leq \overline{\mu}$ for every $k$, so that
    $\Delta_k \leq \overline{\mu} \norm{g_k}$ and, by~\eqref{eqn: grad successful summability},
    \begin{equation}
        \sum_{k \in \calS} \frac{\Delta_k^2}{M_k}
        \;\leq\; \overline{\mu}^{\,2} \sum_{k \in \calS} \frac{\norm{g_k}^2}{M_k} \;<\; +\infty .
        \label{eqn: grad successful radius summability}
    \end{equation}

    The set $\calS$ is infinite. Suppose otherwise: there is an index $k_0$
    such that every iteration $k_0 + j$, $j \geq 1$, is unsuccessful. The update rule then gives
    $\mu_{k_0+j} \leq {\gamma_2^{\,j}} \mu_{k_0}$, while
    $\mu_{k_0+j} \geq \mykappa{lbd}/M_{k_0+j}$, so
    that%
    \begin{equation*}
        \sum_{j \geq 1} \frac{1}{M_{k_0+j}}
        \;\leq\; \frac{\mu_{k_0}}{\mykappa{lbd}} \sum_{j \geq 1} {\gamma_2^{\,j}}
        \;=\; \frac{\mu_{k_0}}{\mykappa{lbd}} \, {\frac{\gamma_2}{1 - \gamma_2}} \;<\; +\infty ,
    \end{equation*}
    contradicting \ref{as: Linear growth bound model Hessian}. The same argument shows that no
    run of consecutive unsuccessful iterations is infinite.

    It remains to bound the radius on the unsuccessful runs. If all iterations are eventually successful,
    \eqref{eqn: grad successful radius summability} already gives~(i). Otherwise, let $k_0 \in \calS$
    be such that $k_0 + 1 \notin \calS$, and set $j^{+} = \min\{ j \geq 1 : k_0 + j \in \calS \}$,
    which is finite by the preceding paragraph. For $1 \leq j \leq j^{+} - 1$ the iterate is
    unchanged, so $\norm{g_{k_0+j}} = \norm{g_{k_0+1}}$, and the rule contracts the scaling
    coefficient, $\mu_{k_0+j} \leq {\gamma_2^{\,j-1}} \mu_{k_0+1}$; hence
    \begin{equation*}
        \Delta_{k_0+j} \;=\; \mu_{k_0+j} \norm{g_{k_0+j}}
        \;\leq\; {\gamma_2^{\,j-1}} \mu_{k_0+1} \norm{g_{k_0+1}} .
    \end{equation*}
    By \ref{as: Lipschitz continuous gradient} and $\norm{s_{k_0}} \leq \Delta_{k_0} \leq
    \overline{\mu} \norm{g_{k_0}}$,
    \begin{equation*}
        \norm{g_{k_0+1}} \;\leq\; \norm{g_{k_0}} + L \norm{s_{k_0}}
        \;\leq\; \big( 1 + \overline{\mu} L \big) \norm{g_{k_0}} ,
    \end{equation*}
    and $\mu_{k_0+1} \leq \gamma_3 \mu_{k_0}$ on every branch of
    \ref{update: rgrad generic}, the expansion one included, so that
    $\Delta_{k_0+j} \leq {\gamma_2^{\,j-1}} \gamma_3 (1 + \overline{\mu} L) \Delta_{k_0}$ for
    $1 \leq j \leq j^{+} - 1$. As $\{M_k\}$ is non-decreasing, it follows that
    \begin{equation*}
        \sum_{j=0}^{j^{+}-1} \frac{\Delta_{k_0+j}^2}{M_{k_0+j}}
        \;\leq\; \left( 1 + \frac{\gamma_3^2 \, (1 + \overline{\mu} L)^2}{{1 - \gamma_2^2}} \right)
                 \frac{\Delta_{k_0}^2}{M_{k_0}} ,
    \end{equation*}
    with the convention $\sum_{j=1}^{0} a_j = 0$. The blocks
    $\{k_0, \dots, k_0 + j^{+} - 1\}$, indexed by $k_0 \in \calS$, cover all but finitely many
    iterations, so summing the last display over $k_0 \in \calS$ and
    invoking~\eqref{eqn: grad successful radius summability} gives
    \begin{equation*}
        \sum_{k=0}^{+\infty} \frac{\Delta_k^2}{M_k}
        \;\leq\; \sum_{k=0}^{k_1-1} \frac{\Delta_k^2}{M_k} \;+\;
                 \left( 1 + \frac{\gamma_3^2 \, (1 + \overline{\mu} L)^2}{{1 - \gamma_2^2}} \right)
                 \sum_{k \in \calS} \frac{\Delta_k^2}{M_k} \;<\; +\infty ,
    \end{equation*}
    where $k_1 = \min\calS$, which is~(i).

    To prove~the second half of~(i), suppose $\liminf_{k \to +\infty} \Delta_k = c > 0$. Then $\Delta_k \geq c/2$ for
    all $k$ sufficiently large, and~(i) forces
    $\sum_k 1/M_k \leq (4/c^2) \sum_k \Delta_k^2/M_k < +\infty$, contradicting
    \ref{as: Linear growth bound model Hessian}.

    Finally, let $\norm{H_k} \leq \mykappa{umh}$ for every $k$. Then $M_k \leq L + \mykappa{umh}$,
    and~(i) gives $\sum_k \Delta_k^2 \leq (L + \mykappa{umh}) \sum_k \Delta_k^2/M_k < +\infty$, so
    that $\Delta_k \to 0$. This is~(ii), and the proof is complete.
\end{proof}

    \begin{remark}
        In Proposition~\ref{prop: CV rate gradient-scaled radius}, we work under the assumption $\eta > 0$.
        In Part~I, we explained that it is possible to consider the case $\eta < \eta_1$ with a strong condition on $\gamma_2$ to decouple the acceptance of the step and the update of the radius, i.e., $\gamma_2 \, (L \overline{\mu}  + 1 ) < 1$.
        However, the Lipschitz constant of the gradient is often unknown and the $\gamma_2$ condition is strong, imposing a conservative decrease of the radius.

        In practice, it is easier to consider $\eta >0$. 
        This implies that when
        \( \rho_k < \eta \), we have \( \mu_{k+1} \in [\gamma_1 \mu_k, \gamma_2 \mu_k) \) and $x_{k+1} = x_k $. Then, 
        $
            \Delta_{k+1} = \mu_{k+1} \| g_{k+1} \| \leq \gamma_2 \mu_k \| g_{k} \| \leq \gamma_2 \Delta_k
        $.
        \label{remark: eta_dec and eta_inc relative gradient}
    \end{remark}

    From the previous proposition, when $\mu_k$ is bounded above and the model Hessian is uniformly bounded, $\sum_k \norm{g_k}^2$ converges. This gives a concise route to first-order convergence.

    \begin{corollary}
        \label{cor:gradient-scaled-liminf}
        Under the assumptions of
        Proposition~\ref{prop: CV rate gradient-scaled radius}, suppose there exists
        $\overline{\mu} > 0$ such that $\mu_k \leq \overline{\mu}$ for all $k$ and that $\norm{H_k} \leq \mykappa{umh}$ for all $k$. Then $\sum_k \norm{g_k}^2 < +\infty$ and $\lim_{k \to +\infty} \norm{g_k} = 0$.
    \end{corollary}

    \begin{proof}
        Assume $\norm{H_k} \leq \mykappa{umh}$, so that
        $M_k \leq L + \mykappa{umh}$ and, by Proposition~6.6 of \PartI,
        \begin{equation*}
            \mu_k \;\geq\; \frac{\mykappa{lbd}}{M_k}
            \;\geq\; \frac{\mykappa{lbd}}{L + \mykappa{umh}} \;=:\; \underline{\mu} \;>\; 0
            \qquad \text{for every } k .
        \end{equation*}
        Let $k_0 \in \calS$ with $k_0 + 1 \notin \calS$ and let $j^{+}$ be as in the proof of
        Proposition~\ref{prop: CV rate gradient-scaled radius}. For
        $1 \leq j \leq j^{+}-1$ the contraction branch gives
        $\mu_{k_0+j} \leq \gamma_2^{\,j-1}\mu_{k_0+1} \leq \gamma_2^{\,j-1}\gamma_3\overline{\mu}$,
        so $\mu_{k_0+j} \geq \underline{\mu}$ forces the run to carry at most $J$
        unsuccessful iterations,
        \begin{equation*}
           j^{+} - 1 \;\leq\; J \;\defeq\; 1 +
            \left\lceil \frac{\ln\bigl(\gamma_3\overline{\mu}/\underline{\mu}\bigr)}
                             {\ln(1/\gamma_2)} \right\rceil ,
        \end{equation*}
        a bound independent of $k_0$. On such a run the iterate is unchanged, so
        $\norm{g_{k_0+j}} = \norm{g_{k_0+1}} \leq (1 + \overline{\mu}L)\norm{g_{k_0}}$ by
        \ref{as: Lipschitz continuous gradient} exactly as above, and $\{M_k\}$ is non-decreasing;
        hence
        \begin{equation*}
            \sum_{j=0}^{j^{+}-1} \norm{g_{k_0+j}}^2
            \;\leq\; \Bigl( 1 + J\,(1 + \overline{\mu}L)^2 \Bigr) \norm{g_{k_0}}^2 .
        \end{equation*}
        Summing over $k_0 \in \calS$ and invoking~\eqref{eqn: grad successful summability}
        together with $M_k \leq L + \mykappa{umh}$, which turns the weighted sum over
        $\calS$ into $\sum_{k \in \calS}\norm{g_k}^2 < +\infty$, gives
        $\sum_{k} \norm{g_k}^2 < +\infty$, whence $\norm{g_k} \to 0$.
    \end{proof}

    The gradient-scaled rule delivers $\lim_{k\to\infty}\norm{g_k} = 0$ under a uniform bound on the model Hessian, with no assumption on the iterates. The proportionality $\Delta_k = \mu_k\norm{g_k}$ and the lower bound $\mu_k \geq \mykappa{lbd}/M_k$ of Part~I suffice by themselves. The admissibility machinery of \PartI{} is needed only for the $\liminf$ statement of~(i), under the weaker growth assumption.

    The conclusion nevertheless says less about the iterates than intuition suggests. Neither $\norm{g_k} \to 0$ nor $\Delta_k \to 0$ forces $\{x_k\}$ to converge, or even to possess a single accumulation point. The next example is a one-dimensional exponential. It satisfies every standing assumption, and all five mechanisms handle it successfully, yet the iterates escape to infinity while the gradient vanishes.

    \subsection{An example with no accumulation points} \label{subsec:Example with no accumulation points}

    Consider $f(x) = e^{-x}$ with $x_0 = 0$. By monotonic decrease of function values, we can restrict to $[0, +\infty)$. On this interval, the gradient $g(x) = -e^{-x}$ is Lipschitz continuous with constant $L = 1$, and the Hessian $H(x) = e^{-x}$ is uniformly bounded by $\mykappa{umh} = 1$. We build the model using the true Hessian $H_k = e^{-x_k}$,
    \begin{equation}    
        m_k(s) = \left( 1 - s + \tfrac{1}{2} s^2 \right) e^{-x_k}.
        \label{eqn: model example}
    \end{equation}

    The unconstrained minimum is attained at $s_k^* = 1$, so the constrained Newton step is
    $s_k^N = \min\{1, \Delta_k\}$. The acceptance ratio
    $
        \rho_k(s) = \tfrac{1 - e^{-s}}{s(1 - s/2)}
    $
    satisfies $\rho_k(s) \geq 1$ for every $s \in (0, 1]$, with $\rho_k(s) \to 1$ as $s \to 0^+$.
    Every iteration is therefore successful and $x_{k+1} = x_k + \min\{1, \Delta_k\}$. Since
    $\norm{H_k} = \norm{g_k} = e^{-x_k} > 0$ and $s_k \in (0,1]$, we have
    $1 - s_k/2 \geq 1/2$, so for any $\mykappa{mdc} \in (0,1)$,
    \[
        m_k(0) - m_k(s_k)
        = \left( s_k - \tfrac{1}{2} s_k^2 \right) e^{-x_k}
        = s_k \left( 1 - \tfrac{1}{2} s_k \right) e^{-x_k}
        \geq \tfrac{1}{2} s_k \, e^{-x_k}
        \geq \tfrac{\mykappa{mdc}}{2} \norm{g_k}
             \min\!\left(\tfrac{\norm{g_k}}{\norm{H_k}}, \Delta_k\right),
    \]
    where the last inequality uses $\norm{g_k}/\norm{H_k} = 1$ and $s_k = \min\{1,\Delta_k\}$, so
    that $\min(\norm{g_k}/\norm{H_k}, \Delta_k) = s_k$. The Cauchy decrease condition
    \ref{as: Cauchy decrease} is therefore satisfied.

    \subsubsection*{Convergence of the gradient norm holds for all four mechanisms}

    For all mechanisms that satisfy Conditions~\ref{cond: PartI C1}, \ref{cond: PartI C2} and~\ref{cond: PartI C3} of Part~I (verified in Section~6 of~\cite{Jeremy2026-I}), together with \ref{as: Bounded below}, \ref{as: Lipschitz continuous gradient}, \ref{as: model agreement with function}, \ref{as: Cauchy decrease}, and the uniform bound $\norm{H_k} \leq \mykappa{umh} = 1$ in this example, Theorems~3.2 and~2.4 of Part~I give
    \begin{equation}
        \lim_{k \to \infty} \norm{g_k} = 0
        \label{eqn:example gradient to zero}
    \end{equation}
    for each of the four mechanisms, without any further argument specific to the example.

    From~\eqref{eqn:example gradient to zero} and $\norm{g_k} = e^{-x_k}$, $x_k = -\ln \norm{g_k} \xrightarrow[k \to \infty]{} +\infty$, and $\sum_{j=0}^{k-1} s_j = x_k - x_0 = -\ln \norm{g_k} \xrightarrow[k \to \infty]{} +\infty$. Since $s_k \leq \Delta_k$ for all $k$, we have
    \begin{equation}
        \sum_{k=0}^{+\infty} \Delta_k \;\geq\; \sum_{k=0}^{+\infty} \norm{s_k} \;=\; +\infty .
        \label{eqn:example sum delta diverges}
    \end{equation}
    In other words, no matter how the radius is chosen, the algorithm must take infinitely many steps of total length $+\infty$ to follow the gradient to zero on $f(x) = e^{-x}$, since the function has no critical point at finite distance.
    The four mechanisms differ only in the \emph{rate} at which $\Delta_k$ shrinks (or fails to). The four mechanisms are treated in turn below.

    \subsubsection*{The criticality-anchored rules: R-DFO}

    Since $\norm{H_k} \leq 1$ for every $k$, Proposition~\ref{prop: behaviour DFO-like radius}(ii) applies and yields $\sum_k \Delta_k^2 < +\infty$ and $\Delta_k \to 0$. Combined with~\eqref{eqn:example sum delta diverges},
    \begin{equation*}
        \sum_{k=0}^{+\infty} \Delta_k = +\infty
        \qquad \text{and} \qquad
        \sum_{k=0}^{+\infty} \Delta_k^2 < +\infty ,
    \end{equation*}
    so the radius shrinks fast enough to be square-summable, yet slowly enough for its sum to diverge. This is the criticality anchoring at work: the test $\Delta_k \lessgtr \zeta \norm{g_k}$ ties the radius to the gradient norm up to the constant $\zeta$, and $\Delta_k$ therefore inherits the exponential decay of $\norm{g_k} = e^{-x_k}$.

    \subsubsection*{The criticality-anchored rules: R-grad}

    For the gradient-scaled rule the outcome depends on whether the scaling coefficient remains bounded, and the example separates the two cases sharply.

    \paragraph{Unbounded $\mu_k$.}
    Here the rule detaches from the gradient, and the radius stays bounded away from zero. Every iteration is very successful with $\rho_k \geq 1 > \eta_2$, so only the last two branches of Update~\ref{update: rgrad uncapped rule} are ever active, and the choice between them is decided by the test $\norm{s_k} \lessgtr \tfrac{1}{2} \Delta_k$. Since $s_k = \min\{1, \Delta_k\}$, this test reduces to one on $\Delta_k$ alone:
    \begin{equation}
        \norm{s_k} > \tfrac{1}{2} \Delta_k
        \quad \Longleftrightarrow \quad
        \Delta_k < 2 .
        \label{eqn:example grad test}
    \end{equation}
    Indeed, if $\Delta_k \leq 1$ then $s_k = \Delta_k > \tfrac{1}{2}\Delta_k$; if $1 < \Delta_k < 2$ then $s_k = 1 > \tfrac{1}{2}\Delta_k$; and if $\Delta_k \geq 2$ then $s_k = 1 \leq \tfrac{1}{2}\Delta_k$. Using $\Delta_k = \mu_k e^{-x_k}$ and $x_{k+1} = x_k + s_k$, Update~\ref{update: rgrad uncapped rule} together with~\eqref{eqn:example grad test} gives
    \begin{equation}
        \Delta_{k+1} = \varphi(\Delta_k),
        \qquad
        \varphi(\Delta) \defeq
        \begin{cases}
            \Delta \, e^{\ln \gamma_3 - \min\{1, \Delta\}} ,
              & \text{if } \Delta < 2 \quad (\mu_{k+1} = \gamma_3 \mu_k) , \\[4pt]
            \Delta \, e^{-1} ,
              & \text{if } \Delta \geq 2 \quad (\mu_{k+1} = \mu_k) .
          \end{cases}
        \label{eqn:example grad map}
    \end{equation}
    The radius therefore follows an autonomous one-dimensional dynamical system. 
    % and its asymptotic behaviour is entirely governed by the position of $\ln \gamma_3$ relative to~$1$.
    The region $\{\Delta \geq 2\}$ is transient. There the second branch
    of~\eqref{eqn:example grad map} contracts by the fixed factor $e^{-1}$ irrespective of
    $\gamma_3$, since the multiplier stalls while the iterate still advances by the full unit
    step, so every trajectory enters $[2e^{-1}, 2)$ in finitely many iterations.

    On $\{\Delta < 2\}$ the growth factor is $\varphi(\Delta)/\Delta = e^{\ln\gamma_3 - \Delta}$ for
    $\Delta \leq 1$ and $\gamma_3 e^{-1}$ for $1 < \Delta < 2$, so the asymptotics are decided by
    the position of $\ln\gamma_3$ relative to $1$.  
    For $\gamma_3 < e$, the region $\{1 < \Delta < 2\}$ is contractive, leading to the region $\{0 < \Delta < 1\}$, where the map is increasing with a unique fixed point $\ln\gamma_3$, which attracts the whole interval. The expansion branch then fires at every iteration and $\mu_k$ diverges geometrically.
    For $\gamma_3 = e$ every point of $[1,2)$ is fixed, and on $\{0 < \Delta < 1\}$, similarly, the map is increasing with a unique fixed point $\ln\gamma_3 = 1$.
    Finally, for $\gamma_3 > e$ the growth factor exceeds $1$ throughout, the radius crosses the threshold in finitely many iterations and oscillates persistently in the invariant band $\calI \defeq [2e^{-1}, 2\gamma_3 e^{-1}]$. Here $\mu_k$ still diverges, but no longer geometrically, since it stalls at every iteration with $\Delta_k \geq 2$.

    In all three regimes $\liminf_k \Delta_k > 0$, the radius being convergent to $\ln\gamma_3$, eventually constant
    in $[1,2)$, or oscillating in $\calI$. The unbounded growth of $\mu_k$ severs the link between
    the radius and the criticality measure, and the conclusions of
    Proposition~\ref{prop: CV rate gradient-scaled radius} become unavailable. The hypothesis
    $\mu_k \leq \overline{\mu}$ is not a technical convenience but the precise dividing line
    between the two regimes.

    \paragraph{Bounded $\mu_k$ with the decay rate of the radius: $\Delta_k = \Theta(1/k)$}
    If $\mu_k \leq \overline{\mu}$ for every $k$, then Proposition~\ref{prop: CV rate gradient-scaled radius}(ii) applies, since $\norm{H_k} \leq 1$, and the conclusion matches that of R-DFO: $\Delta_k \leq \overline{\mu} \norm{g_k} \to 0$ with $\sum_k \Delta_k^2 < +\infty$, while $\sum_k \Delta_k = +\infty$ by~\eqref{eqn:example sum delta diverges}. This gives a concrete example of the general statement that the criticality-anchored rules drive $\Delta_k \to 0$ while $\sum_k \Delta_k = +\infty$. The two statements are not contradictory, but they do not specify the rate at which $\Delta_k$ decays. The next proposition shows that the decay is harmonic, not exponential: the iterates reach only $x_k = \ln k + \mathcal{O}(\ln \ln k)$, and the exponential decay in $x$ becomes a harmonic decay in $k$.
    \begin{proposition}
        \label{prop: example harmonic decay}
        Consider the example above run with R-grad, and suppose that $\mu_k = \overline{\mu}$ for every $k \geq k_0$ and that $\Delta_{k_0} \leq 1$. Then the radius obeys the autonomous recursion
        \begin{equation}
            \Delta_{k+1} = \Delta_k \, e^{-\Delta_k} , \qquad k \geq k_0 ,
            \label{eqn: example exact recursion}
        \end{equation}
        the sequence $\{\Delta_k\}_{k \geq k_0}$ decreases strictly to zero, and $\lim_{k \to +\infty} k \, \Delta_k = 1$, i.e., $\Delta_k = 1/k + o(1/k)$.
    \end{proposition}

    \begin{proof}
        Such an index $k_0$ exists. The acceptance ratio computed above gives $\rho_k \geq 1 \geq \eta_2$
        at every iteration, so no branch of Update~\ref{update: rgrad generic} contracts the
        multiplier and $\mu_k \geq \mu_0$ throughout. With $\mu_k \leq \overline{\mu}$
        and~\eqref{eqn:example gradient to zero}, the radius $\Delta_k = \mu_k \norm{g_k}$ tends to
        zero, so there is $k_1$ with $\Delta_k \leq 1$ for every $k \geq k_1$. Beyond $k_1$ the
        step is $\norm{s_k} = \Delta_k > \tfrac{1}{2}\Delta_k$, the expansion branch fires at every
        iteration and $\mu_{k_1+j} = \min\{\gthree^{\,j}\mu_{k_1}, \overline{\mu}\}$. The cap is
        therefore attained after at most
        $\bigl\lceil \ln(\overline{\mu}/\mu_0) / \ln\gthree \bigr\rceil$ iterations, and any later
        index serves as $k_0$.

        We show by induction that $\Delta_k \leq 1$ for every $k \geq k_0$, the case $k = k_0$ being the hypothesis. Suppose $\Delta_k \leq 1$. Then $s_k = \min\{1, \Delta_k\} = \Delta_k$, so that $x_{k+1} = x_k + \Delta_k$. Using $\Delta_k = \mu_k \norm{g_k} = \overline{\mu} \, e^{-x_k}$ twice,
        \begin{equation*}
            \Delta_{k+1}
            = \overline{\mu} \, e^{-x_{k+1}}
            = \overline{\mu} \, e^{-x_k} e^{-\Delta_k}
            = \Delta_k \, e^{-\Delta_k} ,
        \end{equation*}
        which is~\eqref{eqn: example exact recursion}. As $e^{-u} < 1$ for $u > 0$, the recursion gives $0 < \Delta_{k+1} < \Delta_k \,\leq 1$, which closes the induction, so $\{\Delta_k\}_{k \geq k_0}$ is strictly decreasing and bounded below by $0$; it therefore converges to some $L \geq 0$, and passing to the limit in~\eqref{eqn: example exact recursion} yields $L = L e^{-L}$, that is $L \left( 1 - e^{-L} \right) = 0$, whence $L = 0$.
        Seek $\alpha \in \mathbb{R}$ such that $\Delta_{k+1}^{\alpha} - \Delta_k^{\alpha}$ has a nonzero finite limit. By~\eqref{eqn: example exact recursion},
        \begin{equation*}
            \Delta_{k+1}^{\alpha} - \Delta_k^{\alpha}
            = \Delta_k^{\alpha} \left( e^{-\alpha \Delta_k} - 1 \right)
            = \Delta_k^{\alpha} \left( - \alpha \Delta_k + \mathcal{O}\!\left( \Delta_k^2 \right) \right)
            = - \alpha \, \Delta_k^{\alpha + 1} \left( 1 + \mathcal{O}\!\left( \Delta_k \right) \right) .
        \end{equation*}
        Since $\Delta_k \to 0$, selecting $\alpha = -1$ gives $\Delta_{k+1}^{-1} - \Delta_k^{-1} = 1 + \mathcal{O}(\Delta_k) \to 1$, which is the desired limit.
        By the Ces\`aro mean theorem, applied from $k_0$ on because
        \eqref{eqn: example exact recursion} holds only there,
        \begin{equation*}
            \frac{1}{n - k_0} \sum_{k=k_0}^{n-1}
                \left( \frac{1}{\Delta_{k+1}} - \frac{1}{\Delta_k} \right)
            = \frac{1}{n - k_0}
              \left( \frac{1}{\Delta_n} - \frac{1}{\Delta_{k_0}} \right)
            \xrightarrow[n \to +\infty]{} 1 ,
        \end{equation*}
        and since $k_0$ and $\Delta_{k_0}$ are fixed this gives $n\Delta_n \to 1$.
        Therefore, $\Delta_n \sim 1/n$ and the proof is complete.%
    \end{proof}
    
    \subsubsection*{The criticality-blind rules: R-delta and R-step}

    For $\Rdelta$, every iteration is very successful, so $\Delta_{k+1} \geq \Delta_k$ and therefore $\Delta_k \geq \Delta_0 > 0$ for every $k$. For $\Rstep$, every iteration is very successful, so $\Delta_{k+1} \geq \norm{s_k}$ and therefore $\Delta_k \geq \norm{s_0} = \min\{1, \Delta_0\} > 0$. In both cases the radius is bounded below by a positive constant, in agreement with the general behaviour established in Section~\ref{subsec:asymptotic fixed-factor}.

    \subsection{Towards the local regime}

    The example exhibits, on a single function satisfying all the standing assumptions, the
    dichotomy announced at the beginning of this section. The criticality-anchored rules R-DFO and
    R-grad, the latter provided $\mu_k$ stays bounded, force $\Delta_k \to 0$ at the rate dictated
    by the gradient, even though the iterates possess no accumulation point. 
    Their vanishing radius asks nothing of the iterates and is what derivative-free and
    stochastic variants need to keep their models accurate. However, it could forbid the full
    quasi-Newton step and leave the trust-region constraint active. 
    In contrast, R-delta and R-step read only the previous radius and the realised step, so their radius carries no information about the distance to stationarity. It stays bounded below, which eventually makes the constraint inactive and the superlinear rate available, and that conclusion is unreachable without hypotheses on where the iterates go. 
    The sharper rate $\sum_k \Delta_k < +\infty$ promised by the step-driven analysis
    (Proposition~\ref{prop:step-driven-summable}) is therefore local rather than global, requiring
    the iterates to converge to a point $x^*$ with positive definite Hessian, $\nabla^2 f(x^*) \succ 0$, and it is in that regime
    that the summability of the step series (Theorem~\ref{thm: series of steps converges})
    propagates to the radius. Section~\ref{sec:local dynamics} works under those hypotheses.
    Figure~\ref{fig:asymptotic-regimes} sets the three regimes side by side, and the
    hypotheses each of them needs.
    
    \begin{figure}[tbp]
        \centering
        \begin{tikzpicture}[
        font=\small,
        box/.style={draw, rounded corners=2pt, align=center, inner sep=5pt},
        mech/.style={box, fill=white, minimum width=30mm, minimum height=9mm},
        concl/.style={box, fill=black!2, align=left, inner sep=7pt, text width=62mm},
        lab/.style={font=\scriptsize\itshape, text=black!60},
        ar/.style={-{Latex[length=2.2mm]}, thick},
        ]

        % ---- top group: criticality-anchored
        \node[mech] (dfo)  at (0,0)     {R-DFO\\[-1pt]{\scriptsize test $\Delta_k \lessgtr \zeta\norm{g_k}$}};
        \node[mech] (grad) at (0,-1.45) {R-grad\\[-1pt]{\scriptsize $\Delta_k = \mu_k\norm{g_k}$}\\[-1pt]{\scriptsize $\mu_k \le \overline{\mu}$}};

        \begin{scope}[on background layer]
        \node[draw, dashed, rounded corners=3pt, fit=(dfo)(grad), inner sep=8pt, fill=black!4] (gA) {};
        \end{scope}
        \node[lab, above=2pt of gA] {criticality-anchored};

        \node[concl, right=30mm of gA] (cA)
        {\centering
            $\displaystyle\sum_{k} \frac{\Delta_k^2}{M_k} < +\infty$ \\[4pt]
        {\scriptsize under \ref{as: Linear growth bound model Hessian}: $\displaystyle\liminf_{k} \Delta_k = 0$}\\[3pt]
        {\scriptsize under \ref{as: model hessian}: $\displaystyle \norm{H_k} \le \mykappa{umh} < +\infty$: $\Delta_k \to 0$}
        };

        \draw[ar] (gA.east) -- node[lab, above, align=center, pos=0.5]
        {no hypothesis\\[-2pt]on $\{x_k\}$} (cA.west);

        % ---- bottom group: criticality-blind
        \node[mech] (delta) at (0,-4.6)  {R-delta\\[-1pt]{\scriptsize $\Delta_{k+1}$ from $\Delta_k$}};
        \node[mech] (step)  at (0,-6.05) {R-step\\[-1pt]{\scriptsize $\Delta_{k+1}$ from $\norm{s_k}$}};

        \begin{scope}[on background layer]
        \node[draw, dashed, rounded corners=3pt, fit=(delta)(step), inner sep=8pt, fill=black!4] (gB) {};
        \end{scope}
        \node[lab, below=2pt of gB] {criticality-blind};

        \node[concl, right=30mm of gB] (cB)
        {no conclusion on $\{\Delta_k\}$ at this generality\\[4pt]
        {\scriptsize Example (\S\,\ref{subsec:Example with no accumulation points}): $\Delta_k \ge \Delta_0 > 0$ while $\norm{g_k}\to 0$}\\[3pt]
        {\scriptsize local rates require $x_k \to x^\ast$, $\nabla^2 f(x^\ast)\succ 0$}};

        \draw[ar] (gB.east) -- node[lab, above, align=center, pos=0.5]
        {requires\\[-2pt]Section~\ref{sec:local dynamics}} (cB.west);

        % ---- common bar
        \node[box, fill=black!7, text width=118mm, align=center,
            below=8mm of $(gB.south)!0.5!(cB.south)$, anchor=north] (common)
        {\small common to all mechanisms:\;
        $\displaystyle\lim_{k\to\infty}\norm{g_k} = 0$\;\;
        yet\;\; $\displaystyle\sum_{k}\Delta_k = +\infty$ is possible};

        \end{tikzpicture}
        \caption{Asymptotic regime of the trust-region radius under the standing assumptions of
        Section~\ref{sec: Background}. Anchoring the radius to the criticality measure (top) determines
        the regime of $\{\Delta_k\}$ with no hypothesis on the iterates, whereas rules that ignore
        criticality (bottom) admit no such conclusion until the local hypotheses of
        Section~\ref{sec:local dynamics} are imposed.}
        \label{fig:asymptotic-regimes}
        \end{figure}
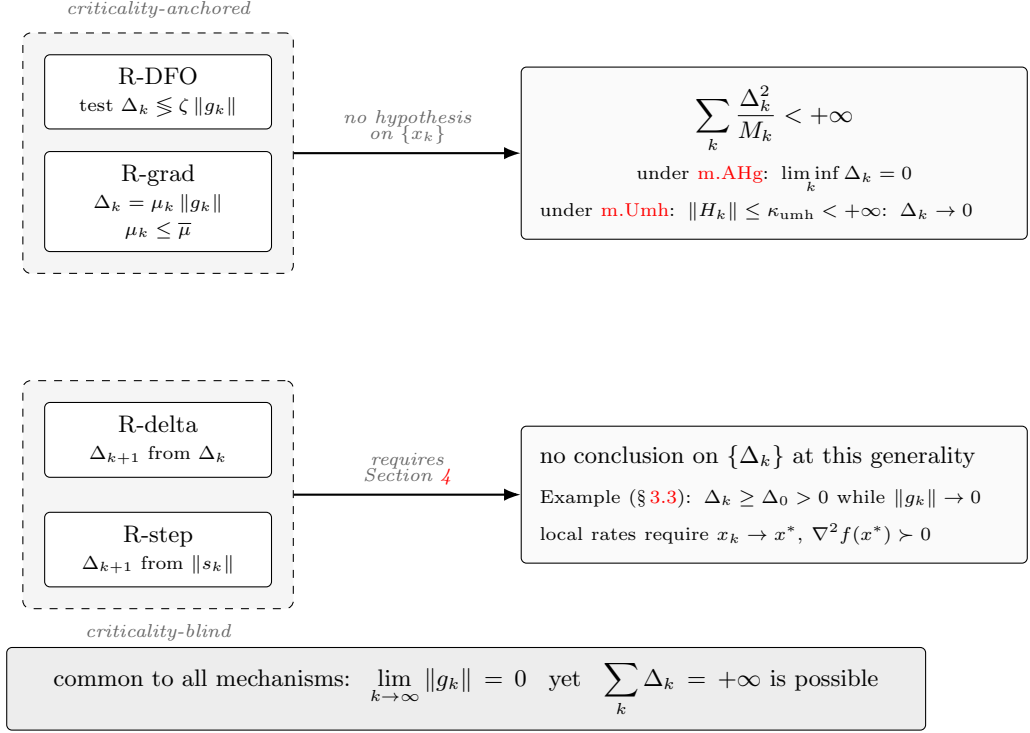

%=======================================================================
%   SECTION: LOCAL DYNAMICS NEAR A FIRST-ORDER CRITICAL POINT
%=======================================================================
\section{Asymptotic behaviour: near a first-order critical point} \label{sec:local dynamics}

    This section studies the local dynamics of trust-region methods near a first-order critical point $x^*$ at which the Hessian is positive definite. Our starting point is a condition on the step, rather than on the model: we show that if a subsequence of iterates converges to $x^*$ and the step length is controlled by the gradient norm (Condition~\ref{cond: s_k bounded by g_k}), then the full sequence converges to $x^*$. For the criticality-anchored rules R-grad and R-DFO this conclusion requires no assumption on the model Hessian, in contrast with the classical rules, for which it holds only under $H_k \succeq \mykappa{lbm} I$. We then derive, in this local regime, convergence rates for the radius generated by the step-driven and fixed-factor rules, and we close with the difficulties raised by saddle points.

    Condition~\ref{cond: s_k bounded by g_k} recurs throughout trust-region analysis. It prevents the step from growing large relative to the local first-order information, and it underwrites both the full-sequence convergence of Theorem~\ref{thm: step bounded by gradient norm implies convergence} and the eventual inactivity of the trust-region constraint established in Section~\ref{sec:convergence rate}.
    \begin{cond}{C.Sg}[step bounded by the gradient]
\label{cond: s_k bounded by g_k}
        There exist constants $\delta^{(s)} > 0$ and $\overline{\kappa} > 0$, and an index $k_\kappa \geq 0$, such that for all $k \geq k_\kappa$,
        \[
            \norm{x_k - x^*} \leq \delta^{(s)} \quad \Longrightarrow \quad \norm{s_k} \leq \overline{\kappa}\, \norm{g_k}.
        \]
    \end{cond}

    The requirement is \emph{eventual}. The implication is imposed only from the index
    $k_\kappa$ on, not from the start of the run. This is what
    Lemma~\ref{lem: step bounded by gradient norm pointwise} delivers, since
    \ref{as: Hessian-step residual condition} controls the model Hessian only asymptotically, and
    it is the reading already used in Section~\ref{sec:convergence rate}, where every statement
    resting on the condition is prefaced by $k \geq k_\kappa$. The gradient-scaled rule
    attains it with $k_\kappa = 0$ and $\overline{\kappa} = \overline{\mu}$ as soon as its
    multiplier is capped; $\Rdfo$ does not satisfy it, and reaches the same conclusions by
    another route. See Corollary~\ref{cor: convergence criticality anchored saddle}.

    The section closes with saddle points and with the model-Hessian assumptions that tie $H_k$ to Condition~\ref{cond: s_k bounded by g_k}.

    \subsection{Local strong convexity and convergence from a subsequence} \label{sec: local strong convexity}

         If a subsequence converges to $x^*$ and Condition~\ref{cond: s_k bounded by g_k} holds, then the whole sequence converges.
        
        \begin{assum}{f.Pd}[positive definite Hessian at $x^*$]
\label{as: positive definite Hessian}
            The function $f$ is twice continuously differentiable in a neighbourhood $B(x^*, \Delta_*)$ of a first-order critical point $x^*$, and the Hessian $\nabla^2 f(x^*)$ is positive definite. 
            
        \end{assum}

        We derive inequalities showing that $f$ is locally strongly convex around $x^*$. These relate the Hessian, gradient and objective in a neighbourhood of $x^*$.

     \begin{proposition}\label{prop: inequalities for positive definite Hessian}
            Let \ref{as: positive definite Hessian} hold and write
            $\lambda^*_{\min} \leq \lambda^*_{\max}$ for the extreme eigenvalues of
            $\nabla^2 f(x^*)$. Then there are constants $0 < m \leq M$ and a neighbourhood
            $B(x^*, \Delta_*)$ such that $m I \preceq \nabla^2 f(x) \preceq M I$ for all
            $x \in B(x^*, \Delta_*)$, and
            \begin{enumerate}[(i)]
                \item $m \norm{s}^2 \leq s^\top (g(x+s) - g(x)) \leq M \norm{s}^2$, for all $s$ with $x + s \in B(x^*, \Delta_*)$;
                \item $m \norm{s} \leq \norm{g(x+s) - g(x)} \leq M \norm{s}$, for all $s$ with $x + s \in B(x^*, \Delta_*)$;
                \item $m \norm{x - x^*} \leq \norm{g(x)} \leq M \norm{x - x^*}$;
                \item $\frac{1}{2} m \norm{x - x^*}^2 \leq \abs{f(x) - f(x^*)} \leq \frac{1}{2} M \norm{x - x^*}^2$.
            \end{enumerate}
        \end{proposition}
        \begin{proof}
            See~\citet[Thm.~2]{Powe75}.
        \end{proof}

        \begin{remark}[$m$ and $M$ against the eigenvalues at $x^*$]
        \label{rem: m M versus eigenvalues}
            The constants $m$ and $M$ bound the Hessian uniformly on a neighbourhood and are
            therefore not the eigenvalues at $x^*$ themselves, satisfying only
            $0 < m \leq \lambda^*_{\min} \leq \lambda^*_{\max} \leq M$. Continuity of
            $\nabla^2 f$ makes the gap as small as desired at the cost of a smaller ball: for
            every $\epsilon > 0$ there is $\Delta_*$ on which $m \geq \lambda^*_{\min} - \epsilon$
            and $M \leq \lambda^*_{\max} + \epsilon$. A convenient concrete choice, used wherever
            a numerical value is wanted below, is
            \begin{equation}
                m = \tfrac12\lambda^*_{\min} ,
                \qquad
                M = 2\lambda^*_{\max} .
                \label{eqn: m M concrete}
            \end{equation}
        \end{remark}
        For $\alpha > 0$, write $L_\alpha = \{ x \in \Rp : f(x) - f(x^*) \leq \alpha \}$ for the corresponding level set of $f$.

        \begin{proposition}\label{prop: level set inclusion}
            For any $0 < \overline{\Delta} < \Delta_*$, there exist $\alpha > 0$ and $0 < \underline{\Delta} \leq \overline{\Delta}$ such that
            \[
                B(x^*, \underline{\Delta}) \subseteq L_\alpha \cap B(x^*, \Delta_*) \subseteq B(x^*, \overline{\Delta}).
            \]
        \end{proposition}
        \begin{proof}
            From Proposition~\ref{prop: inequalities for positive definite Hessian}(iv), for $x \in B(x^*, \Delta_*)$, $\frac{1}{2} m \norm{x - x^*}^2 \leq \abs{f(x) - f(x^*)} \leq \frac{1}{2} M \norm{x - x^*}^2$. Take $0 < \alpha \leq m \overline{\Delta}^2 / 2$. For $\Delta_* \geq \norm{x - x^*} > \overline{\Delta}$, we have $f(x) - f(x^*) > \frac{1}{2} m \overline{\Delta}^2 \geq \alpha$, so by contraposition $L_\alpha \cap B(x^*, \Delta_*) \subseteq B(x^*, \overline{\Delta})$. Take $0 < \underline{\Delta} \leq \sqrt{2\alpha/M}$. Then $\norm{x - x^*} \leq \underline{\Delta}$ implies $f(x) - f(x^*) \leq \frac{1}{2} M \norm{x - x^*}^2 \leq \alpha$, so $B(x^*, \underline{\Delta}) \subseteq L_\alpha$.
        \end{proof}

        \begin{remark}
            We use the triplet
            $\calC_{\overline{\Delta}} = \left(\alpha = \frac{m \overline{\Delta}^2}{2},\, \underline{\Delta} = \sqrt{m/M}\, \overline{\Delta},\, \overline{\Delta}\right)$, specifying only $0 < \overline{\Delta} < \Delta_*$.
        \end{remark}

        \begin{theorem}\label{thm: step bounded by gradient norm implies convergence}
            Let $\{x_k\}$ be generated by Algorithm~\ref{alg:generic BTR} under
            \ref{as: positive definite Hessian}, and suppose that a subsequence $\{x_{k_j}\}$
            converges to a first-order critical point $x^*$ at which $\nabla^2 f(x^*)$ is positive
            definite. If Condition~\ref{cond: s_k bounded by g_k} holds, then the whole sequence
            $\{x_k\}$ converges to $x^*$.
        \end{theorem}

        \begin{proof}
            The argument is that of~\citet{Stei83}, in the form of
            \citet[Thm.~6.5.2]{ConnGoulToin00}; we give it because the constants are used
            again in Section~\ref{sec:convergence rate}. Let $m$, $M$ and $\Delta_*$ be given by
            Proposition~\ref{prop: inequalities for positive definite Hessian}, and
            $\delta^{(s)}$, $\overline{\kappa}$ and $k_\kappa$ by
            Condition~\ref{cond: s_k bounded by g_k}. Fix
            $\delta \in (0,\Delta_*]$ and set
            \begin{equation}
                \theta = \left( \frac{1}{m} + \overline{\kappa} \right) M ,
                \qquad
                \overline{\Delta} = \frac{\min\{\delta, \delta^{(s)}\}}{\theta} ,
                \label{eqn: local radius choice}
            \end{equation}
            with the triplet $\calC_{\overline{\Delta}}$ of
            Proposition~\ref{prop: level set inclusion}. Since $\theta \geq 1$,
            \begin{equation}
                \overline{\Delta} \leq \min\{\delta, \delta^{(s)}\} \leq \delta^{(s)} ,
                \label{eqn: Dbar below deltas}
            \end{equation}
            so Condition~\ref{cond: s_k bounded by g_k} applies on $B(x^*,\overline{\Delta})$.
            The subsequence puts some $x_{k_{j_0}}$ in
            $B(x^*,\underline{\Delta}) \subseteq L_\alpha \cap B(x^*,\Delta_*)$, and
            since $\{x_{k_j}\}$ has infinitely many terms in that ball, $j_0$ may be chosen so that
            in addition $k_{j_0} \geq k_\kappa$, and we show by
            induction that no later iterate leaves that set, the trapping argument of
            Figure~\ref{fig:trapping}. If $x_k$ lies in it for some $k \geq k_{j_0}$ then
            $\norm{x_k - x^*} \leq \overline{\Delta}$, and $k \geq k_{j_0} \geq k_\kappa$, so
            Condition~\ref{cond: s_k bounded by g_k} gives
            $\norm{s_k} \leq \overline{\kappa}\norm{g_k}$ and
            Proposition~\ref{prop: inequalities for positive definite Hessian}(iii), used first as
            $\norm{x_k - x^*} \leq \norm{g_k}/m$ and then as
            $\norm{g_k} \leq M\norm{x_k - x^*}$, yields
            \begin{equation*}
                \norm{x_{k+1} - x^*}
                \;\leq\; \norm{x_k - x^*} + \norm{s_k}
                \;\leq\; \left( \frac{1}{m} + \overline{\kappa} \right) \norm{g_k}
                \;\leq\; \theta\,\overline{\Delta}
                \;=\; \min\{\delta, \delta^{(s)}\} \;\leq\; \Delta_* .
            \end{equation*}
            With $f(x_{k+1}) \leq f(x_k)$ this puts $x_{k+1}$ back in
            $L_\alpha \cap B(x^*,\Delta_*)$, closing the induction. As $\delta$ was arbitrary,
            $\{x_k\}$ converges to $x^*$.
        \end{proof}

        \begin{corollary}\label{cor:single-iterate-in-ball}
            Let $\{x_k\}$ be generated by Algorithm~\ref{alg:generic BTR} under the hypotheses of
            Theorem~\ref{thm: step bounded by gradient norm implies convergence}, with
            $\overline{\Delta}$ and $\calC_{\overline{\Delta}}$ as defined there. If some iterate
            of index $k \geq k_\kappa$ satisfies
            $x_k \in L_\alpha \cap B(x^*, \Delta_*)$, then $\{x_k\}$ converges to $x^*$.
        \end{corollary}
        \begin{proof}
            The induction in the proof of
            Theorem~\ref{thm: step bounded by gradient norm implies convergence} uses only that some
            iterate at which Condition~\ref{cond: s_k bounded by g_k} is in force belongs
            to $L_\alpha \cap B(x^*, \Delta_*)$, and not that it arises from a
            convergent subsequence.
        \end{proof}

        \begin{figure}[tbp]
            \centering
            \begin{tikzpicture}[font=\small]

            \def\Rstar{3.15}
            \def\Rbar{2.3}
            \def\Rund{1.15}

            \draw[dashed, thick, black!50] (0,0) circle (\Rstar);
            \draw[dashed, thick, black!40] (0,0) circle (\Rbar);
            \begin{scope}
            \clip (0,0) circle (\Rstar);
            \filldraw[fill=black!8, draw=black!70, thick]
            plot[smooth cycle, tension=0.85] coordinates {
                (0:2.1) (45:1.75) (90:2.0) (135:1.65) (180:2.02)
                (225:1.7) (270:1.95) (315:1.6)};
            \end{scope}
            \filldraw[fill=black!20, draw=black!65, thick] (0,0) circle (\Rund);
            \fill (0,0) circle (1.6pt);
            \node[below right=0pt and 1pt] at (0,0) {$x^*$};

            % iterates: enter the small ball, then stay trapped and converge
            \coordinate (a) at (152:2.95);
            \coordinate (b) at (128:2.35);
            \coordinate (c) at (100:1.02);   % x_{k_{j_0}} enters B(x*,Dunder)
            \coordinate (d) at (40:1.55);
            \coordinate (e) at (-15:1.12);
            \coordinate (f) at (-70:0.7);
            \coordinate (g) at (200:0.4);
            \coordinate (h) at (30:0.2);

            \draw[-{Latex[length=1.5mm]}, black!45, thick] (a) -- (b);
            \draw[-{Latex[length=1.5mm]}, black!45, thick] (b) -- (c);
            \draw[-{Latex[length=1.7mm]}, thick] (c) -- (d);
            \draw[-{Latex[length=1.7mm]}, thick] (d) -- (e);
            \draw[-{Latex[length=1.7mm]}, thick] (e) -- (f);
            \draw[-{Latex[length=1.7mm]}, thick] (f) -- (g);
            \draw[-{Latex[length=1.7mm]}, thick] (g) -- (h);

            \foreach \p in {a,b} \fill[black!45] (\p) circle (1.5pt);
            \foreach \p in {c,d,e,f,g,h} \fill (\p) circle (1.7pt);

            \node[font=\scriptsize, black!45, above left=0pt and 0pt] at (a) {$x_k$, $k$ small};
            \node[font=\scriptsize, left=1pt] at (c) {$x_{k_{j_0}}$};
            \node[font=\scriptsize, right=1pt] at (d) {$x_{k_{j_0}+1}$};

            \node[font=\scriptsize, black!50] at (255:2.95) {$B(x^*,\Delta_*)$};
            \node[font=\scriptsize, fill=black!8, inner sep=1pt] at (108:2.66) {$L_\alpha$};

            \end{tikzpicture}
            \caption{The trapping argument of
            Theorem~\ref{thm: step bounded by gradient norm implies convergence}. Once a single iterate
            of index $k \geq k_\kappa$ enters $L_\alpha \cap B(x^*,\Delta_*)$, the step bound of
            Condition~\ref{cond: s_k bounded by g_k} keeps the next iterate inside $B(x^*,\Delta_*)$, while the
            monotonicity of $\{f(x_k)\}$ keeps it inside $L_\alpha$. No iterate can escape, and since
            $\overline{\Delta} \le \delta$ with $\delta > 0$ arbitrary, the whole sequence converges to $x^*$.}
            \label{fig:trapping}
        \end{figure}
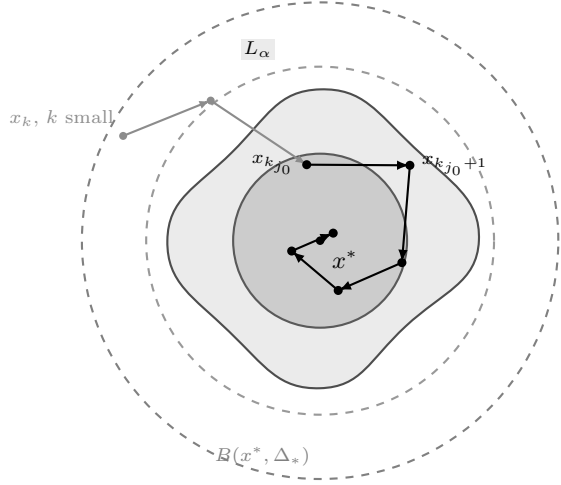

        We also have summability of the steps in this regime, which is a classical result of Powell~\cite{Powe75}.

        \begin{theorem}\label{thm: series of steps converges}
            Consider a sequence $\{x_k\}$ generated by Algorithm~\ref{alg:generic BTR} under \ref{as: Cauchy decrease}, with $\{f(x_k)\}$ non-increasing, and converging to a first-order critical point $x^*$ with $\nabla^2 f(x^*)$ positive definite. Suppose the Hessian approximation satisfies \ref{as: Hessian-step residual condition}. Then
            \[
                \sum_{k \geq 0} \norm{s_k} < +\infty.
            \]
        \end{theorem}
        \begin{proof}
            See~\citet[Thm.~2]{Powe75}.
        \end{proof}

        \begin{remark}\label{rem: assumptions on the model Hessian}
            Therefore, when considering sequences $\{x_k\}$ generated by the trust-region method that converge to second-order critical points with positive definite Hessian, we assume that $\sum_{k \geq 0} \norm{s_k} < +\infty$ and that the Hessian approximation is uniformly bounded: $\norm{H_k} \leq \mykappa{umh} < +\infty$ for all $k$.
        \end{remark}

        \subsection{Assumption on the model Hessian}\label{sec: assumptions on the model hessian}

        \begin{lemma}\label{lem: step bounded by gradient norm pointwise}
            Let \ref{as: positive definite Hessian} hold, and let $m$ and
            $\Delta_*$ be given by
            Proposition~\ref{prop: inequalities for positive definite Hessian}. If
            \ref{as: Hessian-step residual condition} holds and
            $\lim_{k\to\infty}\norm{g_k} = 0$, then there is an index $k_\kappa$ such
            that every iterate of
            Algorithm~\ref{alg:generic BTR} with $k \geq k_\kappa$ satisfies,
            \begin{equation}
                x_k \in B(x^*, \Delta_*)
                \qquad \Longrightarrow \qquad
                \norm{s_k} \leq \overline{\kappa} \, \norm{g_k} ,
                \qquad\text{where}\quad \overline{\kappa} = \frac{4}{m} .
                \label{eqn: pointwise step bound}
            \end{equation}
            In particular, Condition~\ref{cond: s_k bounded by g_k} holds with
            $\delta^{(s)} = \Delta_*$, $k_\kappa$ as above, and $\overline{\kappa} = 4/m$.
        \end{lemma}

        \begin{proof}
            Let $k$ be such that $x_k \in B(x^*, \Delta_*)$. If $s_k = 0$
            then~\eqref{eqn: pointwise step bound} is immediate, so assume $s_k \neq 0$.

            By Proposition~\ref{prop: inequalities for positive definite Hessian},
            $\nabla^2 f(x_k) \succeq m I$, whence
            $s_k^\top \nabla^2 f(x_k) s_k \geq m \norm{s_k}^2$. By the
            Cauchy--Schwarz inequality, and for $k$ large enough that \ref{as: Hessian-step residual condition} gives $\norm{(H_k - \nabla^2 f(x_k)) s_k} \leq \frac{m}{2} \norm{s_k}$,
            \begin{equation*}
                s_k^\top \big( H_k - \nabla^2 f(x_k) \big) s_k
                \;\geq\; - \norm{s_k} \, \norm{\big( H_k - \nabla^2 f(x_k) \big) s_k}
                \;\geq\; - \frac{m}{2} \norm{s_k}^2 ,
            \end{equation*}
            and adding the two bounds gives the uniform model curvature
            \begin{equation}
                s_k^\top H_k s_k
                \;=\; s_k^\top \nabla^2 f(x_k) s_k + s_k^\top \big( H_k - \nabla^2 f(x_k) \big) s_k
                \;\geq\; \frac{m}{2} \norm{s_k}^2 .
                \label{eqn: model curvature lower bound}
            \end{equation}
            The step decreases the model, so that
            $m_k(0) - m_k(s_k) = - g_k^\top s_k - \tfrac{1}{2} s_k^\top H_k s_k > 0$, that is
            $- g_k^\top s_k > \tfrac{1}{2} s_k^\top H_k s_k$. Combining this
            with~\eqref{eqn: model curvature lower bound} and using Cauchy--Schwarz once more,
            \begin{equation*}
                \norm{g_k} \norm{s_k}
                \;\geq\; - g_k^\top s_k
                \;>\; \frac{1}{2} \, s_k^\top H_k s_k
                \;\geq\; \frac{m}{4} \norm{s_k}^2 .
            \end{equation*}
            Dividing by $\norm{s_k} > 0$ gives~\eqref{eqn: pointwise step bound}. The final
            assertion is~\eqref{eqn: pointwise step bound} read as
            Condition~\ref{cond: s_k bounded by g_k} with $\delta^{(s)} = \Delta_*$ and
            with $k_\kappa$ the index beyond which the displayed implication holds; it is here that
            the eventual form of the condition is needed, the residual bound of
            \ref{as: Hessian-step residual condition} being available only for large $k$.
        \end{proof}

        \begin{remark}[how small $\overline{\kappa}$ can be made]
        \label{rem: kappa bar arbitrarily close}
            Neither constant in $\overline{\kappa} = 4/m$ is forced.
            \ref{as: Hessian-step residual condition} makes the residual $o(\norm{s_k})$, so any
            $\epsilon \in (0,m)$ replaces $m/2$ in the proof. The model curvature along the step is
            then bounded below by $(m-\epsilon)\norm{s_k}^2$, arbitrarily close to the
            $m\norm{s_k}^2$ of the true Hessian, and the argument returns
            $\overline{\kappa} = 2/(m-\epsilon)$ at the price of a larger $k_\kappa$.
            Remark~\ref{rem: m M versus eigenvalues} in turn allows
            $m = (1-\epsilon')\lambda^*_{\min}$ for any $\epsilon' \in (0,1)$ at the price of a
            smaller $\Delta_*$, so
            \begin{equation*}
                \overline{\kappa} \;=\; \frac{2}{(1-\epsilon')\lambda^*_{\min} - \epsilon}
            \end{equation*}
            comes as close to $2/\lambda^*_{\min}$ as desired. The convenient choices
            $\epsilon = m/2$ and $\epsilon' = \tfrac12$, the latter
            being~\eqref{eqn: m M concrete}, give
            $\overline{\kappa} = 4/m = 8/\lambda^*_{\min}$.
        \end{remark}

        The lemma discharges the hypothesis of
        Theorem~\ref{thm: step bounded by gradient norm implies convergence}: whenever a
        subsequence approaches $x^*$, the step bound is automatically in force, and the trapping
        argument applies verbatim. The local theory is used below in the following form.

        \begin{corollary}\label{cor: convergence under AH3}
            Let $\{x_k\}$ be generated by Algorithm~\ref{alg:generic BTR} under
            \ref{as: positive definite Hessian} and
            \ref{as: Hessian-step residual condition}, and suppose
            $\lim_{k\to\infty}\norm{g_k} = 0$. If a subsequence $\{x_{k_j}\}$ converges
            to a first-order critical point $x^*$ at which $\nabla^2 f(x^*)$ is positive definite,
            then the whole sequence $\{x_k\}$ converges to $x^*$, and
            \begin{equation*}
                \norm{s_k} \leq \frac{4}{m} \norm{g_k}
                \qquad\text{for all $k$ sufficiently large.}
            \end{equation*}
        \end{corollary}

        \begin{proof}
            By Lemma~\ref{lem: step bounded by gradient norm pointwise},
            Condition~\ref{cond: s_k bounded by g_k} holds with $\delta^{(s)} = \Delta_*$,
            some index $k_\kappa$, and
            $\overline{\kappa} = 4/m$. Theorem~\ref{thm: step bounded by gradient norm implies convergence}
            then gives $x_k \to x^*$. Consequently $x_k \in B(x^*, \Delta_*)$ for all $k$ large
            enough, and the displayed bound is~\eqref{eqn: pointwise step bound}.
        \end{proof}

    \subsection{The acceptance ratio converges to one}
    \label{subsec: rho converges to one}
    
    One further element of the local regime is used repeatedly in
    Section~\ref{sec:convergence rate}. The acceptance ratio settles at $1$, so that every
    iteration is eventually very successful, and the proof uses nothing from the radius rule.
    \begin{theorem}[The acceptance ratio converges to one]
    \label{thm: rho converges to 1}
        Let $\{x_k\}$ be generated by Algorithm~\ref{alg:generic BTR} and converge to a
        first-order critical point $x^*$ with $\nabla^2 f(x^*)$ positive definite. Assume
        \ref{as: positive definite Hessian}, \ref{as: Cauchy decrease},
        \ref{as: Hessian-step residual condition}, $\norm{H_k} \leq \mykappa{umh}$, and
        $s_k \neq 0$ for all $k$ sufficiently large. Then
        \begin{equation*}
            \lim_{k \to \infty} \rho_k = 1 .
        \end{equation*}
        In particular, for every $\eta_2 < 1$ there is $k_\rho$ such that $\rho_k \geq \eta_2$ for
        all $k \geq k_\rho$, so that every iteration is very successful and every step accepted.
    \end{theorem}
    \begin{proof}
        Since $x_k \to x^*$ and $\nabla f$ is continuous with $\nabla f(x^*) = 0$, we have
        $\norm{g_k} \to 0$, so \ref{as: Hessian-step residual condition} is in force. By
        Lemma~\ref{lem: step bounded by gradient norm pointwise},
        $\norm{s_k} \leq \kbar\norm{g_k}$ with $\kbar = 4/m$ for all large $k$, and
        by Theorem~\ref{thm: series of steps converges}, $\norm{s_k} \to 0$. Fix $k$ large enough
        that all three hold and that $s_k \neq 0$.
        \emph{Denominator.} Since $\norm{g_k} \geq \norm{s_k}/\kbar$ and
        $\norm{s_k} \leq \Delta_k$, assumption~\ref{as: Cauchy decrease} gives
        \begin{equation}
            m_k(0) - m_k(s_k)
            \;\geq\; \frac{\mykappa{mdc}}{2}\,\norm{g_k}
                \min\left\{\frac{\norm{g_k}}{\mykappa{umh}},\ \Delta_k\right\}
            \;\geq\; c_0 \norm{s_k}^2 ,
            \qquad
            c_0 \defeq \frac{\mykappa{mdc}}{2\,\kbar}
                \min\left\{\frac{1}{\kbar\,\mykappa{umh}},\ 1\right\} > 0 .
            \label{eqn: pred quadratic}
        \end{equation}
        \emph{Numerator.} By the mean value theorem there is $\xi_k$ on the segment
        $[x_k, x_k + s_k]$ with
        $f(x_k + s_k) - f(x_k) - g_k^\top s_k = \tfrac12 s_k^\top \nabla^2 f(\xi_k) s_k$, so that
        \begin{align*}
            \Bigl| \bigl[m_k(0) - m_k(s_k)\bigr] - \bigl[f(x_k) - f(x_k + s_k)\bigr] \Bigr|
            &= \tfrac12 \Bigl| s_k^\top \bigl(\nabla^2 f(\xi_k) - H_k\bigr) s_k \Bigr| \\
            &\leq \tfrac12 \norm{s_k}\,\norm{\bigl(\nabla^2 f(x_k) - H_k\bigr) s_k}
             + \tfrac12 \norm{\nabla^2 f(\xi_k) - \nabla^2 f(x_k)}\,\norm{s_k}^2 .
        \end{align*}
        The first term is $o(\norm{s_k}^2)$ by \ref{as: Hessian-step residual condition}; the
        second is $o(\norm{s_k}^2)$ because $\norm{\xi_k - x_k} \leq \norm{s_k} \to 0$ and
        $\nabla^2 f$ is continuous. Dividing by~\eqref{eqn: pred quadratic},
        $\abs{1 - \rho_k} \to 0$.
    \end{proof}
    \begin{remark}
    \label{rem: rho to 1 is mechanism free}
        The proof uses only the Cauchy decrease condition, the step bound of
        Lemma~\ref{lem: step bounded by gradient norm pointwise} and the summability of the steps;
        it never reads the radius rule. This is why
        Theorem~\ref{thm: rho converges to 1} appears in the mechanism-independent column of
        Section~\ref{subsec: mechanism-free elements}. The hypothesis $s_k \neq 0$ excludes the
        degenerate case of
        finite termination, in which $\rho_k$ is undefined; it is the same proviso as
        in~\citet[Thm.~3.9]{BastMalmMoufToinToma10}.
    \end{remark}

    \subsection{Step-driven update: a stronger rate in the local regime} \label{subsec:asymptotic step-driven}

        The step-driven update of Part~I~\cite{Jeremy2026-I} (see Appendix~\ref{app: R-step}) ties the trust-region radius to the step length.
        In the local regime, when the iterates converge to a critical point with positive definite Hessian and the model Hessian is sufficiently consistent, Powell's Theorem~2~\cite{Powe75} implies $\sum_k \norm{s_k} < +\infty$. The radius then inherits the summability.

    \begin{proposition}\label{prop:step-driven-summable}
        Suppose $\{x_k\}$ is generated by the step-driven update, that $\{x_k\}$ converges to $x^*$ with $\nabla^2 f(x^*)$ positive definite, and that the Hessian approximation satisfies \ref{as: Hessian-step residual condition}. Suppose furthermore that there is a constant $\gamma_3 > 1$ capping the expansion branch, that is $\Delta_{k+1} \leq \gamma_3\norm{s_k}$ whenever $\rho_k \geq \eta_2$. Then
        \[
            \sum_{k=0}^{+\infty} \Delta_k < +\infty,
            \quad \text{and in particular} \quad \lim_{k \to +\infty} \Delta_k = 0.
        \]
    \end{proposition}

    \begin{proof}
        By Theorem~\ref{thm: series of steps converges}, $\sum_k \norm{s_k} < +\infty$.
        By Appendix~\ref{app: R-step} the first two branches give
        $\Delta_{k+1} < \norm{s_k}$, and the capped expansion branch gives
        $\Delta_{k+1} \leq \gamma_3\norm{s_k}$; since $\gamma_3 > 1$, in all three cases
        $\Delta_{k+1} \leq \gamma_3\norm{s_k}$. Hence
        \[
            \sum_{k=1}^{+\infty} \Delta_k \leq \gamma_3 \sum_{k=0}^{+\infty} \norm{s_k} < +\infty.
        \]
    \end{proof}

    \begin{remark}\label{rem:step-driven faster}
        Proposition~\ref{prop:step-driven-summable} exhibits a \emph{stronger} convergence rate than the criticality-anchored rules of Sections~\ref{subsec:asymptotic DFO-like} and~\ref{subsec:asymptotic gradient-scaled}: those give only $\sum_k \Delta_k^2/M_k < +\infty$, whereas the step-driven rule yields $\sum_k \Delta_k < +\infty$ in the local regime. This is the price of tying the radius rigidly to the step: in the local regime, where iterates converge, $\Delta_k$ inherits the natural summability of $\norm{s_k}$. In the global regime, however, the radius can remain bounded below (see Section~\ref{subsec:asymptotic fixed-factor} below); the strong rate is a \emph{local} phenomenon.
    \end{remark}

    \begin{example}[Strong convexity does not settle the radius]
    \label{ex: strongly convex}
        Let $f(x) = \tfrac12 x^2$, $x_0 = 1$, $g_k = x_k$, and $H_k = c > 1$ constant, a conservative
        overestimate of the curvature; let $\Delta_0 \geq 1/c$. The unconstrained model minimiser is
        $s^{\mathrm{N}}_k = -x_k/c$, and since $1 - 1/c < 1 < \gthree$ an induction gives
        $\abs{s^{\mathrm{N}}_k} \leq \Delta_k$ under either rule below: the constraint is never
        active, $s_k = -x_k/c$, $x_{k+1} = \Bigl(1 - \tfrac{1}{c}\Bigr) x_k$ and $\rho_k = 2 - \tfrac{1}{c} > 1$ for all $k$. The iterates contract to $x^* = 0$.
        The rule $\Rstep$ sets $\Delta_{k+1} = \gthree\abs{s_k} = \gthree \abs{x_{k+1} - x_k} = \gthree \, 1/c \, x_k = \gthree \, 1/c \, (1 - 1/c)^k x_0$, so $\Delta_k \to 0$ and
        $\sum_k \Delta_k < +\infty$. 
        The rule $\Rdelta$ scales by a factor at least $1$ regardless of the step, so $\Delta_{k+1} \geq \Delta_k$ and
        $\sum_k \Delta_k = +\infty$; on the expansion branch $\Delta_k = \gthree^k\Delta_0 \to
        +\infty$ while $x_k \to 0$, and the trust region becomes vacuous.
    \end{example}

    \subsection{Fixed-factor and incremental updates: bounded below} \label{subsec:asymptotic fixed-factor}

        For the incremental fixed-factor rule of Part~I~\cite{Jeremy2026-I} (see Appendix~\ref{app: R-delta}), the radius is contracted only on unsuccessful iterations and is non-decreasing on very successful ones ($\rho_k \geq \eta_2$). When all iterations are eventually very successful, $\Delta_k$ can in fact grow without bound or stabilise above zero.

    \begin{theorem}[The radius is bounded below, and both limits occur] \label{thm:toint-update}
        Let $\{x_k\}$ be generated by the trust-region algorithm under $\Rdelta$. Suppose there
        exists $k_0$ such that $\rho_k \geq \eta_2$ for all $k \geq k_0$. Then
        \begin{enumerate}[(i)]
            \item the sequence $\{\Delta_k\}$ is bounded below by a positive constant: there exists $\underline{\Delta} > 0$ such that $\Delta_k \geq \underline{\Delta}$ for all $k \geq 0$;
            \item both alternatives are realised by admissible selections of
                  Update~\ref{update:multiplicative factor}: either
                  $\Delta_k \to +\infty$, or
                  $\Delta_k \to \Delta > 0 $ a finite positive limit.
        \end{enumerate}
    \end{theorem}

    \begin{proof}
        Under $\Rdelta$, $\rho_k \geq \eta_2$ gives $\Delta_{k+1} \geq \Delta_k$, so the sequence $\{\Delta_k\}_{k \geq k_0}$ is non-decreasing. Taking $\underline{\Delta} = \min\{\Delta_0, \dots, \Delta_{k_0}\} > 0$ gives (i). 
        For~(ii), the sequence $\{\Delta_k\}_{k \geq k_0}$ is non-decreasing, so it either diverges to $+\infty$ or converges to a finite limit. 
    \end{proof}

    The two alternatives are realised by the two admissible selections of Update~\ref{update:multiplicative factor}.
    On very successful iterations take $\Delta_{k+1} = \gthree \Delta_k$, the radius then diverges to $+\infty$. For the second case, take $\Delta_{k+1} = \exp\bigl((k+1)^{-2}\bigr) \Delta_k$, which converges to a finite limit, because $\Delta_k = \Delta_{k_0}\exp\bigl(\sum_{j=k_0}^{k-1}(j+1)^{-2}\bigr)$ and its exponent increases to a limit bounded by $\sum_{j \geq 0}(j+1)^{-2} = \pi^2/6$.
    The second selection is admissible only where
    $\exp\bigl((k+1)^{-2}\bigr) \leq \gthree$; the left-hand side decreases to $1$, so this holds
    for every $k \geq k_0$ as soon as $\gthree \geq \exp\bigl((k_0+1)^{-2}\bigr)$, and for all
    large $k$ whatever $\gthree > 1$.%

    The criticality-anchored mechanisms \emph{fail outright} in the second-order setting when they
    keep $\norm{g_k}$ as their criticality measure: their radius collapses at a saddle
    (Section~\ref{subsec:Danger of saddle points}). Replacing $\norm{g_k}$ by $\tau_k$ repairs them,
    for the following reason. Near a saddle $\norm{g_k} \to 0$, but with a
    Hessian approximation satisfying \ref{as: Approximation of the Hessian} the quantity
    $-\lambda_k^-$ stays bounded away from zero, so $\tau_k$ does too and the radius inherits that
    floor. It is then Theorem~\ref{thm: second-order convergence}, not the mechanism, that converts
    the floor into escape from the saddle. We write $\Rdfo(\tau)$ and $\Rgrad(\tau)$ for the
    re-anchored rules.

    \subsection{Discussion on saddle points} \label{subsec:Danger of saddle points}

     A first-order critical point need not be a local minimum. Consider $f(x, y) = x^4 - x^2 + y^2$: it has a saddle at $(0, 0)$ and two local minima at $(\pm 1, 0)$. A trust-region algorithm may converge to any first-order critical point, depending on the model Hessian. We first state the classical result for nonsingular $\nabla^2 f(x^*)$, then show that criticality-anchored rules force convergence to $x^*$ without imposing conditions on the model Hessian.

        \begin{theorem}\label{thm: convergence to potential saddle point}
            Suppose a subsequence $\{x_{k_j}\}$ converges to a first-order critical point $x^*$ where $\nabla^2 f(x^*)$ is nonsingular. Assume that $f$ is twice continuously differentiable in a neighbourhood of $x^*$, bounded below, the gradient is Lipschitz continuous, and the model Hessian is uniformly bounded, $\norm{H_k} \leq \mykappa{umh} < +\infty$. If Condition~\ref{cond: s_k bounded by g_k} holds, then $\{x_k\}$ converges to $x^*$.
        \end{theorem}
        \begin{proof}
            See the proof of Theorem~6.5.2 in~\citet{ConnGoulToin00}. The proof is adapted by considering
            \[
                \delta_0 = \frac{1}{4} \min\{1, \sigma, 1/\overline{\kappa}\}, \qquad \delta_1 = \frac{\delta}{\overline{\kappa} \delta_0 + 1},
            \]
            where $\sigma > 0$ is the smallest singular value of $\nabla^2 f(x^*)$ and $\delta$ is a small positive number. The iterates stay in $B(x^*, \delta_1)$ for all $k \geq k_{j_0}$ once $\norm{g_k} \leq \delta_0 \delta_1$, the entry index being chosen as in the proof of Theorem~\ref{thm: step bounded by gradient norm implies convergence}, that is with $k_{j_0} \geq k_\kappa$, which is possible because $\{x_{k_j}\}$ has infinitely many terms in the ball.%
        \end{proof}

        \begin{remark}
            \citet{ConnGoulToin00} obtain Condition~\ref{cond: s_k bounded by g_k} (with $\overline{\kappa} = 2/\mykappa{lbm}$) by imposing $H_k \succeq \mykappa{lbm} I$ for $x_k$ close to $x^*$ (Lemma~6.5.1 of~\cite{ConnGoulToin00}). A poor model Hessian that changes the sign of the model curvature can therefore force convergence to a saddle point.
        \end{remark}

        The capped gradient-scaled rule obtains Condition~\ref{cond: s_k bounded by g_k}
        for free, and $\Rdfo$ supplies directly the radius bound that the proof of the theorem
        consumes; both therefore avoid this failure mode without any hypothesis on the model
        Hessian.

        \begin{corollary}\label{cor: convergence criticality anchored saddle}
            Under the hypotheses of Theorem~\ref{thm: convergence to potential saddle point},
            $\{x_k\}$ converges to $x^*$ if the mechanism is $\Rgrad$ with $\mu_k \leq
            \overline{\mu}$ for all $k$, or if it is $\Rdfo$.
        \end{corollary}
        \begin{proof}
            For $\Rgrad$, Condition~\ref{cond: s_k bounded by g_k} holds with
            $\overline{\kappa} = \overline{\mu}$ and $k_\kappa = 0$ because
            $\norm{s_k} \leq \Delta_k = \mu_k\norm{g_k} \leq \overline{\mu}\norm{g_k}$
            at every iteration, and
            Theorem~\ref{thm: convergence to potential saddle point} applies.

            For $\Rdfo$ we show directly that $\norm{s_k} \leq \Delta_k \leq \delta_0\delta_1$
            for large $k$, with $\delta_1 = \delta/(\delta_0+1)$, which is what the proof
            of~\citet[Thm.~6.5.2]{ConnGoulToin00} consumes. Take $j_0$ with
            $\norm{x_{k_j} - x^*} \leq \delta_1$ for $j \geq j_0$; since
            $\lim_k\norm{g_k} = 0$ (Theorems~3.2 and~2.4 of \PartI), we may also assume
            $\norm{g_k} \leq \min\{1, 1/(\zeta\gamma_3)\}\delta_0\delta_1$ for $k \geq k_{j_0}$.
            Put $\calS^+ = \{k : \Delta_k \leq \zeta\norm{g_k} \text{ and } \rho_k \geq \eta_1\}$
            and $\calS^+_0 = \calS^+ \cap [k_{j_0}, +\infty)$, the elements at which the
            bound on $\norm{g_k}$ just fixed is available.
            If $\calS^+_0$ is finite, then every iteration beyond its largest
            element, or beyond $k_{j_0}$ when it is empty, contracts the radius by $\gtwo$, so
            $\Delta_k \to 0$ and $\Delta_k \leq \delta_0\delta_1$ for all large $k$. If it is
            infinite, let $k$ lie beyond the first element of $\calS^+_0$. Either
            $k \in \calS^+$ and $\Delta_k \leq \zeta\norm{g_k} \leq
            \delta_0\delta_1$, or $k \notin \calS^+$ and, with
            $k^- = \max \calS^+_0 \cap \{1,\dots,k\} \geq k_{j_0}$,
            $\Delta_k \leq \zeta\gamma_3\gamma_2^{\,k-k^--1}\norm{g_{k^-}} \leq \delta_0\delta_1$.%
            
            In every case $\norm{x_k - x^*} \leq \delta_1 \leq \delta$ for all large $k$,
            and $\delta$ is arbitrary.
        \end{proof}

        \begin{remark}
            Whether the iterates escape a saddle is decided by the model Hessian and by the
            starting point, not by the radius rule. \PartIII{} examines this empirically on a
            quartic with a nonsingular saddle and on a family with tunable escape rate.
        \end{remark}

%=======================================================================
%   SECTION 5: SECOND-ORDER CONVERGENCE AND COMPLEXITY
%=======================================================================
\section{Second-order convergence and complexity} \label{sec:second-order convergence}
%=======================================================================

The second-order theory of trust-region methods with quadratic models is settled.
Convergence to second-order critical points is Chapter~6.6 of~\citet{ConnGoulToin00}, and the
worst-case bound $\calO(\max\{\varepsilon_g^{-2}\varepsilon_H^{-1}, \varepsilon_H^{-3}\})$,
sharp for this class, is Theorem~4.7 of~\citet{CartGoulToin12}. This section does not reprove
that theory. What it does is replace the hypotheses under which those results are stated,
which bear on the model and on the algorithm, by a single structural condition on the sequence
$\{\Delta_k\}$, Condition~\ref{condition: lower bound Delta (2nd order)} below, so that
Section~\ref{sec:second-order mechanism survey} can then ask, rule by rule, which mechanisms
supply it. 
\citet{GrapigliaYuanYuan15, GrapigliaYuanYuan16} do comparable work for the first- and
second-order theory within the nonlinear stepsize control framework of~\citet{Toint13}, and the
overlap deserves to be stated precisely rather than waved past. That framework updates a scalar
parameter $\delta_k$ by reading $\rho_k$ alone, and sets the radius to a function
$\Delta(\delta_k, \chi_k)$ of that parameter and a criticality measure $\chi_k$, in some cases
$\norm{g_k}$. At second order the form is restricted to
$\Delta_k = \delta_k^{\alpha}\,\tau_k^{\beta}$, the measure being
$\tau_k = \max\{\chi_k, -\lambda_k^-\}$, which is the measure used here when $\chi_k = \norm{g_k}$. Taking
$(\alpha,\beta) = (1,0)$ gives $\Rdelta$ and $(\alpha,\beta) = (1,1)$ gives $\Rgrad$, so those two
mechanisms are already covered, and the complexity bound of
Theorem~\ref{thm: Second-order worst-case complexity} is theirs for them. The three remaining
families lie outside the prescribed form, each for a different reason. Their stepsize update leaves
$\delta_{k+1} \geq \gone\delta_k$ on every branch, and $\Delta(\cdot,\chi)$ is concave with
$\Delta(0,\chi) = 0$, so $\Delta(\gone\delta,\chi) \geq \gone\Delta(\delta,\chi)$ and the radius
cannot fall below $\gone\Delta_k$ in one iteration. $\Rstep$ admits
$\Delta_{k+1} < \gtwo\norm{s_k}$ at an unsuccessful iteration, and $\norm{s_k}$ may be arbitrarily
small relative to $\Delta_k$ when the step is interior. The criticality test of $\Rdfo$ fires
whenever $\Delta_k > \zeta\omega_k$ ($\omega_k$ being a first- or second-order criticality measure), including at $\rho_k \geq \eta_2$, where the framework requires $\delta_{k+1} \geq \delta_k$ and so forbids the contraction that test demands. And
$\widetilde\rho_{k+1}$, which selects the branch for $\Rrtr$ and $\Rgrtr$, is read off $m_{k+1}$,
a model chosen one step after the parameter has been updated: the quantity does not yet exist at
the moment the framework needs it. The difference is
not one of generality but of kind: their hypothesis prescribes the algebraic form of the rule,
whereas Condition~\ref{condition: lower bound Delta (2nd order)} prescribes a behaviour, and any
rule exhibiting that behaviour is covered whatever its algebra. That is what admits the three
remaining families.
Accordingly, we recall the ingredients and cite their proofs, and prove only the two
statements that the change of hypotheses actually alters: the unified threshold of
Theorem~\ref{thm: successful iteration for low Delta (2nd order)}, which merges the first- and
second-order criticality measures into one, and the constant in
Theorem~\ref{thm: Second-order worst-case complexity}, which now depends on
$\mykappa{lbd}^{(2)}$ rather than on a rule.

\subsection{Second-order decrease}
\label{subsec:second-order decrease conditions}

Write $\lambda_k^-$ for the leftmost eigenvalue of the model Hessian $H_k$. When
$\lambda_k^- < 0$ one can compute a unit direction $u_k$ with
\begin{equation}
    \langle u_k, g_k \rangle \leq 0,
    \qquad \norm{u_k} = 1,
    \qquad \langle u_k, H_k u_k \rangle \leq \mykappa{snc}\, \lambda_k^- ,
    \label{eqn: sufficient negative curvature direction}
\end{equation}
for a constant $\mykappa{snc} \in (0, 1]$; the value $\mykappa{snc} = 1$ corresponds to an exact
leftmost eigenvector, and any Lanczos process applied to $H_k$ produces such a $u_k$ with
$\mykappa{snc}$ bounded away from zero after enough iterations. Since the model restricted to
$\{t u_k : t \geq 0\}$ has non-positive linear and negative quadratic coefficients, it decreases
monotonically in $t$, and the constrained minimiser over $\calB_k$ is the \emph{eigenpoint}
\begin{equation}
    s_k^E \defeq \Delta_k u_k , \qquad \norm{s_k^E} = \Delta_k .
    \label{eqn: eigenpoint}
\end{equation}

\begin{theorem}[Eigenpoint decrease, {\citep[Thm.~6.6.1]{ConnGoulToin00}}]
\label{thm: Eigenpoint decrease}
    If $\lambda_k^- < 0$ and $u_k$ satisfies~\eqref{eqn: sufficient negative curvature direction},
    then
    \begin{equation*}
        m_k(0) - m_k(s_k^E) \;\geq\; \tfrac{1}{2}\,\mykappa{snc}\, \abs{\lambda_k^-}\, \Delta_k^2 .
    \end{equation*}
\end{theorem}

The one-line verification, the linear term of $m_k$ along $u_k$ is non-negative
by~\eqref{eqn: sufficient negative curvature direction} and the quadratic term is at least
$\tfrac12\mykappa{snc}\abs{\lambda_k^-}\Delta_k^2$, is that of
\citet[Thm.~6.6.1]{ConnGoulToin00}, stated there for an exact leftmost eigenvector; the factor
$\mykappa{snc}$ is what an inexact direction costs.
The factor $\mykappa{snc}$ cannot be dropped. With an inexact direction the guaranteed decrease is
$\mykappa{snc}$ times smaller. As in \PartI, only a fraction of this decrease is required.
\begin{assum}{a.Ed}[eigenpoint decrease]
\label{as: Eigenpoint decrease}
    If $\lambda_k^- < 0$, then for some $\mykappa{fed} \in (0, 1]$,
    \begin{equation*}
        m_k(0) - m_k(s_k) \;\geq\; \tfrac{1}{2}\,\mykappa{fed}\, \abs{\lambda_k^-}\, \Delta_k^2 .
    \end{equation*}
\end{assum}

\begin{assum}{a.CdEd}[maximum of Cauchy and eigenpoint decrease]
\label{as: max Cauchy/eigen decrease}
    For all $k$ and some $\mykappa{moc} \in (0, 1]$,
    \begin{equation*}
        m_k(0) - m_k(s_k) \;\geq\; \frac{\mykappa{moc}}{2}
        \max\left\{ \norm{g_k}\min\!\left(\frac{\norm{g_k}}{\norm{H_k}},\, \Delta_k\right),\;
                    \max(0, -\lambda_k^-)\, \Delta_k^2 \right\} .
    \end{equation*}
\end{assum}

Taking the better of the Cauchy point and the eigenpoint satisfies
\ref{as: max Cauchy/eigen decrease} with $\mykappa{moc} = \min\{\mykappa{mdc}, \mykappa{snc}\}$.

\begin{assum}{f.LH}[Lipschitz Hessian]
\label{as: Lipschitz Hessian}
    There is $L_H > 0$ with
    $\norm{\nabla^2 f(x) - \nabla^2 f(y)} \leq L_H \norm{x - y}$ for all $x, y$.
\end{assum}

\begin{assum}{a.BH}[Hessian approximation along the step, \citep{CartGoulToin12}]
\label{as: Approximation of the Hessian}
    For some $\mykappa{BH} > 0$,
    $\norm{\bigl(H_k - \nabla^2 f(x_k)\bigr) s_k} \leq \mykappa{BH} \norm{s_k}^2$.
\end{assum}

\subsection{The second-order lower-bound condition}
\label{subsec:second-order lower bound}

\begin{definition}
    The \emph{second-order criticality measure} is
    \begin{equation}
        \tau_k \defeq \max\{\norm{g_k},\, -\lambda_k^-\} \;\geq\; 0 .
        \label{eqn:tau definition}
    \end{equation}
\end{definition}
The measure $\tau_k$ is a property of the pair $(x_k, H_k)$, not of $x_k$ alone, since
$\lambda_k^-$ is the leftmost eigenvalue of the \emph{model} Hessian.

\begin{theorem}[after {\citep[Lem.~4.3]{CartGoulToin12}}]
\label{thm: second-order dynamics on rho_k}
    Under \ref{as: Lipschitz Hessian}, \ref{as: model agreement with function},
    \ref{as: max Cauchy/eigen decrease} and \ref{as: Approximation of the Hessian}, for any
    $\eta \in [0,1)$ and any $k$ with $\lambda_k^- < 0$,
    \begin{equation*}
        \Delta_k \;\leq\; \frac{3\,\mykappa{moc}\,(1 - \eta)}{L_H + 3\,\mykappa{BH}}\,
        \abs{\lambda_k^-}
        \qquad\Longrightarrow\qquad \rho_k \geq \eta .
    \end{equation*}
\end{theorem}

\begin{proof}
    Write
    \begin{equation*}
        \abs{1 - \rho_k} \leq
        \frac{\abs{f(x_k + s_k) - f(x_k) - \langle g_k, s_k\rangle
                   - \tfrac12 s_k^\top \nabla^2 f(x_k) s_k}
              + \tfrac12\abs{s_k^\top (\nabla^2 f(x_k) - H_k) s_k}}
             {m_k(0) - m_k(s_k)} .
    \end{equation*}
    \ref{as: Lipschitz Hessian} bounds the first numerator term by
    $\tfrac{L_H}{6}\norm{s_k}^3$; \ref{as: Approximation of the Hessian} with
    Cauchy--Schwarz bounds the second by $\tfrac{\mykappa{BH}}{2}\norm{s_k}^3$;
    \ref{as: max Cauchy/eigen decrease} bounds the denominator below by
    $\tfrac{\mykappa{moc}}{2}\abs{\lambda_k^-}\Delta_k^2$. Using $\norm{s_k} \leq \Delta_k$,
    \begin{equation*}
        \abs{1 - \rho_k}
        \;\leq\; \frac{\bigl(\tfrac{L_H}{6} + \tfrac{\mykappa{BH}}{2}\bigr)\norm{s_k}^3}
                      {\tfrac{\mykappa{moc}}{2}\abs{\lambda_k^-}\Delta_k^2}
        \;\leq\; \frac{(L_H + 3\mykappa{BH})\,\Delta_k}{3\,\mykappa{moc}\,\abs{\lambda_k^-}} ,
    \end{equation*}
    which is at most $1 - \eta$ under the stated bound on $\Delta_k$.
\end{proof}
This is Lemma~4.3 of~\citet{CartGoulToin12}, whose constant $\kappa_{\mathrm{fm}}$ we
have made explicit as $\tfrac{L_H}{6} + \tfrac{\mykappa{BH}}{2}$, since
Theorem~\ref{thm: successful iteration for low Delta (2nd order)} needs its value.

Combining with the first-order analogue (Theorem~2.2 of \PartI, which gives $\rho_k \geq \eta$
whenever $\Delta_k \leq \mykappa{moc}(1-\eta)\norm{g_k}/(L + \norm{H_k})$) yields a single
threshold in $\tau_k$. Set
\begin{equation}
    \widetilde{L}_k \;\defeq\; \frac{1}{\mykappa{moc}}
    \max\left\{ L + \norm{H_k},\; \frac{L_H + 3\,\mykappa{BH}}{3} \right\} .
    \label{eqn: tilde L}
\end{equation}

\begin{theorem}
\label{thm: successful iteration for low Delta (2nd order)}
    Under the assumptions of Theorem~\ref{thm: second-order dynamics on rho_k}, for every
    $\eta \in [0,1)$,
    \begin{equation*}
        \Delta_k \;\leq\; \frac{1 - \eta}{\widetilde{L}_k}\, \tau_k
        \qquad\Longrightarrow\qquad \rho_k \geq \eta .
    \end{equation*}
\end{theorem}

\begin{proof}
    Since $\widetilde{L}_k$ is a maximum, $(1-\eta)\tau_k/\widetilde{L}_k$ lies below each of the
    two individual thresholds. If $\tau_k = \norm{g_k}$, the bound gives
    $\Delta_k \leq \mykappa{moc}(1-\eta)\norm{g_k}/(L+\norm{H_k})$ and Theorem~2.2 of \PartI{}
    applies. Otherwise $\tau_k = \abs{\lambda_k^-} > 0$ and the bound gives
    $\Delta_k \leq 3\mykappa{moc}(1-\eta)\abs{\lambda_k^-}/(L_H + 3\mykappa{BH})$, so
    Theorem~\ref{thm: second-order dynamics on rho_k} applies.
\end{proof}

\begin{corollary}
\label{cor: successful iteration for low Delta (2nd order)}
    For every $\eta \in [0,1)$,
    $\ \rho_k < \eta \;\Longrightarrow\; \Delta_k > \dfrac{1-\eta}{\widetilde{L}_k}\, \tau_k$.
\end{corollary}

\begin{proof}
    This is Theorem~\ref{thm: successful iteration for low Delta (2nd order)} by contraposition.
\end{proof}

Corollary~\ref{cor: successful iteration for low Delta (2nd order)} is the engine of
Section~\ref{sec:second-order mechanism survey}: a mechanism can only contract the radius from a
value already proportional to $\tau_k$. It motivates the second-order analogue of Condition~\ref{cond: PartI C1}.

\begin{cond}{C.Lb2}[second-order lower bound on the radius]
\label{condition: lower bound Delta (2nd order)}
    There are a constant $\mykappa{lbd}^{(2)} > 0$ and a non-decreasing sequence
    $\{M_k^{(2)}\}$ dominating the thresholds of~\eqref{eqn: tilde L},
    \begin{equation}
        M_k^{(2)} \;\geq\; \widetilde{L}_k \qquad \text{for all } k \geq 0 ,
        \label{eqn: M2 dominates}
    \end{equation}
    such that
    \begin{equation}
        \Delta_k \;\geq\; \frac{\mykappa{lbd}^{(2)}}{M_k^{(2)}}\, \varsigma_k ,
        \qquad
        \varsigma_k \defeq \min_{0 \leq j \leq k} \tau_j ,
        \qquad \text{for all } k \geq 0 .
        \label{eqn: condition 2nd order}
    \end{equation}
\end{cond}
\begin{remark}[The canonical choice of $M_k^{(2)}$]
\label{rem: canonical M2}
    Requirement~\eqref{eqn: M2 dominates} is what lets
    Corollary~\ref{cor: successful iteration for low Delta (2nd order)} be read against
    \eqref{eqn: condition 2nd order}, and it is met by the obvious candidate. Let $\{M_k\}$ be the
    non-decreasing sequence of Condition~\ref{cond: PartI C1} of \PartI, chosen so that $M_k \geq \norm{H_k}$, and
    put
    \begin{equation}
        M_k^{(2)} \;\defeq\; \frac{1}{\mykappa{moc}}
        \max\left\{ L + M_k,\; \frac{L_H + 3\,\mykappa{BH}}{3} \right\} .
        \label{eqn: canonical M2}
    \end{equation}
    This is non-decreasing because $\{M_k\}$ is, and dominates $\widetilde{L}_k$ termwise because
    $M_k \geq \norm{H_k}$. It is also the exact second-order counterpart of the first-order
    $M_k$, differing only in the outer maximum and the factor $\mykappa{moc}^{-1}$; the
    growth condition on $\{M_k\}$ of \PartI{} transfers to $\{M_k^{(2)}\}$ unchanged, since the
    two sequences differ by an affine map with positive slope.
\end{remark}

\subsection{Global convergence to second-order critical points}
\label{subsec:second-order global cv}

\begin{theorem}[{\citep[Thm.~6.6.4]{ConnGoulToin00}}, under
Condition~\ref{condition: lower bound Delta (2nd order)}]
\label{thm: second-order convergence}
    \emergencystretch=2em  % lets the hypothesis list break without an overfull line
    Suppose \ref{as: Bounded below}, \ref{as: Lipschitz continuous gradient}, \ref{as: Lipschitz Hessian},
    \ref{as: max Cauchy/eigen decrease}, \ref{as: Approximation of the Hessian},
    $\norm{H_k} \leq \mykappa{umh}$, $\eta > 0$, Condition~\ref{cond: PartI C2} of \PartI, and
    Condition~\ref{condition: lower bound Delta (2nd order)} with
    $M^{(2)} \defeq \sup_k M_k^{(2)} < +\infty$. Then
    $\liminf_{k\to\infty} \tau_k = 0$.
\end{theorem}

\begin{proof}
    Suppose not, so that $\liminf_k \tau_k = \varepsilon_0 > 0$. Only finitely many indices
    satisfy $\tau_k < \varepsilon_0/2$, and each of those has $\tau_k > 0$, since $\tau_k = 0$
    means $x_k$ is second-order critical and the method terminates. Hence $\varepsilon \;\defeq\; \inf_{k \geq 0} \tau_k \;>\; 0$, 
    so $\varsigma_k \geq \varepsilon$ for \emph{all} $k$, and
    Condition~\ref{condition: lower bound Delta (2nd order)} gives
    \begin{equation}
        \Delta_k \;\geq\; \frac{\mykappa{lbd}^{(2)}\,\varsigma_k}{M_k^{(2)}}
        \;\geq\; \frac{\mykappa{lbd}^{(2)}\,\varepsilon}{M^{(2)}}
        \;=:\; \Delta_{\min} \;>\; 0
        \qquad \text{for all } k .
        \label{eqn: Delta bounded below}
    \end{equation}

    \emph{The set $\calS$ of successful iterations is infinite.} If it were finite, every
    iteration beyond $\max\calS$ would be unsuccessful, and Condition~\ref{cond: PartI C2} of \PartI{} contracts
    the radius by a factor at most $\gtwo < 1$ at each of them, so $\Delta_k \to 0$, contradicting
    \eqref{eqn: Delta bounded below}.

    \emph{A uniform decrease on $\calS$.} Let $k \in \calS$. Since
    $\tau_k \geq \varepsilon$, at least one of $\norm{g_k} \geq \varepsilon$ and
    $-\lambda_k^- \geq \varepsilon$ holds. In the first case
    \ref{as: max Cauchy/eigen decrease} and~\eqref{eqn: Delta bounded below} give
    \begin{equation*}
        m_k(0) - m_k(s_k) \;\geq\; \frac{\mykappa{moc}}{2}\,\varepsilon
        \min\left\{\frac{\varepsilon}{\mykappa{umh}},\, \Delta_{\min}\right\} ,
    \end{equation*}
    and in the second
    $m_k(0) - m_k(s_k) \geq \tfrac{\mykappa{moc}}{2}\varepsilon\,\Delta_{\min}^2$. Writing
    $\delta > 0$ for the smaller of the two right-hand sides, $\rho_k \geq \eta > 0$ gives
    $f(x_k) - f(x_{k+1}) \geq \eta\,\delta$ at every $k \in \calS$. Summing over the infinite set
    $\calS$ contradicts $f$ being bounded below.
\end{proof}

The argument is that of \citet[Thm.~6.6.4]{ConnGoulToin00}. The only substantive change is the
source of the radius floor~\eqref{eqn: Delta bounded below}. There it follows from their model and
algorithm assumptions AM.5--AM.6 and AA.2, here from
Condition~\ref{condition: lower bound Delta (2nd order)} alone, which is what makes the conclusion
available to any mechanism supplying that condition rather than to one fixed rule.

A good approximation of the leftmost eigenvector therefore suffices for convergence to second-order critical points. \citet{ConnGoulToin00} obtain one from the coherence of the model Hessian with the true Hessian, \ref{as: naive model Hessian CV}.

\subsection{Worst-case complexity}
\label{subsec:second-order worst-case complexity}

The bound below is Lemma~4.5 and Theorem~4.7 of~\citet{CartGoulToin12}, restated under
Condition~\ref{condition: lower bound Delta (2nd order)} in place of their model and algorithm
hypotheses; the eigenpoint branch is already supplied by
\ref{as: max Cauchy/eigen decrease}.

\begin{theorem}[Second-order worst-case complexity]
\label{thm: Second-order worst-case complexity}
    Under the assumptions of Theorem~\ref{thm: second-order convergence}
    together with Condition~\ref{cond: PartI C3} of \PartI, let
    $\varepsilon_g$ and $\varepsilon_H$ lie in $(0,1]$ and set
    \begin{equation}
        \varepsilon \;\defeq\; \min\{\varepsilon_g, \varepsilon_H\} ,
        \qquad
        c \;\defeq\; \frac{\mykappa{lbd}^{(2)}}{M^{(2)}} \;>\; 0 ,
        \qquad M^{(2)} = \sup_k M_k^{(2)} .
        \label{eqn: complexity c}
    \end{equation}
    Then the number of successful iterations before an iterate with
    $\norm{g_k}\leq\varepsilon_g$ and $-\lambda_k^-\leq\varepsilon_H$ is reached is at most
    \begin{equation}
        \frac{2\bigl(f(x_0) - f_{\mathrm{low}}\bigr)}{\eta\,\mykappa{moc}}
        \left[
            \frac{1}{\varepsilon_g \min\{\varepsilon_g/\mykappa{umh},\, c\,\varepsilon\}}
            \;+\;
            \frac{1}{c^2\, \varepsilon_H\, \varepsilon^2}
        \right] ,
        \label{eqn: complexity mixed}
    \end{equation}
    and hence, by Condition~\ref{cond: PartI C3} of \PartI{} bounding unsuccessful iterations in terms of
    successful ones,
    the total satisfies
    \begin{equation*}
        T_{\varepsilon}^{g,H}
        \;=\; \calO\!\left( \max\{\varepsilon_g^{-2}\varepsilon_H^{-1},\,
            \varepsilon_H^{-3}\} \right)
        \;=\; \calO\!\left(\varepsilon^{-3}\right) \text{ when }
            \varepsilon_g = \varepsilon_H = \varepsilon .
    \end{equation*}
\end{theorem}

% \begin{remark}[Why the mixed term cannot be dropped]
% \label{rem: mixed complexity term}
%     The bound~\eqref{eqn: complexity mixed} is driven by the radius floor
%     $\Delta_k \geq c\,\varsigma_k$ of
%     Condition~\ref{condition: lower bound Delta (2nd order)}, and before termination only one of
%     the two tests need fail: the guarantee available is $\tau_k > \min\{\varepsilon_g,
%     \varepsilon_H\}$, hence $\Delta_k \geq c\,\varepsilon$, and not $\Delta_k \gtrsim
%     \varepsilon_H$. The curvature term therefore carries $\varepsilon^{-2}$ rather than
%     $\varepsilon_H^{-2}$, and no separation of $\calS$ can repair this, since a single sequence
%     $\{\Delta_k\}$ serves both branches. The resulting order is exactly the one obtained
%     by~\citet{CartGoulToin12} for second-order trust-region methods, and it collapses to
%     $\calO(\varepsilon^{-3})$ as soon as the two tolerances are of the same order.
% \end{remark}

\begin{proof}
    The counting argument is that of \citet[Lem.~4.5]{CartGoulToin12} and we do not repeat it; we
    record only the step at which the hypotheses differ. There, the radius floor on which the
    count rests is supplied by their Lemma~4.4, proved for the trust-region rule they fix. Here it
    is supplied by Condition~\ref{condition: lower bound Delta (2nd order)}: before termination at
    least one of the two tests fails at every iteration, so $\tau_k > \varepsilon$ and hence
    $\varsigma_k > \varepsilon$, and since $M_k^{(2)} \leq M^{(2)}$,
    \begin{equation}
        \Delta_k \;\geq\; \frac{\mykappa{lbd}^{(2)}}{M_k^{(2)}}\,\varsigma_k
        \;\geq\; \frac{\mykappa{lbd}^{(2)}}{M^{(2)}}\,\varepsilon \;=\; c\,\varepsilon .
        \label{eqn: complexity floor}
    \end{equation}
    With that floor in hand the count proceeds verbatim. Split $\calS$ into
    $\calS_g = \{k \in \calS : \norm{g_k} > \varepsilon_g\}$ and
    $\calS_H = \calS\setminus\calS_g$, on which $-\lambda_k^- > \varepsilon_H$ because the
    iteration has not terminated. By \ref{as: max Cauchy/eigen decrease},
    \eqref{eqn: complexity floor} and $\rho_k \geq \eta$, each iteration of $\calS_g$ reduces $f$
    by at least
    $\tfrac{\eta\mykappa{moc}}{2}\varepsilon_g\min\{\varepsilon_g/\mykappa{umh}, c\varepsilon\}$
    and each iteration of $\calS_H$ by at least
    $\tfrac{\eta\mykappa{moc}}{2}\varepsilon_H(c\varepsilon)^2$. Boundedness below bounds
    $\abs{\calS_g}$ and $\abs{\calS_H}$ by the corresponding terms of
    \eqref{eqn: complexity mixed}, and $\abs{\calS} = \abs{\calS_g} + \abs{\calS_H}$.

    For the stated order, the second term needs no inequality at all: since
    $\varepsilon^{-2} = \max\{\varepsilon_g^{-2}, \varepsilon_H^{-2}\}$,
    \begin{equation*}
        \frac{1}{c^2\,\varepsilon_H\,\varepsilon^2}
        \;=\; \frac{1}{c^2}
            \max\bigl\{ \varepsilon_g^{-2}\varepsilon_H^{-1},\;
                        \varepsilon_H^{-3} \bigr\} .
    \end{equation*}
    For the first term,
    \begin{equation*}
        \min\{\varepsilon_g/\mykappa{umh},\, c\varepsilon\}^{-1}
        = \max\{\mykappa{umh}\varepsilon_g^{-1},\, c^{-1}\varepsilon^{-1}\} ,
    \end{equation*}
    so it equals
    $\max\{\mykappa{umh}\varepsilon_g^{-2},\, c^{-1}\varepsilon_g^{-1}\varepsilon^{-1}\}$. Here
    $\varepsilon_g^{-2} \leq \varepsilon_g^{-2}\varepsilon_H^{-1}$ because $\varepsilon_H \leq 1$,
    and $\varepsilon_g^{-1}\varepsilon^{-1} = \max\{\varepsilon_g^{-2},
    \varepsilon_g^{-1}\varepsilon_H^{-1}\} \leq \varepsilon_g^{-2}\varepsilon_H^{-1}$ because
    $\varepsilon_g \leq 1$ and $\varepsilon_H \leq 1$. Both terms are therefore
    $\calO(\max\{\varepsilon_g^{-2}\varepsilon_H^{-1}, \varepsilon_H^{-3}\})$, and taking
    $\varepsilon_g = \varepsilon_H = \varepsilon$ gives $\calO(\varepsilon^{-3})$.
\end{proof}

    This bound is sharp for trust-region methods with quadratic models: \citet{CartGoulToin10}
    show the exponent $2$ in $\varepsilon_g^{-2}$ cannot be improved for trust-region and
    Newton-type methods, and \citet[\S4]{CartGoulToin12} exhibit a unidimensional example
    attaining~\eqref{eqn: complexity mixed}. It is not optimal among second-order methods, cubic
    regularisation attaining $\calO(\max\{\varepsilon_g^{-3/2}, \varepsilon_H^{-3}\})$
    \citep{NestPoly06, CartGoulToin11}, which is optimal for the wider
    class~\citep{CartGoulToin22}. The $\varepsilon_H^{-3}$ term is common to both and is sharp.

    \citet{GrapigliaYuanYuan16} obtain second-order complexity bounds of the same order for the
    nonlinear stepsize control framework of \citet{Toint13}, which subsumes several of the radius
    mechanisms surveyed here. Their analysis fixes the stepsize rule inside the framework;
    Section~\ref{sec:second-order mechanism survey} instead asks which of the surveyed mechanisms
    supply Condition~\ref{condition: lower bound Delta (2nd order)}.

%=======================================================================
%  Section 6, rewritten for the iteration-dependent threshold.
%
%  WHAT CHANGED, AND WHY
%
%  1. Every constant and every conversion now uses $M_k^{(2)}$ rather than a
%     fixed $\widetilde{L}$.  With $\widetilde{L}_k$ iteration-dependent, the
%     old inductions on $\min\{\Delta_0, c\varsigma_{k-1}\}$ no longer close.
%
%  2. The four radius-based verifications now share one induction, isolated as
%     Lemma~\ref{lem: radius induction}.  Each proposition supplies only its
%     constant $c$.  The $\varsigma_{-1} \defeq +\infty$ convention is gone.
%
%  3. The claim that truncated CG satisfies \ref{as: boundary step} was
%     false and is corrected; see Remark~\ref{rem: sob solvers}.
%
%  4. The constant of Lemma~\ref{lem: retro pred difference} is carried
%     symbolically as $\mykappa{ret}$ --- see the note there.
%
%  OPEN ITEMS FLAGGED IN PLACE:
%   - Prop.~\ref{prop: so rtr}: expansion branch written $\gthree$; the draft
%     had $\gtwo$ in the proof and $\gthree$ elsewhere.  Check app:~R-RTR.
%   - $\mykappa{ret}$: derive from \ref{as: Approximation of the Hessian}
%     or leave symbolic.
%=======================================================================

%=======================================================================
\section{Verification of the second-order condition for individual mechanisms}
\label{sec:second-order mechanism survey}
%=======================================================================

\PartI{} verified Condition~\ref{cond: PartI C1}, the first-order lower bound
$\Delta_k \geq \mykappa{lbd} M_k^{-1}\min_{j\leq k}\norm{g_j}$, rule by rule. We do the same for
Condition~\ref{condition: lower bound Delta (2nd order)}. 
Throughout, $\{M_k^{(2)}\}$ is the non-decreasing sequence of
Condition~\ref{condition: lower bound Delta (2nd order)}, so that $M_k^{(2)} \geq \widetilde{L}_k$
for every $k$, and $\varsigma_k = \min_{j \leq k}\tau_j$. We assume $\tau_0 > 0$; otherwise $x_0$
is second-order critical and the algorithm terminates at once.

%-----------------------------------------------------------------------
\subsection{One induction, four mechanisms}
\label{subsec:so induction}
%-----------------------------------------------------------------------

All the radius-based verifications below run the same induction, which we isolate once. 

\begin{lemma}[Induction skeleton]
\label{lem: radius induction}
    Let $\{M_k^{(2)}\}$ be as in Condition~\ref{condition: lower bound Delta (2nd order)} and let
    $\tau_0 > 0$. Suppose there is $c > 0$ such that at every $k \geq 0$ at least one of
    \begin{equation}
        \text{(a)}\quad \Delta_{k+1} \geq \Delta_k ,
        \qquad\qquad
        \text{(b)}\quad \Delta_{k+1} \geq \frac{c\,\tau_k}{M_k^{(2)}}
        \label{eqn: skeleton branches}
    \end{equation}
    holds. Then Condition~\ref{condition: lower bound Delta (2nd order)} holds with
    \begin{equation}
        \mykappa{lbd}^{(2)} \;=\; \min\left\{\frac{\Delta_0\, M_0^{(2)}}{\tau_0},\; c\right\}
        \;>\; 0 .
        \label{eqn: skeleton constant}
    \end{equation}
\end{lemma}

\begin{proof}
    Write $\kappa \defeq \mykappa{lbd}^{(2)}$ and induct on
    $\Delta_k \geq \kappa\,\varsigma_k/M_k^{(2)}$.

    At $k = 0$, $\varsigma_0 = \tau_0$ and $\kappa \leq \Delta_0 M_0^{(2)}/\tau_0$ give
    $\kappa\varsigma_0/M_0^{(2)} \leq \Delta_0$.

    Assume the bound at $k$. Under branch~(a),
    \begin{equation*}
        \Delta_{k+1} \;\geq\; \Delta_k \;\geq\; \frac{\kappa\,\varsigma_k}{M_k^{(2)}}
        \;\geq\; \frac{\kappa\,\varsigma_{k+1}}{M_{k+1}^{(2)}} ,
    \end{equation*}
    using $\varsigma_{k+1} \leq \varsigma_k$ in the numerator and $M_{k+1}^{(2)} \geq M_k^{(2)}$ in
    the denominator. Under branch~(b), $\kappa \leq c$, $\tau_k \geq \varsigma_{k+1}$ and the same
    monotonicity give
    \begin{equation*}
        \Delta_{k+1} \;\geq\; \frac{c\,\tau_k}{M_k^{(2)}}
        \;\geq\; \frac{\kappa\,\varsigma_{k+1}}{M_{k+1}^{(2)}} .
    \end{equation*}
\end{proof}

% \begin{remark}
% \label{rem: skeleton reading}
%     Branch~(a) is ``the rule did not contract'' and branch~(b) is ``the rule contracted, but from a
%     radius already proportional to $\tau_k$'', the latter being exactly what
%     Corollary~\ref{cor: successful iteration for low Delta (2nd order)} supplies. Each verification
%     below therefore reduces to two questions: on which branches does the radius fail to be
%     non-decreasing, and what constant does each of those branches yield? Nothing else about the
%     mechanism enters.
% \end{remark}

%-----------------------------------------------------------------------
\subsection{Fixed-factor update (\texorpdfstring{$\Rdelta$}{R-delta})}
\label{subsec:second-order fixed-factor}
%-----------------------------------------------------------------------

\begin{proposition}
\label{prop: so rdelta}
    Under the hypotheses of Theorem~\ref{thm: successful iteration for low Delta (2nd order)},
    $\Rdelta$ satisfies Condition~\ref{condition: lower bound Delta (2nd order)} with
    \begin{equation*}
        \mykappa{lbd}^{(2)} = \min\left\{
            \frac{\Delta_0\, M_0^{(2)}}{\tau_0},\;
            \gone(1 - \eta_2) \right\} ,
    \end{equation*}
    under no further assumption.
\end{proposition}

\begin{proof}
    Recall Update~\ref{update:multiplicative factor}. The radius is non-decreasing on the third branch and on that branch only; in particular it may \emph{decrease} at a successful iteration with $\eta_1 \leq \rho_k < \eta_2$. We apply Lemma~\ref{lem: radius induction} with the split at $\eta_2$.

    If $\rho_k \geq \eta_2$, branch~(a) of~\eqref{eqn: skeleton branches} holds.

    If $\rho_k < \eta_2$, then $\Delta_{k+1} \geq \gone\Delta_k$: on the first branch directly, on
    the second because $\gtwo \geq \gone$. Theorem~\ref{thm: successful iteration for low Delta
    (2nd order)} holds for every threshold in $[0,1)$; applied at $\eta_2$, its contrapositive is
    \begin{equation}
        \rho_k < \eta_2
        \quad\Longrightarrow\quad
        \Delta_k \;>\; \frac{1-\eta_2}{\widetilde{L}_k}\,\tau_k
        \;\geq\; \frac{1-\eta_2}{M_k^{(2)}}\,\tau_k ,
        \label{eqn: rdelta contrapositive}
    \end{equation}
    the last step by $M_k^{(2)} \geq \widetilde{L}_k$. Hence
    $\Delta_{k+1} > \gone(1-\eta_2)\tau_k/M_k^{(2)}$, which is branch~(b) with
    $c = \gone(1-\eta_2) > 0$, positive because $\eta_2 < 1$. Lemma~\ref{lem: radius induction}
    gives the stated constant.
\end{proof}

% \begin{remark}[Why the split is on $\eta_2$, not on success]
% \label{rem: rdelta dichotomy}
%     The natural dichotomy, successful versus unsuccessful, does not work here, and the reason is
%     specific to $\Rdelta$. At a successful but not very successful iteration,
%     $\eta_1 \leq \rho_k < \eta_2$, the second branch of $\Rdelta$ returns
%     $\Delta_{k+1} < \Delta_k$, so branch~(a) of the skeleton fails; and
%     \eqref{eqn: rdelta contrapositive} is unavailable at threshold $\eta$, since $\rho_k \geq \eta$
%     there, so branch~(b) fails too. That intermediate case belongs to neither.

%     Splitting at $\eta_2$ aligns the dichotomy with the rule's own monotonicity: the radius is
%     non-decreasing exactly when $\rho_k \geq \eta_2$. The price is the constant, which becomes
%     $\gone(1-\eta_2)$ rather than $\gone(1-\eta_1)$, and is smaller because $\eta_1 \leq \eta_2$.
%     This is the dichotomy used for the same family of rules in~\citet[Prop.~6.2]{Jeremy2026-I}.
% \end{remark}

The proof is the first-order one of \PartI{} with $\norm{g_k}$ replaced by $\tau_k$ throughout.
Nothing else changes, because $\Rdelta$ never reads either quantity: it is precisely its blindness
to criticality that makes it insensitive to the order of the theory.

%-----------------------------------------------------------------------
\subsection{Step-driven update (\texorpdfstring{$\Rstep$}{R-step})}
\label{subsec:second-order step-driven}
%-----------------------------------------------------------------------

Verifying Condition~\ref{condition: lower bound Delta (2nd order)} for the step-driven update $\Rstep$ needs one further assumption, a restriction on the subproblem solver.
\begin{assum}{a.Sob}[step on the boundary under negative curvature]
\label{as: boundary step}
    If $\lambda_k^- < 0$, then $\norm{s_k} = \Delta_k$.
\end{assum}
    \ref{as: boundary step} holds for the eigenpoint
    $s_k^E = \Delta_k u_k$ of~\eqref{eqn: eigenpoint}, and for any exact solution of the
    trust-region subproblem when $H_k$ is indefinite, since the constraint is then
    active~\citep{MoreSorensen1983}.
    For a truncated conjugate-gradient step it holds \emph{when the Krylov subspace generated by
    $g_k$ meets the negative-curvature subspace}, and not otherwise. 
     \ref{as: boundary step} is thus a genuine restriction on the subproblem solver rather than a
    formality, and it is needed only by the step-driven rules.
    Consider a two-dimensional example in which $g_k \parallel e_2$ lies in the positive-curvature eigenspace of
    $H_k = \operatorname{diag}(-\lambda, \mykappa{umh})$, so the first conjugate-gradient direction
    is $-g_k \parallel e_2$, the curvature along it is $\mykappa{umh} > 0$, the method steps to the
    minimiser along $e_2$, the residual vanishes, and the Steihaug--Toint iteration terminates at an
    interior point without ever generating $e_1$, although $\lambda_k^- = -\lambda < 0$. The
    data of that construction are therefore realisable by a standard solver, not
    merely admissible in the abstract.

    \begin{remark}[Why the boundary, and not a fixed fraction of it]
\label{rem: sob must be one}
    Suppose \ref{as: boundary step} were relaxed to the apparently harmless
    requirement that there be a constant $\mykappa{sob} \in (0,1)$ with
    $\norm{s_k} \geq \mykappa{sob}\Delta_k$ whenever $\lambda_k^- < 0$, a fixed
    fraction of the radius in place of the boundary itself. That relaxation is
    not available. On the very successful branch $\Rstep$ permits
    $\Delta_{k+1} = \norm{s_k}$ exactly, so $\mykappa{sob} < 1$ admits
    $\Delta_{k+1} = \mykappa{sob}\Delta_k$ at every iteration: with $\lambda_k^-$
    held at a fixed negative value the radius decays geometrically while
    $\tau_k \geq \abs{\lambda_k^-}$ stays bounded away from zero, so
    $\Delta_k \to 0$ with $\varsigma_k$ bounded below and
    Condition~\ref{condition: lower bound Delta (2nd order)} fails.
\end{remark}

\begin{lemma}[Step length under \ref{as: boundary step}]
\label{cor: step length with boundary}
    Under \ref{as: max Cauchy/eigen decrease}, \ref{as: boundary step} and
    $\norm{H_k}\leq\mykappa{umh}$,
    \begin{equation}
        \norm{s_k} \;\geq\; \mykappa{sls}
        \min\left\{\Delta_k,\, \frac{\tau_k}{\mykappa{umh}}\right\} ,
        \label{eqn: step length sob}
    \end{equation}
    with $\mykappa{sls} \defeq \mykappa{moc}/3 \in (0,1]$.
\end{lemma}

\begin{proof}
    If $\tau_k = \abs{\lambda_k^-} > 0$ then \ref{as: boundary step} gives
    $\norm{s_k} = \Delta_k$, which dominates the right-hand side because
    $\mykappa{sls} \leq 1$. 
    If $\tau_k = \norm{g_k}$, write $t \defeq \norm{s_k}$ and
    $m \defeq \min\{\Delta_k,\, \norm{g_k}/\mykappa{umh}\}$, and assume $\norm{g_k} > 0$, the
    bound being trivial otherwise.  Since $\norm{H_k} \leq \mykappa{umh}$,
    \ref{as: max Cauchy/eigen decrease} bounds the model decrease below by
    $\tfrac12\,\mykappa{moc}\norm{g_k}\,m$, while
    $m_k(0) - m_k(s_k) = -g_k^\top s_k - \tfrac12 s_k^\top H_k s_k$ bounds it above by
    $\norm{g_k}\,t + \tfrac12\,\mykappa{umh}\,t^2$.  If $t \geq m$ the claim holds outright.
    Otherwise $t < m \leq \norm{g_k}/\mykappa{umh}$, so $\mykappa{umh}\,t \leq \norm{g_k}$ and the
    upper bound is at most $\tfrac32\norm{g_k}\,t$, whence
    \[
        \tfrac12\,\mykappa{moc}\norm{g_k}\,m \;\leq\; \tfrac32\,\norm{g_k}\,t
        \qquad\text{and}\qquad
        t \;\geq\; \tfrac13\,\mykappa{moc}\,m .
    \]%
    
\end{proof}

\begin{proposition}
\label{prop: so rstep}
    Assume the hypotheses of Theorem~\ref{thm: successful iteration for low Delta (2nd order)},
    that \ref{as: boundary step} holds, and that
    Condition~\ref{cond:inexactness} holds with $\xi \defeq \sup_k \xi_k < 1$. Then $\Rstep$
    satisfies Condition~\ref{condition: lower bound Delta (2nd order)} with
    \begin{equation*}
        \mykappa{lbd}^{(2)} = \min\left\{
            \frac{\Delta_0\, M_0^{(2)}}{\tau_0},\;
            \gone\,\mykappa{sls} \min\left\{1 - \eta_2,\,
                \frac{M_0^{(2)}}{\mykappa{umh}}\right\}\!,\;
            (1-\xi)\frac{M_0^{(2)}}{\mykappa{umh}} \right\}.
    \end{equation*}
\end{proposition}

\begin{proof}
    Recall from Update~\ref{update:step-dependent update} that $\Delta_{k+1} \geq \gone\norm{s_k}$ on every branch
    and $\Delta_{k+1} \geq \norm{s_k}$ on the very successful one. As for $\Rdelta$, the split is
    at $\eta_2$, where the rule's monotonicity changes. We verify the two branches of
    Lemma~\ref{lem: radius induction}.

    \emph{Case $\rho_k < \eta_2$.} By~\eqref{eqn: rdelta contrapositive},
    $\Delta_k > (1-\eta_2)\tau_k/M_k^{(2)}$, and $M_k^{(2)} \geq M_0^{(2)}$ gives
    $\tau_k/\mykappa{umh} \geq (M_0^{(2)}/\mykappa{umh})\,\tau_k/M_k^{(2)}$. Hence, by
    Lemma~\ref{cor: step length with boundary} and monotonicity of the minimum,
    \begin{equation*}
        \Delta_{k+1} \;\geq\; \gone \mykappa{sls}
            \min\left\{\Delta_k,\, \frac{\tau_k}{\mykappa{umh}}\right\}  
            \;\geq\; \gone \mykappa{sls}
            \min\left\{1-\eta_2,\, \frac{M_0^{(2)}}{\mykappa{umh}}\right\}
            \frac{\tau_k}{M_k^{(2)}} ,
    \end{equation*}
    which is branch~(b) with $c_1 \defeq \gone\mykappa{sls}\min\{1-\eta_2,
    M_0^{(2)}/\mykappa{umh}\}$.

    \emph{Case $\rho_k \geq \eta_2$.} We use $\Delta_{k+1} \geq \norm{s_k}$ together with a lower bound on $\norm{s_k}$, and split on the sign of $\lambda_k^-$.
    \begin{itemize}[leftmargin=1.6em]
        \item \emph{$\lambda_k^- < 0$.} \ref{as: boundary step} with $\mykappa{sob} = 1$ gives
              $\norm{s_k} = \Delta_k$, so $\Delta_{k+1} \geq \Delta_k$: branch~(a).
        \item \emph{$\lambda_k^- \geq 0$}, trust region active. Then $\norm{s_k} = \Delta_k$ and again
              $\Delta_{k+1} \geq \Delta_k$: branch~(a).
        \item \emph{$\lambda_k^- \geq 0$}, trust region inactive. Then $-\lambda_k^- \leq 0$, so $\tau_k = \norm{g_k}$. The
              inactive case is the one in which Condition~\ref{cond:inexactness} constrains the
              step, the residual bound being a property of the approximate unconstrained model
              minimiser and not of a step stopped on the boundary; this is why the split is needed.
              It gives $\mykappa{umh}\norm{s_k} \geq \norm{H_k s_k} \geq (1-\xi_k)\norm{g_k}$, so
              \begin{equation*}
                  \Delta_{k+1} \;\geq\; \norm{s_k}
                  \;\geq\; \frac{1-\xi}{\mykappa{umh}}\,\tau_k
                  \;\geq\; (1-\xi)\frac{M_0^{(2)}}{\mykappa{umh}}\cdot\frac{\tau_k}{M_k^{(2)}} ,
              \end{equation*}
              which is branch~(b) with $c_2 \defeq (1-\xi)M_0^{(2)}/\mykappa{umh}$.
    \end{itemize}
    Lemma~\ref{lem: radius induction} applies with $c = \min\{c_1, c_2\}$, giving the stated
    constant.
\end{proof}

%-----------------------------------------------------------------------
\subsection{Criticality-anchored update (\texorpdfstring{$\Rdfo$}{R-DFO})}
\label{subsec:second-order DFO-like}
%-----------------------------------------------------------------------

\begin{proposition}
\label{prop: so rdfo}
    Under the hypotheses of Theorem~\ref{thm: successful iteration for low Delta (2nd order)},
    $\Rdfo(\tau)$ of Update~\ref{update: rdfo generic} satisfies
    Condition~\ref{condition: lower bound Delta (2nd order)} with
    \begin{equation*}
        \mykappa{lbd}^{(2)} = \min\left\{
            \frac{\Delta_0\, M_0^{(2)}}{\tau_0},\;
            \gone(1-\eta_1),\;
            \gtwo\,\zeta\, M_0^{(2)} \right\} ,
    \end{equation*}
    and requires no assumption on the step beyond \ref{as: max Cauchy/eigen decrease}; in
    particular \ref{as: boundary step} is not needed. The unmodified rule, with $\norm{g_k}$ in
    place of $\tau_k$, does not (Corollary~\ref{cor: convergence criticality anchored saddle}).
\end{proposition}

\begin{proof}
    Recall Update~\ref{update: rdfo generic}. The radius decreases only on the first two branches, so the third gives branch~(a) of
    Lemma~\ref{lem: radius induction} at once.

    \emph{First branch, $\rho_k < \eta_1$.} Theorem~\ref{thm: successful iteration for low Delta
    (2nd order)} applied at $\eta_1$ gives, as in~\eqref{eqn: rdelta contrapositive},
    $\Delta_k > (1-\eta_1)\tau_k/M_k^{(2)}$, whence
    $\Delta_{k+1} \geq \gone\Delta_k > \gone(1-\eta_1)\tau_k/M_k^{(2)}$: branch~(b) with
    $c_1 = \gone(1-\eta_1)$.

    \emph{Second branch, $\rho_k \geq \eta_1$ and $\Delta_k > \zeta\tau_k$.} No appeal to the
    theorem is needed: the branch condition is itself a lower bound on $\Delta_k$ in terms of
    $\tau_k$, and with $M_k^{(2)} \geq M_0^{(2)}$,
    \begin{equation*}
        \Delta_{k+1} \;\geq\; \gtwo\Delta_k \;>\; \gtwo\zeta\,\tau_k
        \;\geq\; \gtwo\zeta\, M_0^{(2)}\,\frac{\tau_k}{M_k^{(2)}} ,
    \end{equation*}
    which is branch~(b) with $c_2 = \gtwo\zeta M_0^{(2)}$. Lemma~\ref{lem: radius induction} with
    $c = \min\{c_1, c_2\}$ gives the stated constant.
\end{proof}

\subsection{Gradient-scaled update (\texorpdfstring{$\Rgrad$}{R-grad})}
\label{subsec:second-order gradient-scaled}
%-----------------------------------------------------------------------

\begin{proposition}
\label{prop: so rgrad}
    Under the hypotheses of Theorem~\ref{thm: successful iteration for low Delta (2nd order)},
    \ref{update: rgrad uncapped rule} satisfies
    Condition~\ref{condition: lower bound Delta (2nd order)} with
    \begin{equation*}
        \mykappa{lbd}^{(2)}
          = \min\left\{ \mu_0\, M_0^{(2)},\; \gone(1-\eta_2) \right\} ,
    \end{equation*}
    under no assumption on the step beyond \ref{as: max Cauchy/eigen decrease}.
\end{proposition}

\begin{proof}
    Recall Update~\ref{update: rgrad uncapped rule}.
    The multiplier decreases on the two branches with $\rho_k < \eta_2$, by $\gone$ on the first and
    by $\gtwo \geq \gone$ on the second, and is non-decreasing on the two with
    $\rho_k \geq \eta_2$; the split is therefore at $\eta_2$, as for $\Rdelta$.

    We may assume $\tau_k > 0$ at every iteration reached. If $\tau_k = 0$ then
    $\Delta_k = \mu_k\tau_k = 0$ and the rule halts; but under $\tau$-anchoring $\tau_k = 0$ means
    that $x_k$ is second-order critical, so halting is termination at the sought solution rather
    than failure. This is exactly the contrast with the unmodified rule of
    Corollary~\ref{cor: convergence criticality anchored saddle}, which halts by the same mechanism when
    $\norm{g_k} = 0$ with $\lambda_k^- < 0$, that is, at a point the second-order theory exists to
    escape.

    Rather than invoke Lemma~\ref{lem: radius induction} we induct directly on the multiplier,
    writing $\kappa \defeq \mykappa{lbd}^{(2)}$ and showing
    \begin{equation}
        \mu_k \;\geq\; \frac{\kappa}{M_k^{(2)}}
        \qquad \text{for all } k \geq 0 .
        \label{eqn: rgrad induction}
    \end{equation}
    At $k = 0$ this is $\kappa \leq \mu_0 M_0^{(2)}$, the first entry of the constant.

    If $\rho_k \geq \eta_2$, then $\mu_{k+1} \geq \mu_k \geq \kappa/M_k^{(2)} \geq
    \kappa/M_{k+1}^{(2)}$ by the induction hypothesis and monotonicity of $\{M_k^{(2)}\}$.

    If $\rho_k < \eta_2$, no induction hypothesis is needed. By~\eqref{eqn: rdelta contrapositive},
    $\Delta_k > (1-\eta_2)\tau_k/M_k^{(2)}$; since $\Delta_k = \mu_k\tau_k$ and $\tau_k > 0$,
    dividing by $\tau_k$ gives $\mu_k > (1-\eta_2)/M_k^{(2)}$, whence
    \begin{equation*}
        \mu_{k+1} \;\geq\; \gone\mu_k \;>\; \frac{\gone(1-\eta_2)}{M_k^{(2)}}
        \;\geq\; \frac{\kappa}{M_{k+1}^{(2)}} ,
    \end{equation*}
    using $\kappa \leq \gone(1-\eta_2)$ and $M_{k+1}^{(2)} \geq M_k^{(2)}$.

    Finally $\varsigma_k \leq \tau_k$ and~\eqref{eqn: rgrad induction} give
    $\Delta_k = \mu_k\tau_k \geq \kappa\,\tau_k/M_k^{(2)} \geq \kappa\,\varsigma_k/M_k^{(2)}$,
    which is Condition~\ref{condition: lower bound Delta (2nd order)}.
\end{proof}

    The induction runs on $\mu_k$ rather than on $\Delta_k$, and that is the whole difference.
    Lemma~\ref{lem: radius induction} compares the radius with $\varsigma_k$ at every step, because
    for the other mechanisms the two can drift apart. Here~\eqref{eqn: rgrad induction} is a scalar
    statement in which no criticality measure appears, and $\varsigma_k$ enters only in the final
    line. The reason is that $\Rgrad(\tau)$ controls the ratio $\Delta_k/\tau_k = \mu_k$ itself, so
    Condition~\ref{condition: lower bound Delta (2nd order)} constrains the quantity the rule
    updates.

    The constant has the same origin. Its first entry $\mu_0 M_0^{(2)}$ is the initial-radius entry
    $\Delta_0 M_0^{(2)}/\tau_0$ of the other three propositions written in the rule's own variable,
    and its second, $\gone(1-\eta_2)$, is $\Rdelta$'s. Both rules contract by $\gone$ on the
    branches with $\rho_k < \eta_2$, one applying the factor to $\Delta_k$ and the other to
    $\mu_k = \Delta_k/\tau_k$. Nothing in the argument depends on what $\tau$ denotes, so the rule
    is as strong at second order as at first.

%-----------------------------------------------------------------------
\subsection{Retrospective update (\texorpdfstring{$\Rrtr$}{R-RTR})}
\label{subsec:second-order retrospective}
%-----------------------------------------------------------------------

$\Rrtr$ chooses its branch by the retrospective ratio
$\widetilde{\rho}_{k+1} = \bigl(f(x_k) - f(x_{k+1})\bigr)/
\bigl(m_{k+1}(-s_k) - m_{k+1}(0)\bigr)$, so
Corollary~\ref{cor: successful iteration for low Delta (2nd order)} does not apply directly: it
bounds $\Delta_k$ from below when $\rho_k$ is small, not when $\widetilde{\rho}_{k+1}$ is. The
required substitute is the second-order analogue of \citet[Thm.~3.3]{BastMalmMoufToinToma10}, and
it needs one ingredient the first-order theory does not.

\begin{lemma}
\label{lem: retro pred difference}
    Under \ref{as: Lipschitz Hessian}, \ref{as: Approximation of the Hessian} and
    the quantitative form of \ref{as: retrospective directional},
    \begin{equation}
        \norm{\bigl(H_{k+1} - \nabla^2 f(x_{k+1})\bigr) s_k} \;\leq\; \mykappa{BH}\norm{s_k}^2 ,
        \label{eqn: ah4 quantitative}
    \end{equation}
    there is
    $\mykappa{ret} > 0$, depending only on $L_H$ and $\mykappa{BH}$, such that on a successful
    iteration
    \begin{equation}
        \Bigl| \bigl[m_k(0) - m_k(s_k)\bigr]
             - \bigl[m_{k+1}(-s_k) - m_{k+1}(0)\bigr] \Bigr|
        \;\leq\; \mykappa{ret}\,\norm{s_k}^3 .
        \label{eqn: retro pred difference}
    \end{equation}
\end{lemma}

\begin{proof}
    The two predicted reductions are $-g_k^\top s_k - \tfrac12 s_k^\top H_k s_k$ and
    $-g_{k+1}^\top s_k + \tfrac12 s_k^\top H_{k+1} s_k$, so their difference is
    $(g_{k+1} - g_k)^\top s_k - \tfrac12 s_k^\top (H_k + H_{k+1}) s_k$. By Taylor expansion and
    \ref{as: Lipschitz Hessian},
    $(g_{k+1}-g_k)^\top s_k = s_k^\top\nabla^2 f(x_k)s_k + \calO(L_H\norm{s_k}^3)$; by
    \ref{as: Approximation of the Hessian} at $k$ and~\eqref{eqn: ah4 quantitative} at
    $k+1$, together with
    \ref{as: Lipschitz Hessian} to move $\nabla^2 f(x_{k+1})$ to $\nabla^2 f(x_k)$,
    $\tfrac12 s_k^\top(H_k + H_{k+1})s_k = s_k^\top\nabla^2 f(x_k)s_k
    + \calO((L_H + \mykappa{BH})\norm{s_k}^3)$. The two leading terms cancel, leaving a remainder of
    order $\norm{s_k}^3$ with a constant determined by $L_H$ and $\mykappa{BH}$.
\end{proof}

\begin{proposition}
\label{prop: so rtr}
    Under \ref{as: Lipschitz Hessian}, \ref{as: max Cauchy/eigen decrease},
    \ref{as: Approximation of the Hessian}, {$\norm{H_k} \leq \mykappa{umh}$,
    \eqref{eqn: ah4 quantitative}}
    and \ref{as: boundary step} with
    $\mykappa{sob} = 1$, there is a non-decreasing sequence
    $\{M_k^{(2),\mathrm{r}}\}$ such that
    $\Delta_k \leq (1-\widetilde{\eta}_1)\tau_k/M_k^{(2),\mathrm{r}}$ implies
    $\widetilde{\rho}_{k+1} \geq \widetilde{\eta}_1$. Consequently $\Rrtr$ satisfies
    Condition~\ref{condition: lower bound Delta (2nd order)} with
    \begin{equation*}
        \mykappa{lbd}^{(2)} = \min\left\{
            \frac{\Delta_0\, M_0^{(2),\mathrm{r}}}{\tau_0},\;
            \gone\,{\mykappa{sls}}\min\left\{1 - \eta_1,\;
                1 - \widetilde{\eta}_1,\,
                \frac{M_0^{(2),\mathrm{r}}}{\mykappa{umh}}\right\} \right\} .
    \end{equation*}
\end{proposition}

\begin{proof}
    \emph{The threshold.} By Lemma~\ref{lem: retro pred difference} and
    \ref{as: max Cauchy/eigen decrease},
    \begin{equation*}
        m_{k+1}(-s_k) - m_{k+1}(0)
        \;\geq\; \frac{\mykappa{moc}}{2}\abs{\lambda_k^-}\Delta_k^2
                 - \mykappa{ret}\,\Delta_k^3 ,
    \end{equation*}
    which is at least half the first term once
    $\Delta_k \leq \mykappa{moc}\abs{\lambda_k^-}/(4\mykappa{ret})$. The argument of
    Theorem~\ref{thm: second-order dynamics on rho_k} with this denominator, and its first-order
    counterpart on the gradient branch, produce a threshold of the stated form; take
    $M_k^{(2),\mathrm{r}}$ to be any non-decreasing majorant of it
    and of $\widetilde{L}_k$, so that
    Corollary~\ref{cor: successful iteration for low Delta (2nd order)} may be read against it as
    well.

    \emph{The induction.} The expansion branch of $\Rrtr$ is
    $\Delta_{k+1} \geq \max\{\gthree\norm{s_k},\, \Delta_k\}$, which gives
    $\Delta_{k+1} \geq \Delta_k$ outright: branch~(a) of Lemma~\ref{lem: radius induction}, with no
    lower bound on $\norm{s_k}$ required and hence no appeal to
    Condition~\ref{cond:inexactness}. The branch $\Delta_{k+1} = \Delta_k$, fired at
    $\widetilde\rho_{k+1} \in [\widetilde\eta_1, \widetilde\eta_2)$, is branch~(a) as well.
    On the four remaining branches
    $\Delta_{k+1} \geq \gone\norm{s_k}$, and the computation of
    Proposition~\ref{prop: so rstep}, with $M_k^{(2),\mathrm{r}}$ in place
    of $M_k^{(2)}$, gives branch~(b).
    Two thresholds enter, according to what selects the branch. The two rejected-step
    branches are selected by $\rho_k < \eta_1$, and
    Corollary~\ref{cor: successful iteration for low Delta (2nd order)} applied at $\eta_1$ gives
    $\Delta_k > (1-\eta_1)\tau_k/M_k^{(2),\mathrm{r}}$; the two accepted branches with
    $\widetilde\rho_{k+1} < \widetilde\eta_1$ are selected by the retrospective ratio, and the
    threshold of the first part gives
    $\Delta_k > (1-\widetilde\eta_1)\tau_k/M_k^{(2),\mathrm{r}}$. Branch~(b) therefore holds with
    $c = \gone{\mykappa{sls}}\min\{1-\eta_1,\, 1-\widetilde{\eta}_1,\,
    M_0^{(2),\mathrm{r}}/\mykappa{umh}\}$.
    Lemma~\ref{lem: radius induction} concludes.
\end{proof}

\begin{remark}
\label{rem: rtr second order}
    $\Rrtr$ inherits the step-driven family's need for \ref{as: boundary step} and adds a
    dependence on \ref{as: Approximation of the Hessian} that $\Rdelta$ does not have. In
    exchange its expansion branch $\max\{\gthree\norm{s_k}, \Delta_k\}$ is simultaneously
    step-anchored and non-decreasing, so branch~(a) of Lemma~\ref{lem: radius induction} is
    immediate and the rule needs neither Condition~\ref{cond:inexactness} nor the corresponding
    third entry in its constant. It is the only step-driven rule with that property.
\end{remark}

\subsection{Summary of verifications} \label{sec:summary of verifications}

Table~\ref{tab: second-order summary} collects the verifications.
\emph{$\Rdelta$ is indifferent.} Its verification is the first-order one of \PartI{} with
$\norm{g_k}$ replaced by $\tau_k$, and it requires nothing new.
Section~\ref{sec:asymptotic radius} shows that this insensitivity has a price elsewhere, namely
$\liminf_k\Delta_k > 0$, so that the radius carries no asymptotic information about criticality, and
the second-order setting neither adds to that price nor reduces it.
\emph{The constants have a common shape.} By Lemma~\ref{lem: radius induction} each
$\mykappa{lbd}^{(2)}$ is the minimum of $\Delta_0 M_0^{(2)}/\tau_0$, which records the initial
radius and is common to all mechanisms, with one entry per branch on which the rule contracts.
$\Rdelta$, $\Rrtr$ and $\Rgrad(\tau)$ have one such entry, $\Rdfo(\tau)$ and $\Rstep$ two. That
$\Rdelta$ and $\Rgrad(\tau)$ arrive at the same constant is not a coincidence: both contract by
$\gone$ on the branches with $\rho_k < \eta_2$, one applying the factor to $\Delta_k$ and the other
to $\mu_k = \Delta_k/\tau_k$. These constants are not merely bookkeeping: by
Theorem~\ref{thm: Second-order worst-case complexity} the curvature term of the complexity bound
carries $(\mykappa{lbd}^{(2)})^{-2}$, so a mechanism with a smaller constant pays quadratically.

\begin{table}[tbp]
    \centering
    \caption{Verification of Condition~\ref{condition: lower bound Delta (2nd order)}, assuming
        $\tau_0 > 0$ throughout. Every entry in the third column is additionally intersected with
        the initial-radius term $\Delta_0 M_0^{(2)}/\tau_0$, which is common to all mechanisms,
        with $M_0^{(2),\mathrm{r}}$ in place of $M_0^{(2)}$ for $\Rrtr$. For $\Rgrad(\tau)$ that
        term reads $\mu_0 M_0^{(2)}$ in Proposition~\ref{prop: so rgrad}, the two being equal
        since $\Delta_0 = \mu_0\tau_0$.}
    \label{tab: second-order summary}
    \footnotesize
    \begin{tabular}{@{}l c >{\raggedright\arraybackslash}p{80mm}@{}}
        \toprule
        Mechanism & C.~\ref{condition: lower bound Delta (2nd order)}
            & $\mykappa{lbd}^{(2)}$, remaining entries \\
        \midrule

        $\Rdelta$
            & yes
            & $\gone(1-\eta_2)$ \\[2pt]

        $\Rstep$
            & yes
            & $\min\bigl\{\gone{\mykappa{sls}}
                \min\{1-\eta_2,\, M_0^{(2)}/\mykappa{umh}\},\;
                (1-\xi)M_0^{(2)}/\mykappa{umh}\bigr\}$ \\[2pt]

        $\Rrtr$
            & yes
            & $\gone{\mykappa{sls}}
                \min\{1-\eta_1,\; 1-\widetilde{\eta}_1,\,
                M_0^{(2),\mathrm{r}}/\mykappa{umh}\}$ \\

        \midrule

        $\Rdfo(\norm{g})$
            & \textbf{no}
            & --- \\[2pt]

        $\Rdfo(\tau)$
            & yes
            & $\min\{\gone(1-\eta_1),\; \gtwo\zeta M_0^{(2)}\}$ \\[2pt]

        $\Rgrad(\norm{g})$
            & \textbf{no}
            & --- \\[2pt]

        $\Rgrad(\tau)$
            & yes
            & $\gone(1-\eta_2)$, as $\Rdelta$ \\

        \bottomrule
    \end{tabular}
\end{table}

%=======================================================================

\section{Inactivity of the trust-region constraint and superlinear convergence} \label{sec:convergence rate}

    This section establishes the conditions under which the trust-region constraint stops binding, so that the method takes full (quasi-)Newton steps.
    The argument follows three questions in turn: which mechanisms reach eventual
    inactivity and what each one pays for it, whether that price can be avoided, and what rate
    follows once the constraint is gone. The last is answered without reference to the radius rule
    at all.

    Throughout this section $\{x_k\}$ converges to a first-order critical point $x^*$ with
    $\nabla^2 f(x^*)$ positive definite, the model Hessian is uniformly bounded,
    $\norm{H_k} \leq \mykappa{umh}$ (see Remark~\ref{rem: assumptions on the model Hessian} following
    Theorem~\ref{thm: series of steps converges}), and the subproblem solver satisfies
    Condition~\ref{cond:inexactness}, $\norm{H_k s_k + g_k} \leq \xi_k \norm{g_k}$ with
    $\xi_k \to 0$ when the trust-region constraint is inactive.
    Condition~\ref{cond: s_k bounded by g_k} is invoked below in the form
    $\norm{s_k} \leq \kbar\norm{g_k}$, without restating its proximity clause: since
    $x_k \to x^*$, the index $k_\kappa$ it supplies may be enlarged once and for all so that
    $\norm{x_k - x^*} \leq \delta^{(s)}$ for every $k \geq k_\kappa$, and the implication is then
    available at every such $k$.
    
    Two results from the preceding sections are used repeatedly, that
    $\sum_k \norm{s_k} < \infty$, hence $\norm{s_k} \to 0$
    (Theorem~\ref{thm: series of steps converges}) and that
    $\gamma_k \defeq \norm{g(x_k + s_k) - g_k - H_k s_k}/\norm{s_k} \to 0$
    (this is \ref{as: Finite-difference Hessian consistency}, equivalent to
    \ref{as: Hessian-step residual condition} by Step~1 of the proof of
    Theorem~\ref{thm: equivalence quasi-Newton conditions}).%
    We also use that every iteration is
    eventually very successful, $\rho_k \geq \eta_2$ for $k$ large
    (Theorem~\ref{thm: rho converges to 1}), so that $\{f(x_k)\}$ is non-increasing.

    Two further hypotheses are used only where indicated. The first is the side condition on the
    mechanism constants required by the criticality-anchored rules; the second is the secant
    equation, used only for the retrospective family in
    Section~\ref{subsec:retrospective inactivity}.

    \begin{assum}{a.Sc}[side conditions on $\zeta$ and $\overline{\mu}$]
\label{as: conditions on constants criticality related updates}
        For \Rdfo($\cdot$) and \Rgrad($\cdot$) the constants $\zeta$ and $\overline{\mu}$ are chosen so that
        \begin{equation}
            \text{(a)}\quad \zeta > \kbar
            \qquad\qquad 
            \text{(b)}\quad \overline{\mu} > \kbar .
        \end{equation}
            
    \end{assum}

    \begin{assum}{H.Sec}[secant equation]
\label{as: secant condition}
        The model Hessian satisfies the secant equation $H_{k+1} s_k = g_{k+1} - g_k$ on accepted
        steps. This holds by construction for BFGS and for SR1.
    \end{assum}

\subsection{Elements independent of the update mechanism}
\label{subsec: mechanism-free elements}
\label{subsec:elements independent of the update mechanism}
%=======================================================================

We first separate what the argument owes to the radius rule from what it does not. The
mechanism-dependence turns out to localise entirely into two inequalities, stated as
Properties~\ref{prop: no contraction while active} and~\ref{prop: radius dominates step} below;
everything else in this section holds for any rule.

\subsubsection{An explicit inactivity criterion}

The following lemma is used repeatedly below.

\begin{lemma}[Inactivity criterion]
\label{lem: inactivity criterion}
Let $k \geq k_\kappa$ , the index of Condition~\ref{cond: s_k bounded by g_k}, so that
$\norm{s_k} \leq \kbar\norm{g_k}$ by
Condition~\ref{cond: s_k bounded by g_k}. If
\begin{equation}
  \Delta_k > \kbar \norm{g_k},
  \label{eqn: inactivity criterion}
\end{equation}
then $\norm{s_k} < \Delta_k$: the trust-region constraint is inactive at iteration $k$.
Equivalently, if the constraint is active at $k$ then $\Delta_k \leq \kbar\norm{g_k}$.

If moreover $g_k \neq 0$, then writing
\begin{equation}
  \theta_k \defeq \frac{\Delta_k}{\norm{g_k}}
  \label{eqn: theta}
\end{equation}
for the \emph{radius-to-criticality ratio}, the constraint is inactive whenever
$\theta_k > \kbar$, and $\theta_k \leq \kbar$ whenever it is active.
\end{lemma}

\begin{proof}
If~\eqref{eqn: inactivity criterion} holds then
$\norm{s_k} \leq \kbar\norm{g_k} < \Delta_k$. Conversely, if the constraint is active then
$\norm{s_k} = \Delta_k$, so $\Delta_k = \norm{s_k} \leq \kbar\norm{g_k}$. When $g_k \neq 0$,
dividing by $\norm{g_k}$ turns each statement into the corresponding one for $\theta_k$.
\end{proof}

The criterion is sufficient for inactivity and not necessary: the constraint may well be inactive
with $\theta_k \leq \kbar$. Sufficiency is all the argument needs, and since the lemma rests on
Condition~\ref{cond: s_k bounded by g_k} alone it reduces the question, for every mechanism
alike, to whether the single scalar sequence $\{\theta_k\}$ eventually exceeds $\kbar$. The
unnormalised form~\eqref{eqn: inactivity criterion} is the primitive one, being insensitive to
$g_k = 0$; the ratio~\eqref{eqn: theta} is the form in which the mechanisms can be compared.

\subsubsection{Decay of the steps once inactive}

The proposition on strictly decreasing steps is mechanism-independent: its proof invokes the
inexactness condition, the acceptance test, Condition~\ref{cond: s_k bounded by g_k} and
\ref{as: Gradient residual condition}, and at no point the radius rule.

\begin{proposition}[Steps contract once the region is inactive, after~{\citep[Thm.~4.2]{Stei83}}]
\label{prop: steps are strictly decreasing}
Let $k \geq k_\kappa$, so that Condition~\ref{cond: s_k bounded by g_k} is in force at both $k$
and $k+1$, and suppose $s_k \neq 0$,
\begin{equation*}
    \norm{s_k} < \Delta_k , \qquad \rho_k \geq \eta ,
\end{equation*}
and
\begin{equation}
  \kbar\bigl(\gamma_k + \xi_k \mykappa{umh}\bigr) < \gone\bigl(1 - \xi_k\bigr).
  \label{eqn: prop hypothesis}
\end{equation}
Then
\begin{equation}
  \norm{s_{k+1}} < \gone \norm{s_k}.
  \label{eqn: step contraction}
\end{equation}
The constant $\gone$ is borrowed from the radius rule only so that the induction of
Section~\ref{subsec:common inactivity condition} lines up: the same argument gives
$\norm{s_{k+1}} < c\,\norm{s_k}$ for any fixed $c \in (0,1)$, as soon as $k$ is large enough that
$\kbar(\gamma_k + \xi_k\mykappa{umh}) < c\,(1-\xi_k)$, which holds eventually because
$\gamma_k, \xi_k \to 0$.
\end{proposition}

\begin{proof}
The case $s_k = 0$ is excluded for a reason: the inexactness condition would then read
$\norm{g_k} \leq \xi_k\norm{g_k}$ with $\xi_k < 1$, forcing $g_k = 0$, so that $x_k$ is critical
and the method terminates. For $s_k \neq 0$ the ratio $\gamma_k$ is well defined.

\emph{Step 1: $\norm{g_{k+1}} \leq \bigl[(\gamma_k + \xi_k\norm{H_k})/(1-\xi_k)\bigr]\norm{s_k}$.}
Since $\rho_k \geq \eta$ the step is accepted and $x_{k+1} = x_k + s_k$. Because the trust region
is inactive at $k$, Condition~\ref{cond:inexactness} applies. Following~\citet[Thm.~4.2]{Stei83},
split the new gradient into the model's linearisation error and the subproblem residual,
\begin{equation}
  g_{k+1} = \underbrace{\bigl[g(x_k + s_k) - g_k - H_k s_k\bigr]}_{\text{model error}}
          + \underbrace{\bigl[H_k s_k + g_k\bigr]}_{\text{residual}} ,
  \label{eqn: gradient split}
\end{equation}
and, $\norm{g_{k+1}} \leq \gamma_k \norm{s_k} + \xi_k \norm{g_k}$.
Then, using $g_k = (H_k s_k + g_k) - H_k s_k$ with $\gamma_k \leq \gamma_k/(1-\xi_k)$ and the
standing bound $\norm{H_k} \leq \mykappa{umh}$,
\begin{equation}
  \frac{\norm{g_{k+1}}}{\norm{s_k}}
    \;\leq\; \gamma_k + \frac{\xi_k \norm{H_k}}{1 - \xi_k}
    \;\leq\; \frac{\gamma_k + \xi_k \norm{H_k}}{1 - \xi_k}
    \;\leq\; \frac{\gamma_k + \xi_k \mykappa{umh}}{1 - \xi_k} .
  \label{eqn: gk1 final}
\end{equation}

\emph{Step 2: conversion.} Applying Condition~\ref{cond: s_k bounded by g_k} at $k+1$ and
then~\eqref{eqn: gk1 final} along with the hypothesis~\eqref{eqn: prop hypothesis} gives the conclusion~\eqref{eqn: step contraction}, $\norm{s_{k+1}} \leq \kbar\norm{g_{k+1}} < \gone\norm{s_k}$.
\end{proof}

%=======================================================================
\subsection{A common condition for \texorpdfstring{$\Rdelta$}{R-delta},
         \texorpdfstring{$\Rstep$}{R-step} and \texorpdfstring{$\Rdfo$}{R-DFO}}
\label{subsec:common inactivity condition}
%=======================================================================

The three non-gradient-scaled mechanisms share a single pair of conditions under which the inactivity theorem holds.

\subsubsection{The two properties}

We introduce two properties of the mechanism that are sufficient to guarantee that the trust-region constraint eventually becomes inactive.
Property~\ref{prop: no contraction while active} excludes a permanently binding constraint, for on
such a run $\Delta_k = \norm{s_k}$ while~\eqref{eqn: I1} makes the radius non-decreasing, so the
step cannot vanish as Theorem~\ref{thm: series of steps converges} requires.

\begin{property}[No contraction while the region binds]
\label{prop: no contraction while active}
    There is $k_1$ such that, for all $k \geq k_1$ with $\rho_k \geq \eta_2$ and
    $\norm{s_k} = \Delta_k$,
    \begin{equation}
        \Delta_{k+1} \geq \Delta_k .
        \label{eqn: I1}
    \end{equation}
\end{property}
The next property propagates inactivity,
Proposition~\ref{prop: steps are strictly decreasing} having contracted the step by $\gone$ across
an inactive iteration and~\eqref{eqn: I2} keeping the next radius above the contracted step.
\begin{property}[The next radius dominates the current step]
\label{prop: radius dominates step}
    There is $k_2$ such that, for all $k \geq k_2$ with $\rho_k \geq \eta_2$,
    \begin{equation}
        \Delta_{k+1} \geq \gone \norm{s_k} .
        \label{eqn: I2}
    \end{equation}
\end{property}
Neither implies the other, \eqref{eqn: I1} relating $\Delta_{k+1}$ to $\Delta_k$ on active
iterations and~\eqref{eqn: I2} relating it to $\norm{s_k}$ on all of them, where $\norm{s_k}$ may
be far below $\Delta_k$. They are the only demands made on the mechanism and the two properties enter the proof of Theorem~\ref{thm: trust-region radius becomes inactive} independently.

\begin{theorem}[Inactivity from the two properties]
\label{thm: trust-region radius becomes inactive}
    Under the standing assumptions of this section, suppose the mechanism satisfies
    Properties~\ref{prop: no contraction while active} and~\ref{prop: radius dominates step}. Then
    there is $k_\star$ with $\norm{s_k} < \Delta_k$ for all $k \geq k_\star$.
\end{theorem}

\begin{proof}
    The argument has two parts, one for each property.

    \emph{Step~1. The region is inactive at infinitely many iterations.} Suppose otherwise, so
    that there is $\ell \geq \max\{k_\rho, k_\kappa, k_1\}$ with $\norm{s_k} = \Delta_k$ for all
    $k \geq \ell$. By Theorem~\ref{thm: rho converges to 1} every such iteration has
    $\rho_k \geq \eta_2$, so Property~\ref{prop: no contraction while active} applies throughout
    and $\{\Delta_k\}_{k\geq\ell}$ is non-decreasing, giving $\Delta_k \geq \Delta_\ell > 0$. But
    $\Delta_k = \norm{s_k} \to 0$ by Theorem~\ref{thm: series of steps converges}, a
    contradiction.

    \emph{Step~2. Inactivity persists.} Since $\gamma_k \to 0$
    by \ref{as: Finite-difference Hessian consistency}, equivalent to
    \ref{as: Hessian-step residual condition} by Step~1 of the proof of
    Theorem~\ref{thm: equivalence quasi-Newton conditions}, and $\xi_k \to 0$ by
    Condition~\ref{cond:inexactness}, fix
    $k_3 \geq \max\{k_\rho, k_\kappa, k_2\}$ beyond which~\eqref{eqn: prop hypothesis} holds, and
    use Step~1 to choose $k_\star \geq k_3$ at which the region is inactive. We show by induction
    that it is inactive at every $k \geq k_\star$.

    Let $\norm{s_k} < \Delta_k$ for some $k \geq k_\star$. Every hypothesis of
    Proposition~\ref{prop: steps are strictly decreasing} is then met, so
    $\norm{s_{k+1}} < \gone\norm{s_k}$, while
    Property~\ref{prop: radius dominates step} gives $\Delta_{k+1} \geq \gone\norm{s_k}$.
    Combining the two,
    \begin{equation*}
        \norm{s_{k+1}} \;<\; \gone \norm{s_k} \;\leq\; \Delta_{k+1} ,
    \end{equation*}
    so the region is inactive at $k+1$ as well.
\end{proof}

\begin{remark}[The rate theorem does not see the mechanism]
\label{rem: inactivity mechanism free}
    Once Theorem~\ref{thm: trust-region radius becomes inactive} holds, the step is the
    unconstrained model minimiser at every large iteration and the rate is that of the underlying
    quasi-Newton iteration. All the mechanism can decide is when that regime is entered, and
    whether it is entered at all.
\end{remark}

\subsubsection{The retrospective ratio converges to one}

$\Rrtr$, recalled in Appendix~\ref{app: R-RTR}, branches on $\widetilde\rho_{k+1}$ rather than
on $\rho_k$, so the goal is $\widetilde\rho_{k+1} \to 1$. The ratio attached to the update of
$\Delta_{k+1}$ is
\begin{equation}
    \widetilde\rho_{k+1}
      \;=\; \frac{f(x_k) - f(x_{k+1})}{m_{k+1}(-s_k) - m_{k+1}(0)} .
    \label{eqn: retrospective ratio shifted}
\end{equation}
\citet[Lem.~3.8]{BastMalmMoufToinToma10} argue from asymptotic second-order coherence,
$\norm{\nabla^2 f(x_k) - \nabla^2 m_k(x_k)} \to 0$ whenever $\norm{g_k} \to 0$, together with a
quadratic model decrease and $\norm{s_k} \to 0$. \citet[Thm.~3.1]{FanPan16} argue instead from an
exactly solved subproblem, so that $(H_k + \lambda_k I)s_k + g_k = 0$ for some $\lambda_k \geq 0$,
and read the required comparison off the multiplier.

We follow the first route, which survives a truncated conjugate-gradient solver, with \ref{as: retrospective directional} below in place of coherence. 
\begin{assum}{H.Rd}[retrospective directional consistency]
\label{as: retrospective directional}
    The model Hessian at $k+1$ is accurate along the step just taken,
    \begin{equation}
        \lim_{k\to\infty}
        \frac{\norm{\bigl(H_{k+1} - \nabla^2 f(x_{k+1})\bigr) s_k}}{\norm{s_k}} = 0
        \qquad\text{whenever}\qquad \lim_{k\to\infty}\norm{g_k} = 0 .
        \label{eqn: retrospective directional}
    \end{equation}
\end{assum}
The substitution is possible because the quantity governing $\widetilde\rho_{k+1}$ is
$f(x_k) - m_{k+1}(-s_k)$, a quadratic form in $s_k$, so only the action of $H_{k+1}$ on $s_k$ ever
enters and the operator norm is never needed. The weaker hypothesis is what the retrospective setting needs.
Writing
$y_k \defeq g_{k+1} - g_k = \int_0^1 \nabla^2 f(x_k + t s_k)\,s_k\,dt$, any secant update therefore
satisfies \ref{as: retrospective directional} with an explicit rate, since
\ref{as: Lipschitz Hessian} gives
\begin{equation}
    \norm{\bigl(H_{k+1} - \nabla^2 f(x_{k+1})\bigr)s_k}
      = \norm{\int_0^1 \bigl[\nabla^2 f(x_k + t s_k) - \nabla^2 f(x_{k+1})\bigr] s_k\, dt}
      \;\leq\; \frac{L_H}{2}\,\norm{s_k}^2 .
    \label{eqn: secant gives ah4}
\end{equation}
\begin{remark}[Which quasi-Newton updates qualify]
\label{rem: ah4 quasi newton}
    Three qualifications bear on where each update may be used.

    First, SR1 satisfies the secant equation only while
    $(y_k - H_k s_k)^\top s_k \neq 0$, and the usual safeguard skips the update when that quantity
    is small relative to $\norm{s_k}\norm{y_k - H_k s_k}$; on a skipped iteration $H_{k+1} = H_k$
    and~\eqref{eqn: secant gives ah4} is unavailable. Since
    \ref{as: retrospective directional} is asymptotic, finitely many skips are harmless.
    Infinitely many are not, unless supplemented by a condition such as the uniform linear
    independence of~\citet{ConnGouldToint1991, KhalfanByrdSchnabel93}, under which
    $\norm{H_k - \nabla^2 f(x^*)} \to 0$ and the hypothesis holds outright.

    Second, BFGS and DFP require the curvature condition $y_k^\top s_k > 0$. Near a minimiser with
    $\nabla^2 f(x^*)$ positive definite this holds for all large $k$, which is the regime of
    Section~\ref{sec:convergence rate}; near a saddle it need not, and the update must be skipped
    or damped. Powell damping replaces $y_k$ by $\theta y_k + (1-\theta)H_k s_k$ and so satisfies
    the secant equation only when $\theta = 1$, which again holds for all large $k$ near a strong
    minimiser.

    Third, BFGS and DFP generate positive definite $H_{k+1}$ by construction and therefore cannot
    approximate an indefinite Hessian. They are consequently unavailable where
    \ref{as: Approximation of the Hessian} is required with $\lambda_k^- < 0$, as in
    Proposition~\ref{prop: so rtr}. Of the standard secant updates only SR1, which admits
    indefinite iterates, serves in the second-order setting.
\end{remark}
The exact Hessian $H_{k+1} = \nabla^2 f(x_{k+1})$ satisfies
\ref{as: retrospective directional} trivially, although it does not satisfy the secant
equation. The hypothesis therefore covers exact-Hessian and secant methods alike, which neither
\ref{as: secant condition} nor operator-norm coherence does on its own. 

The following lemma is the analogue for $\widetilde\rho_{k+1}$ of
\citet[Lem.~3.8]{BastMalmMoufToinToma10} under \ref{as: retrospective directional}, which
by~\eqref{eqn: secant gives ah4} brings the secant methods inside the result.

\begin{lemma}[after {\citet[Lem.~3.8]{BastMalmMoufToinToma10}}]
\label{lem: rho tilde A}
    Let $f \in C^2$ satisfy \ref{as: Lipschitz Hessian}, let the model satisfy
    \ref{as: model agreement with function}, and suppose the step is accepted at every
    sufficiently large $k$, so that $x_{k+1} = x_k + s_k$ and $\widetilde\rho_{k+1}$ is well
    defined by~\eqref{eqn: retrospective ratio shifted}. Under the standing hypotheses of
    Section~\ref{sec:convergence rate} the last requirement holds from $k_\rho$ on, by
    Theorem~\ref{thm: rho converges to 1}. Assume moreover
    \begin{enumerate}[label=(\alph*),leftmargin=2.2em]
      \item \ref{as: Hessian-step residual condition} and
            \ref{as: retrospective directional}: the model Hessians at $k$ and at $k+1$ are
            both accurate \emph{along $s_k$};
      \item \emph{quadratic model decrease}: there is $\mykappa{mqd} > 0$ with
            \begin{equation}
              m_k(0) - m_k(s_k) \;\geq\; \mykappa{mqd}\norm{s_k}^2 \;>\; 0
              \label{eqn: mqd}
            \end{equation}
            for all large $k$;
      \item $\norm{s_k} \to 0$.
    \end{enumerate}
    Then $\widetilde\rho_{k+1} \to 1$; in particular $\widetilde\rho_{k+1} \geq \widetilde\eta_2$
    for all large $k$.
\end{lemma}

\begin{proof}
    Every index below is large enough that the step is accepted, so that
    $x_{k+1} - x_k = s_k$ and the segment $[x_k, x_{k+1}]$ has length
    $\norm{s_k}$. 
    Write $\delta_k \defeq \norm{(H_k - \nabla^2 f(x_k))s_k}/\norm{s_k}$ and
    $\delta_k^{+} \defeq \norm{(H_{k+1} - \nabla^2 f(x_{k+1}))s_k}/\norm{s_k}$, both $o(1)$ by~(a),
    and note that \ref{as: Lipschitz Hessian} with $\norm{x_{k+1} - x_k} = \norm{s_k}$ gives
    \begin{equation}
        \norm{\nabla^2 f(y) - \nabla^2 f(y^\prime)} \;\leq\; L_H\norm{s_k}
        \qquad \text{for every } y, \, y^\prime \in [x_k, x_{k+1}] .
        \label{eqn: hessian along segment}
    \end{equation}

    \begin{equation}
        \widetilde\rho_{k+1} - 1
          = \frac{f(x_k) - m_{k+1}(-s_k)}{m_{k+1}(-s_k) - m_{k+1}(0)} .
        \label{eqn: rho tilde minus one}
    \end{equation}

    \emph{Step 1: the numerator is $o(\norm{s_k}^2)$.} Since $m_{k+1}(0) = f(x_{k+1})$, expanding $f$ about $x_{k+1}$ with the mean-value form of the remainder, there is $z_k \in [x_k, x_{k+1}]$ with
    $f(x_k) = f(x_{k+1}) - g_{k+1}^\top s_k + \tfrac12 s_k^\top\nabla^2 f(z_k)s_k$, while
    $m_{k+1}(-s_k) = f(x_{k+1}) - g_{k+1}^\top s_k + \tfrac12 s_k^\top H_{k+1}s_k$. Subtracting,
    \begin{equation}
        f(x_k) - m_{k+1}(-s_k)
          = \tfrac12\, s_k^\top\bigl(\nabla^2 f(z_k) - H_{k+1}\bigr) s_k ,
        \label{eqn: numerator identity}
    \end{equation}
    By Cauchy--Schwarz, \eqref{eqn: hessian along segment} and the definition of $\delta_k^{+}$,
    {
        \small
        \begin{equation}
        \abs{f(x_k) - m_{k+1}(-s_k)} 
            \;\leq\; \tfrac12\norm{s_k}\Bigl[
            \norm{\bigl(\nabla^2 f(z_k) - \nabla^2 f(x_{k+1})\bigr)s_k}
            + \norm{\bigl(\nabla^2 f(x_{k+1}) - H_{k+1}\bigr)s_k} \Bigr] 
            \;\leq\; \tfrac12\bigl(L_H\norm{s_k} + \delta_k^{+}\bigr)\norm{s_k}^2 .
    \end{equation}
    }

    \emph{Step 2: the denominator is at least $\tfrac12\mykappa{mqd}\norm{s_k}^2$.} By arguments similar to those of Lemma~\ref{lem: retro pred difference},
    \begin{align*}
        \bigl[m_{k+1}(-s_k) - m_{k+1}(0)\bigr] - \bigl[m_k(0) - m_k(s_k)\bigr] & = -(g_{k+1} - g_k)^\top s_k + \tfrac12 s_k^\top (H_k + H_{k+1}) s_k \\
        & = \calO\bigl((\delta_k + \delta_k^{+} + L_H\norm{s_k})\norm{s_k}^2\bigr)
          = o(\norm{s_k}^2).
    \end{align*}
    Here \ref{as: Lipschitz Hessian} gives $(g_{k+1}-g_k)^\top s_k = s_k^\top\nabla^2 f(x_k)s_k + \calO\bigl(L_H\norm{s_k}^3\bigr)$, \ref{as: Hessian-step residual condition} gives $s_k^\top H_k s_k = s_k^\top\nabla^2 f(x_k)s_k + \calO\bigl(\delta_k\norm{s_k}^2\bigr)$, and \ref{as: retrospective directional} with \ref{as: Lipschitz Hessian} gives
    $s_k^\top H_{k+1} s_k = s_k^\top\nabla^2 f(x_{k})s_k + \calO\bigl((\delta_k^{+} + L_H\norm{s_k})\norm{s_k}^2\bigr)$.
    Therefore, 
    \begin{equation*}
        m_{k+1}(-s_k) - m_{k+1}(0)
          \;\geq\; \mykappa{mqd}\norm{s_k}^2 - o(\norm{s_k}^2)
          \;\geq\; \tfrac12\,\mykappa{mqd}\,\norm{s_k}^2
    \end{equation*}
    for all large $k$; in particular the denominator of~\eqref{eqn: rho tilde minus one} is
    positive and $\widetilde\rho_{k+1}$ is well defined.

    \emph{Step 3.} Combining the two steps,
    \begin{equation*}
        \abs{\widetilde\rho_{k+1} - 1}
          \;\leq\; \frac{\tfrac12\bigl(L_H\norm{s_k} + \delta_k^{+}\bigr)\norm{s_k}^2}
                        {\tfrac12\,\mykappa{mqd}\,\norm{s_k}^2}
          \;=\; \frac{L_H\norm{s_k} + \delta_k^{+}}{\mykappa{mqd}}
          \;\xrightarrow[k\to\infty]{}\; 0 ,
    \end{equation*}
    by~(c) and \ref{as: retrospective directional}. The final assertions follow since
    $\widetilde\eta_2 < 1$.
\end{proof}

\begin{remark}[On the quadratic model decrease]
\label{rem: mqd pointer}
    Hypothesis~\eqref{eqn: mqd} is a condition on the subproblem solver rather than on the radius
    rule, and it holds for the solvers in common use, for example the truncated conjugate-gradient method~\citep{Stei83,Toint81}.
    Ongoing work treats the question, the constants involved being properties of the solver and of the model rather than of the mechanisms surveyed here.
\end{remark}

\subsubsection{Verification for each mechanism}

\begin{proposition}
    \label{prop: verification}
    For $k$ beyond the thresholds of Theorem~\ref{thm: rho converges to 1} and
    Condition~\ref{cond: s_k bounded by g_k}, that is for $k \geq \max\{k_\rho, k_\kappa\}$:
        \begin{enumerate}[label=(\roman*),leftmargin=2.2em]
            \item $\Rdelta$ satisfies Properties~\ref{prop: no contraction while active}
                    and~\ref{prop: radius dominates step} unconditionally;
            \item $\Rstep$ satisfies Properties~\ref{prop: no contraction while active}
                    and~\ref{prop: radius dominates step} unconditionally;
            \item $\Rdfo$ satisfies Property~\ref{prop: radius dominates step} unconditionally, and
                    Property~\ref{prop: no contraction while active} if the side condition $\zeta > \kbar$
                    of \ref{as: conditions on constants criticality related updates}(a) holds.
            \item $\Rgrad$ satisfies neither in general.
        \end{enumerate}
    The side condition in~(iii) is not an artefact of the argument:
    Theorem~\ref{thm: sinc rdfo} exhibits a function and a run on which, below the threshold, the
    constraint is active at every sufficiently large iteration.
\end{proposition}

\begin{proof}
    Throughout, $\rho_k \geq \eta_2$ by Theorem~\ref{thm: rho converges to 1}, and
    $\norm{s_k} \leq \Delta_k$.

    \emph{(i) $\Rdelta$.} The very successful branch gives $\Delta_{k+1} \geq \Delta_k$, which
    is~\eqref{eqn: I1}; and then $\Delta_{k+1} \geq \Delta_k \geq \norm{s_k} \geq \gone\norm{s_k}$,
    which is~\eqref{eqn: I2}.

    \emph{(ii) $\Rstep$.} The very successful branch gives $\Delta_{k+1} \geq \norm{s_k}$. If the
    region is active then $\norm{s_k} = \Delta_k$, so $\Delta_{k+1} \geq \Delta_k$, which
    is~\eqref{eqn: I1}; and $\Delta_{k+1} \geq \norm{s_k} \geq \gone\norm{s_k}$ on every iteration,
    which is~\eqref{eqn: I2}. Neither uses the value of the expansion factor.

    \emph{(iii) $\Rdfo$.} For~\eqref{eqn: I2}: with $\rho_k \geq \eta_2 \geq \eta_1$ the unsuccessful
    branch is excluded, so $\Delta_{k+1} \geq \gtwo\Delta_k \geq \gtwo\norm{s_k} \geq \gone\norm{s_k}$.
    For~\eqref{eqn: I1}: the radius decreases only on the middle branch, which requires
    $\Delta_k > \zeta\norm{g_k}$. If the region is active then
    $\Delta_k = \norm{s_k} \leq \kbar\norm{g_k}$ by Condition~\ref{cond: s_k bounded by g_k}, so
    $\zeta > \kbar$ gives $\Delta_k \leq \kbar\norm{g_k} < \zeta\norm{g_k}$: the middle branch cannot
    fire, the expansion branch does, and $\Delta_{k+1} \geq \Delta_k$.

    \emph{(iv) $\Rgrad$.} Here $\Delta_k = \mu_k\norm{g_k}$, so
    \begin{equation}
    \frac{\Delta_{k+1}}{\Delta_k}
        = \frac{\mu_{k+1}}{\mu_k}\cdot\frac{\norm{g_{k+1}}}{\norm{g_k}} .
    \label{eqn: rgrad ratio}
    \end{equation}
    Suppose that on active very successful iteration the rule gives $\mu_{k+1} = \gthree\mu_k$, and that $\{x_k\}$ is eventually a superlinearly convergent sequence. 
    Therefore, the product is eventually below $1$ and~\eqref{eqn: I1} fails. The same factor defeats~\eqref{eqn: I2}, which would require $\mu_{k+1} \, \norm{g_{k+1}} \geq \gone\norm{s_k}$ with no available lower bound on $\norm{g_{k+1}}$ in terms of $\norm{s_k}$.
\end{proof}

\begin{proposition}[$\Rrtr$ supplies both properties]
\label{prop: rtr verification}
    Assume the hypotheses of Theorem~\ref{thm: rho converges to 1} and of
    Lemma~\ref{lem: rho tilde A}, and let $\{\Delta_k\}$ be generated by $\Rrtr$
    (Update~\ref{update:retrospective}). Then, for all $k$ beyond the thresholds of
    Theorem~\ref{thm: rho converges to 1} and Condition~\ref{cond: s_k bounded by g_k}:
    \begin{enumerate}[label=(\roman*),leftmargin=2.2em]
        \item Property~\ref{prop: radius dominates step} holds on every branch, with no hypothesis
              on $\widetilde\rho_{k+1}$;
        \item at an iteration at which the trust region is active,
              Property~\ref{prop: no contraction while active} holds if and only if
              $\widetilde\rho_{k+1} \geq \widetilde\eta_1$, and
              Lemma~\ref{lem: rho tilde A} makes this the case for all large $k$.
    \end{enumerate}
    Theorem~\ref{thm: trust-region radius becomes inactive} therefore applies and the trust-region
    constraint is eventually permanently inactive. No user parameter is constrained; the price is
    \ref{as: retrospective directional} and the quadratic model decrease~\eqref{eqn: mqd}, both
    carried by Lemma~\ref{lem: rho tilde A}.
\end{proposition}

\begin{proof}
    Let $k \geq \max\{k_\rho, k_\kappa\}$. By Theorem~\ref{thm: rho converges to 1} we have
    $\rho_k \geq \eta_2 \geq \eta_1$, so the step is accepted and $\Delta_{k+1}$ is set by the
    accepted-step case of Update~\ref{update:retrospective}.
    Throughout we use $\norm{s_k} \leq \Delta_k$ and $0 < \gone \leq \gtwo < 1 < \gthree$.

    \emph{(i).} On the first branch $\Delta_{k+1} \geq \gthree\norm{s_k} \geq \gone\norm{s_k}$; on
    the second $\Delta_{k+1} = \Delta_k \geq \norm{s_k} \geq \gone\norm{s_k}$; on the third
    $\Delta_{k+1} = \gtwo\norm{s_k} \geq \gone\norm{s_k}$. On the fourth both members of the
    minimum are bounded below by $\gone\norm{s_k}$: the first because $\gtwo \geq \gone$, the
    second because
    $\max[\gone,\widetilde\vartheta_{k+1}]\,\Delta_k \geq \gone\Delta_k \geq \gone\norm{s_k}$. Hence
    \eqref{eqn: I2} holds on all four, whatever the value of $\widetilde\rho_{k+1}$.

    \emph{(ii).} The first branch gives $\Delta_{k+1} \geq \Delta_k$ by the maximum and the second
    gives $\Delta_{k+1} = \Delta_k$, so \eqref{eqn: I1} holds on both, whether or not the region is
    active. Suppose now that the region is active at $k$, so that $\norm{s_k} = \Delta_k$. The
    third branch then gives $\Delta_{k+1} = \gtwo\Delta_k < \Delta_k$ and the fourth
    $\Delta_{k+1} \leq \gtwo\Delta_k < \Delta_k$, so \eqref{eqn: I1} fails on both. Since the
    branch is selected by $\widetilde\rho_{k+1}$ alone, \eqref{eqn: I1} holds at an active $k$
    exactly when $\widetilde\rho_{k+1} \geq \widetilde\eta_1$. By Lemma~\ref{lem: rho tilde A},
    $\widetilde\rho_{k+1} \geq \widetilde\eta_2 \geq \widetilde\eta_1$ for all large $k$, which
    supplies the threshold $k_1$ of Property~\ref{prop: no contraction while active}.
\end{proof}

% \subsubsection{A single condition, if one insists}

% Properties~\ref{prop: no contraction while active} and~\ref{prop: radius dominates step} can be packaged into one statement, at the cost of
% some transparency. In terms of the ratio $\theta_k$ of~\eqref{eqn: theta}, both say that the rule
% must not let $\theta$ stall below $\kbar$ while the region binds. The following is equivalent to
% the pair for the three mechanisms considered, and is stated for completeness rather than for use.

% \begin{cond}{C.Ia}[inactivity-admissible mechanism]
% \label{cond: inactivity admissible}
% 
% A mechanism is \emph{inactivity-admissible} if there exist $\bar{k}$ and $c \in (0,1]$ such that,
% for all $k \geq \bar{k}$ with $\rho_k \geq \eta_2$,
% \begin{equation}
%   \Delta_{k+1} \geq \max\bigl\{\,\mathbb{1}[\,\norm{s_k} = \Delta_k\,]\,\Delta_k,\;
%                                c\,\norm{s_k} \,\bigr\}.
%   \label{eqn: admissible}
% \end{equation}
% \end{cond}

% We record~\eqref{eqn: admissible} for completeness but do not use it. The indicator makes it
% awkward to verify, and the two halves are used at different points and for different reasons;
% keeping them separate is both easier to check and more informative about what each mechanism
% does and does not provide.

%=======================================================================
\subsection{\texorpdfstring{$\Rgrad$}{R-grad}: a mechanism that controls the ratio directly}
\label{subsec:second-order gradient-scaled inactivity}
%=======================================================================

Section~\ref{sec:second-order mechanism survey} analysed $\Rgrad(\tau)$, whose radius is
$\Delta_k = \mu_k\tau_k$ with $\tau_k = \max\{\norm{g_k}, -\lambda_k^-\}$.
The standing assumptions place $x_k \to x^*$ with $\nabla^2 f(x^*)$ positive
definite, and as soon as $\lambda_{\min}(H_k) > 0$ one has $-\lambda_k^- < 0$, hence
$\tau_k = \norm{g_k}$, so $\Rgrad(\tau)$ and $\Rgrad(\norm{g})$ generate the same iterates. The
requirement costs nothing in the cases of interest, holding for all large $k$ whenever
$\norm{H_k - \nabla^2 f(x_k)} \to 0$, and at every iteration for BFGS and DFP, whose updates are
positive definite by construction. We therefore write $\Rgrad$ and take
$\Delta_k = \mu_k\norm{g_k}$ throughout this section.

The argument for $\Rgrad$ takes a different route. Rather than showing that $\mu_k$ reaches a
prescribed threshold, it shows that $\mu_k$ cannot grow indefinitely, and obtains inactivity from
the failure of the growth mechanism.

    \subsubsection{Uncapped \texorpdfstring{$\Rgrad(\omega)$}{R-grad}: the geometry bounds the multiplier} 

        The ceiling is due to \citet{YuanFan01}, whose Theorem~4.1 derives it by
        contradiction from the local geometry, and the explicit threshold is that
        of~\citet{FanPan16}, who carry the same argument over to the retrospective ratio. In both
        cases $\mu_k$ is governed by the rule under study and the bound serves a single theorem.
        The argument needs only the local geometry, so we state it for an arbitrary sequence
        $\{\mu_k\}$, with no rule assumed. That is all the extraction adds, and it is what allows
        one lemma to serve the uncapped rule, the capped rule, the quantitative ceiling of
        Corollary~\ref{cor: mu ceiling} and the retrospective variant, in place of four separate
        arguments; the constant $C$ below is theirs.

    \begin{lemma}[The local geometry bounds the multiplier, after {\citep{YuanFan01,FanPan16}}]
    \label{lem: mu at expansion}
        Let $\{x_k\} \to x^*$ with $\nabla^2 f(x^*) \succ 0$ and assume the hypotheses of
        Theorem~\ref{thm: rho converges to 1}. Suppose the radius has the form
        $\Delta_k = \mu_k \norm{g_k}$ for some sequence $\{\mu_k\}$ of positive numbers, and put
        \begin{equation}
            C \;\defeq\; \frac{2}{m}\left(1 + \sqrt{M/m}\right) .
            \label{eqn: mu ceiling}
        \end{equation}
        Then at every $k \geq k_\rho$ with $x_k$, $x_{k+1}$ in the neighbourhood $B(x^*,\Delta_*)$ of
        Proposition~\ref{prop: inequalities for positive definite Hessian} and
        $\norm{s_k} > \tfrac12\Delta_k$, one has
        \begin{equation}
            \mu_k \;\leq\; C .
            \label{eqn: mu ceiling bound}
        \end{equation}
        No rule for $\{\mu_k\}$ is assumed. The conclusion refers to no acceptance threshold, no
        expansion threshold and no ratio, and therefore holds for every mechanism of the form
        $\Delta_k = \mu_k\norm{g_k}$, whatever drives the update of $\mu$.
    \end{lemma}

\begin{proof}
    We may assume $x_k \neq x^*$, since $x_k = x^*$ gives $g_k = 0$ and the method terminates.
    Since $M \geq m$ we have $C \geq 4/m > 2/m$, so if $\mu_k \leq 2/m$ there is nothing to prove;
    assume therefore $\mu_k > 2/m$, that is $\tfrac12\mu_k m > 1$.

    By Theorem~\ref{thm: rho converges to 1} we have $\rho_k \geq \eta_2 > 0$ for $k \geq k_\rho$,
    so the step is accepted, $x_{k+1} = x_k + s_k$ and $f(x_{k+1}) \leq f(x_k)$. By
    Proposition~\ref{prop: inequalities for positive definite Hessian}(iii),
    $\norm{g_k} \geq m \norm{x_k - x^*}$, whence
    \begin{equation}
        \norm{s_k} \;>\; \tfrac12 \Delta_k \;=\; \tfrac12 \mu_k \norm{g_k}
        \;\geq\; \tfrac12 \mu_k m \norm{x_k - x^*} \;>\; \norm{x_k - x^*} \;>\; 0 ;
        \label{eqn: sk dominates}
    \end{equation}
    in particular every quantity appearing in the chain below is positive, so each squaring
    preserves its inequality. Using the reverse triangle inequality in the form
    $\norm{x_{k+1} - x^*} \geq \norm{s_k} - \norm{x_k - x^*}$ together with
    Proposition~\ref{prop: inequalities for positive definite Hessian}(iv),
    \begin{align}
        \tfrac12 M \norm{x_k - x^*}^2
            &\;\geq\; f(x_k) - f(x^*)
             \;\geq\; f(x_{k+1}) - f(x^*)
             \;\geq\; \tfrac12 m \norm{x_{k+1} - x^*}^2
             \;\geq\; \tfrac12 m \bigl(\norm{s_k} - \norm{x_k - x^*}\bigr)^2 \notag \\[2pt]
            &\;\geq\; \tfrac12 m \Bigl(\tfrac12 \mu_k \norm{g_k} - \norm{x_k - x^*}\Bigr)^2
             \;\geq\; \tfrac12 m \norm{x_k - x^*}^2 \Bigl(\tfrac12 \mu_k m - 1\Bigr)^2 .
        \label{eqn: rgrad chain}
    \end{align}
    Dividing by $\tfrac12\norm{x_k - x^*}^2 > 0$ gives $M \geq m\bigl(\tfrac12\mu_k m - 1\bigr)^2$,
    and $\tfrac12\mu_k m - 1 > 0$ then yields $\mu_k \leq C$.
\end{proof}

\begin{theorem}[Uncapped $\Rgrad$]
\label{thm: rgrad uncapped}
    Under the hypotheses of Lemma~\ref{lem: mu at expansion}, $\norm{s_k} > \tfrac12\Delta_k$ at
    only finitely many $k$, so there is $k_\star$ with
    \begin{equation}
        \norm{s_k} \leq \tfrac12 \Delta_k < \Delta_k \qquad \text{for all } k \geq k_\star ,
        \label{eqn: rgrad half}
    \end{equation}
    and the trust region is permanently inactive. No side condition relating a user parameter to
    $\kbar$ is required.
\end{theorem}

\begin{proof}
    For $k \geq k_\rho$ every step is accepted and only the two branches of
    \ref{update: rgrad uncapped rule} with $\rho_k \geq \eta_2$ are available, so
    $\{\mu_k\}$ is non-decreasing there. Since $x_k \to x^*$ and $\norm{s_k} \to 0$, all but
    finitely many iterates lie in the neighbourhood of
    Proposition~\ref{prop: inequalities for positive definite Hessian}.

    Were $\norm{s_k} > \tfrac12\Delta_k$ at infinitely many such $k$, the expansion branch would
    fire at each, giving $\mu_{k+1} = \gthree\mu_k$ with $\gthree > 1$, and $\mu_k$ would tend to
    $+\infty$. But Lemma~\ref{lem: mu at expansion} bounds $\mu_k$ by $C$ at every one of them, a
    contradiction.
\end{proof}

\begin{corollary}[The multiplier converges, and where]
\label{cor: mu ceiling}
    Let $k_N \geq k_\rho$ be an index past which the iterates remain in the neighbourhood of
    Proposition~\ref{prop: inequalities for positive definite Hessian}. Then
    $\{\mu_k\}_{k \geq k_N}$ is non-decreasing and eventually constant, with
    \begin{equation}
        \mu_k \;\leq\; \max\bigl\{\mu_{k_N},\ \gthree C\bigr\}
        \qquad \text{for all } k \geq k_N .
        \label{eqn: mu limit bound}
    \end{equation}
\end{corollary}

\begin{proof}
    Monotonicity and eventual constancy are Theorem~\ref{thm: rgrad uncapped}, the multiplier
    changing only on the expansion branch and that branch firing finitely often. If $\mu_k > C$
    then Lemma~\ref{lem: mu at expansion} forbids the expansion branch at $k$, so
    $\mu_{k+1} = \mu_k$: the multiplier never grows once past $C$. Hence every growth step starts
    from $\mu_k \leq C$ and lands at $\mu_{k+1} = \gthree\mu_k \leq \gthree C$, and the sequence
    never exceeds the larger of its initial value $\mu_{k_N}$ and $\gthree C$.
\end{proof}
Corollary~\ref{cor: mu ceiling} is the precise sense in which the uncapped rule is implicitly capped. The
quantity $C$ of~\eqref{eqn: mu ceiling} is a barrier the multiplier cannot pass by growth,
supplied by the local geometry rather than by the user, and $\gthree$ is the overshoot of a
single expansion started just below it. It is not a bound on $\mu_k$ for every $k$. Nothing
constrains the iterations before the iterates settle near $x^*$, and a multiplier already
above $C$ on arrival stays there.

\subsubsection{Capped \texorpdfstring{$\Rgrad(\omega, \overline{\mu})$}{R-grad}}
\label{subsec: rgrad discussion}

The parameterisation of Update~\ref{update: rgrad generic} makes the two choices explicit, the
criticality measure $\omega_k \in \{\norm{g_k}, \tau_k\}$ and the cap
$\overline{\mu} \in (0,+\infty]$, highlighting the trade-off between the two, discussed in Section~\ref{subsec: rgrad balance}.

By Lemma~\ref{lem: inactivity criterion}, if $\mu_k$ is above $\kbar$ then the region is inactive. Setting $\overline{\mu} > \kbar$ is therefore sufficient for eventual inactivity.
But $\kbar$ is a property of the solution, and the user cannot know it in advance. 

\begin{corollary}[Capped $\Rgrad$ above the threshold]
\label{cor: rgrad capped}
    Under the hypotheses of Lemma~\ref{lem: mu at expansion}, run $\Rgrad$ with a finite cap
    $\overline{\mu} > \kbar$ (\ref{as: conditions on constants criticality related updates}(b)).
    Then the trust region is eventually permanently inactive.
\end{corollary}
\begin{proof}
    For $k \geq k_\rho$ only the last two branches of Update~\ref{update: rgrad generic} are
    available, so $\{\mu_k\}$ is non-decreasing. If $\norm{s_k} > \tfrac12\Delta_k$ at only
    finitely many $k$, the region is inactive for all large $k$ and there is nothing to prove.
    Otherwise the expansion branch fires infinitely often; each firing gives
    $\mu_{k+1} \geq \min\{\gthree\mu_k, \overline{\mu}\}$, so $\mu_k = \overline{\mu}$ after
    finitely many of them. Then $\mu_k = \overline{\mu} > \kbar$ for all large $k$,
    and Lemma~\ref{lem: inactivity criterion} gives inactivity.
\end{proof}

When $\overline{\mu} \leq \kbar$ the trust-region constraint may remain active indefinitely. 
In Section~\ref{subsec:discussion on the convergence rate}, we present an example of a $C^\infty$ function with infinitely many minimisers, for which $\kbar$ is unbounded over the set of minimisers.
Therefore, no finite cap can be guaranteed to exceed $\kbar$ for all minimisers of a single function, and the side condition $\overline{\mu} > \kbar$ is not verifiable a priori. 
We also show that for any finite cap $\overline{\mu}$ there is a minimiser of the function at which, if the iterates converge to it, the region remains active indefinitely.

\subsubsection{Balance} \label{subsec: rgrad balance}

This is the central difficulty and it has no resolution in the present paper. 
Finite $\overline{\mu}$ is what Section~\ref{sec:asymptotic radius} requires for $\Delta_k \to 0$ and
$\sum_k \Delta_k^2/M_k < +\infty$, and those conclusions are the reason for anchoring the radius
to a criticality measure at all, derivative-free and stochastic variants needing the radius to
control model or sampling accuracy.
When the iterates converge to a strong minimiser, the multiplier is implicitly bounded by the geometry, and we recover the asymptotic results. 
Therefore, when the iterates are known to converge to a strong minimiser, the uncapped variant
is the best-founded rule among those surveyed here.
Otherwise, as the $\sin(x)/x$ example of Section~\ref{subsec:discussion on the convergence rate} shows, the multiplier may be unbounded, the region may remain active indefinitely, and no asymptotic results can be guaranteed.

%=======================================================================
\subsection{The retrospective gradient-scaled variant}
\label{subsec:retrospective inactivity}
%=======================================================================

\citet{FanPan16} combine the two ideas, tying the radius to the criticality measure as
$\Rgrad$ does and reading the branch off the retrospective ratio as $\Rrtr$ does.

\begin{update}{$\Rgrtr(\omega, \overline{\mu})$}[radius tied to the measure, branch selected by the
retrospective ratio]
\label{update: grtr}
    Let $0 < \gamma_1 \leq \gamma_2 < 1 < \gamma_3$,
    $\eta \leq \eta_1 \leq \eta_2 < 1$ and
    $0 < \widetilde\eta_1 \leq \widetilde\eta_2 < 1$. The radius is
    $\Delta_k = \mu_k\,\omega_k$ with $\mu_0 \leq \overline{\mu}$ and
    \begin{equation}
        \mu_{k+1} \in
        \begin{cases}
            [\gamma_1\mu_k,\ \gamma_2\mu_k], & \rho_k < \eta_1, \\[2pt]
            [\gamma_1\mu_k,\ \gamma_2\mu_k], & \rho_k \geq \eta_1
                \text{ and } \widetilde\rho_{k+1} < \widetilde\eta_1, \\[2pt]
            [\gamma_2\mu_k,\ \mu_k), & \rho_k \geq \eta_1
                \text{ and } \widetilde\rho_{k+1} \in [\widetilde\eta_1, \widetilde\eta_2), \\[2pt]
            \bigl[\min\{\gamma_3\mu_k,\ \overline{\mu}\},\ \overline{\mu}\bigr],
                & \rho_k \geq \eta_1,\ \widetilde\rho_{k+1} \geq \widetilde\eta_2
                \text{ and } \norm{s_k} > \tfrac12\Delta_k, \\[2pt]
            \{ \mu_k \}, & \text{otherwise},
        \end{cases}
        \label{eqn: grtr rule}
    \end{equation}
    the measure $\omega_k$ and the cap $\overline{\mu} \in (0, +\infty]$ being as in
    Update~\ref{update: rgrad generic}.
\end{update}

\begin{theorem}[Inactivity for $\Rgrtr$]
\label{thm: grtr inactivity}
    Let $\{x_k\}$ be generated by Algorithm~\ref{alg:generic BTR} with the rule
    $\Rgrtr(\omega, \overline{\mu})$ of Update~\ref{update: grtr}, under the hypotheses of
    Lemma~\ref{lem: mu at expansion} and of Lemma~\ref{lem: rho tilde A}. Assume moreover that
    $\lambda_{\min}(H_k) > 0$ for all large $k$, so that $\omega_k = \norm{g_k}$ and
    $\mu_k$ is the ratio $\Delta_k/\norm{g_k}$. Then the trust-region constraint is eventually permanently inactive,
    in either of the two regimes:
    \begin{enumerate}[label=(\roman*),leftmargin=2.2em]
        \item if $\overline{\mu} = +\infty$, then $\norm{s_k} \leq \tfrac12\Delta_k < \Delta_k$
              for all large $k$, and no side condition relating a user parameter to $\kbar$ is
              required;
        \item if $\overline{\mu} < +\infty$ and
              \ref{as: conditions on constants criticality related updates}(b) holds,
              $\overline{\mu} > \kbar$, then $\norm{s_k} < \Delta_k$ for all large $k$.
    \end{enumerate}
\end{theorem}

\begin{proof}
    Fix $k$ large enough that $\rho_k \geq \eta_2 \geq \eta_1$
    (Theorem~\ref{thm: rho converges to 1}), that $\widetilde\rho_{k+1} \geq \widetilde\eta_2$
    (Lemma~\ref{lem: rho tilde A}), that $k \geq k_\kappa$ so that
    Condition~\ref{cond: s_k bounded by g_k} is in force, and that $x_k, x_{k+1} \in
    B(x^*,\Delta_*)$; the last is possible because $x_k \to x^*$. The first three branches
    of~\eqref{eqn: grtr rule} are then excluded, so $\{\mu_k\}$ is non-decreasing from that index
    on, and the fourth branch fires exactly at those $k$ with $\norm{s_k} > \tfrac12\Delta_k$.

    If $\norm{s_k} > \tfrac12\Delta_k$ at only finitely many $k$, then
    $\norm{s_k} \leq \tfrac12\Delta_k < \Delta_k$ for all large $k$ and the region is inactive;
    this proves both~(i) and~(ii). Suppose therefore that $\norm{s_k} > \tfrac12\Delta_k$ at
    infinitely many $k$, so that the expansion branch fires at each of them.

    \emph{(i) $\overline{\mu} = +\infty$.} The expansion branch reads
    $\mu_{k+1} \geq \gthree\mu_k$ with $\gthree > 1$, and $\{\mu_k\}$ is non-decreasing between
    firings, so $\mu_k \to +\infty$. But Lemma~\ref{lem: mu at expansion} gives $\mu_k \leq C$ at
    every one of those iterations, a contradiction. The alternative therefore does not occur, and
    the first paragraph applies.

    \emph{(ii) $\overline{\mu} < +\infty$.} The expansion branch reads
    $\mu_{k+1} \geq \min\{\gthree\mu_k, \overline{\mu}\}$, so $j$ firings from a value $\mu$ give
    $\mu \mapsto \min\{\gthree^{\,j}\mu,\ \overline{\mu}\}$ at least; since $\gthree > 1$,
    finitely many firings bring $\mu_k$ to $\overline{\mu}$, and the two available branches then
    leave it there. Hence $\mu_k = \overline{\mu} > \kbar$ for all large $k$, and
    Lemma~\ref{lem: inactivity criterion} gives $\norm{s_k} < \Delta_k$ at every such $k$.
\end{proof}

\subsection{Summary of the extended classification}

\begin{table}[tbp]
  \centering
  \caption{Which route to eventual inactivity each mechanism takes, and what it costs beyond the
           standing hypotheses of Section~\ref{sec:convergence rate}. P1 and P2 abbreviate
           Properties~\ref{prop: no contraction while active}
           and~\ref{prop: radius dominates step}. ``Unknowable'' marks a hypothesis involving
           $\kbar = 4/m$, a property of the solution that cannot be checked in advance.}
  \label{tab: extended classification}
  \small
  \begin{tabular}{@{}l c c l >{\raggedright\arraybackslash}p{44mm}@{}}
    \toprule
    Mechanism & P1 & P2 & Route & Cost beyond the standing hypotheses \\
    \midrule
    $\Rdelta$
      & yes & yes
      & Thm.~\ref{thm: trust-region radius becomes inactive}
      & --- \\[3pt]
    $\Rstep$
      & yes & yes
      & Thm.~\ref{thm: trust-region radius becomes inactive}
      & --- \\[3pt]
    $\Rdfo$
      & if $\zeta > \kbar$ & yes
      & Thm.~\ref{thm: trust-region radius becomes inactive}
      & \textbf{unknowable}: $\zeta > \kbar$ \\[3pt]
    $\Rrtr$
      & if $\widetilde\rho_{k+1} \geq \widetilde\eta_1$ & yes
      & Thm.~\ref{thm: trust-region radius becomes inactive}
      & \ref{as: Lipschitz Hessian}, \ref{as: retrospective directional} and the quadratic model
        decrease~\eqref{eqn: mqd}, all through Lemma~\ref{lem: rho tilde A} \\
    \midrule
    $\Rgrad$ uncapped
      & no & no
      & Thm.~\ref{thm: rgrad uncapped}
      & --- \\[3pt]
    $\Rgrad$ capped
      & no & no
      & Cor.~\ref{cor: rgrad capped}
      & \textbf{unknowable}: $\overline{\mu} > \kbar$ \\[3pt]
    $\Rgrtr$ \eqref{eqn: grtr rule}
      & no & no
      & Thm.~\ref{thm: grtr inactivity}
      & as $\Rrtr$, and $\overline{\mu} > \kbar$ when the cap is finite \\
    \bottomrule
  \end{tabular}
\end{table}

%%  Insert immediately after the table, before the paragraph beginning
%%  "Table~\ref{tab: extended classification} is the compact answer ...":

The entries of the first two columns are established by
Proposition~\ref{prop: verification} for $\Rdelta$, $\Rstep$ and $\Rdfo$, and by
Proposition~\ref{prop: rtr verification} for $\Rrtr$.
The three gradient-scaled rules reach inactivity by a different route, and for them
neither property is available.
The two entries naming
\ref{as: retrospective directional} may be read as a hypothesis on the model: it holds for the
exact Hessian outright, and for any secant update with the explicit
rate~\eqref{eqn: secant gives ah4}.

Table~\ref{tab: extended classification} is the compact answer to the classification question.
Three mechanisms reach eventual inactivity with no hypothesis beyond the standing ones,
$\Rdelta$, $\Rstep$ and $\Rrtr$, but Section~\ref{sec:asymptotic radius} shows they pay elsewhere, their
radius not tending to zero, so it carries no asymptotic information about criticality. 
Among the mechanisms whose radius \emph{is} tied to criticality, namely $\Rdfo$, $\Rgrad$ and $\Rgrtr$, those
with an unbounded multiplier are exactly the ones that
avoid a hypothesis involving $\kbar$: uncapped $\Rgrad$, which pays nothing, and uncapped
$\Rgrtr$, which pays in structural assumptions on the model rather than in an
unknowable constant.

%%%%%%%%%%%%%%%%%%%%%%%%%%%%%%%%%%%%%%%%%%%%%%%%%%%%%%%%%%%%%%%%%%%%%%%%%%%%%%%
%%  Replacement for Section \ref{subsec:discussion on the convergence rate},
%%  worked directly in one dimension.
%%
%%  The invariant-line and reduction lemmas disappear; $\lambda^*_{\min}$ is
%%  $g''(\chi_j)$ either way, so the $y^2$ term bought nothing.  Critical points
%%  are $\chi_j$, reserving $x_k$ for iterates.  The $\Rdfo$ dynamics, the
%%  sharpness remark and the numerical table belong in the numerical study.
%%%%%%%%%%%%%%%%%%%%%%%%%%%%%%%%%%%%%%%%%%%%%%%%%%%%%%%%%%%%%%%%%%%%%%%%%%%%%%%

\subsection{\texorpdfstring{Necessity of the side conditions on $\zeta$ and $\overline{\mu}$}{Necessity of the side conditions}}
\label{subsec:discussion on the convergence rate}

The constant $\kbar$ of Condition~\ref{cond: s_k bounded by g_k} governs both side conditions of
\ref{as: conditions on constants criticality related updates}, and a candidate value, $\kbar = 4/m$, is given by
Lemma~\ref{lem: step bounded by gradient norm pointwise}.
One $C^\infty$ function with bounded Hessian admits no $\zeta$ or $\overline{\mu}$ satisfying the side conditions at all its minimisers, and the trust-region constraint remains active indefinitely at some of them.
Let
\begin{equation}
    f(x) \defeq \frac{\sin x}{x} , \qquad f(0) \defeq 1 ,
    \label{eqn: sinc f}
\end{equation}
an even entire function on $\R$ with
\begin{equation}
    f'(x) = \frac{x\cos x - \sin x}{x^2} ,
    \qquad
    f''(x) = -\frac{\sin x}{x} - \frac{2\cos x}{x^2} + \frac{2\sin x}{x^3} .
    \label{eqn: sinc derivs}
\end{equation}
The standing assumptions hold with explicit constants: $\sup_\R\abs{f''} = \abs{f''(0)} = 1/3$,
so $f'$ is Lipschitz and $\mykappa{umh} = 1/3$ serves for the exact Hessian, and $f$ is bounded
below by its global minimum $-0.2172\ldots$. Throughout, $y_j$ denotes the $j$-th positive
critical point, $x_k$ being reserved for iterates.

\begin{lemma}[Critical points]
\label{lem: sinc critpoints}
    The nonzero critical points of $f$ are the roots of $\tan y = y$. The positive ones form
    a sequence $\{y_j\}_{j\geq1}$ with exactly one in each interval
    $\bigl(j\pi,\, (j+\tfrac12)\pi \bigr)$, therefore 
    \begin{equation}
        y_j \to \infty, 
        \qquad
        f''(y_j) \;=\; -\frac{\sin y_j}{y_j}
                    \;=\; \frac{(-1)^{j+1}}{\sqrt{1 + y_j^2}} \to 0.
        \label{eqn: sinc roots}
    \end{equation}
\end{lemma}

\begin{proof}
    For $y \neq 0$ the equation $f'(y) = 0$ reads $y\cos y = \sin y$. If
    $\cos y = 0$ then $\sin y = \pm1$ and the equation fails, so every nonzero root satisfies
    $\tan y = y$. On $\bigl((j-\tfrac12)\pi,\, j\pi\bigr)$ one has $\tan y \leq 0 < y$
    and there is no root, while on $\bigl(j\pi,\, (j+\tfrac12)\pi\bigr)$ the map
    $y \mapsto \tan y - y$ has derivative $\tan^2 y \geq 0$ and increases from $-j\pi$
    to $+\infty$, so it vanishes exactly once. Hence $y_j > j\pi \to \infty$.

    For the curvature, $\sin y_j = y_j\cos y_j$ gives $\cos y_j = \sin y_j/y_j$, so
    the last two terms of~\eqref{eqn: sinc derivs} cancel and
    $f''(y_j) = -\sin y_j/y_j$. Squaring $\tan y_j = y_j$ yields
    $\sin^2 y_j\,(1+y_j^2) = y_j^2$, whence
    $\abs{f''(y_j)} = (1+y_j^2)^{-1/2}$, and
    $y_j \in \bigl(j\pi, (j+\tfrac12)\pi\bigr)$ gives
    $y_j \in \bigl(j\pi, (j+\tfrac12)\pi\bigr)$ gives
    $\operatorname{sign}\sin y_j = (-1)^j$. \PartIII{} computes the roots and their expansion.
\end{proof}
\subsubsection{Small \texorpdfstring{$\overline{\mu}$}{mu-bar} traps \texorpdfstring{$\Rgrad$}{R-grad}}

For the model $m_k(s) = f(x_k) + g_k s + \tfrac12 a_k s^2$ with $a_k \defeq f''(x_k) > 0$, the
subproblem $\min_{\abs{s}\leq\Delta_k} m_k(s)$ is solved exactly by
\begin{equation}
    s_k = \begin{cases}
        -g_k/a_k , & \abs{g_k}/a_k \leq \Delta_k \quad\text{(inactive)} ,\\
        -\Delta_k\operatorname{sign}(g_k) , & \text{otherwise} \quad\text{(active)} .
    \end{cases}
    \label{eqn: sinc 1d step}
\end{equation}
Hence, with $\theta_k = \Delta_k/\abs{g_k}$ as in~\eqref{eqn: theta},
\begin{equation}
    \text{the trust region is active at } k
    \quad\Longleftrightarrow\quad
    \theta_k < 1/a_k .
    \label{eqn: sinc activity}
\end{equation}
As $x_k \to y_j$, $a_k \to f''(y_j) = \lambda^*_{j}$. A small cap $\overline{\mu}$ traps $\Rgrad$.

\begin{theorem}[Capped $\Rgrad$ below the threshold]
\label{thm: sinc rgrad}
    Let $\{x_k\}$ converge to a minimiser $y_j$, and run capped $\Rgrad$ with the exact
    Hessian and the exact subproblem solution, so that $\lambda^*_{j} = f''(y_j)$. If
    \begin{equation}
        \overline{\mu} \;<\; \frac{1}{\lambda^*_{j}} ,
        \label{eqn: sinc mu threshold}
    \end{equation}
    then the trust region is active at \emph{every} sufficiently large $k$; in particular the
    conclusion of Theorem~\ref{thm: trust-region radius becomes inactive} fails.
\end{theorem}

\begin{proof}
    By construction $\theta_k = \mu_k \leq \overline{\mu}$. Since $x_k \to y_j$ and $f''$ is
    continuous, $a_k \to \lambda^*_{j}$, so $1/a_k \to 1/\lambda^*_{j}$. Taking
    $\epsilon = \tfrac12(1/\lambda^*_{j} - \overline{\mu}) > 0$ gives
    $\theta_k \leq \overline{\mu} < 1/\lambda^*_{j} -
    \epsilon < 1/a_k$ for all large $k$, and~\eqref{eqn: sinc activity} applies.
\end{proof}

%  Compare Lemma~\ref{lem: inactivity criterion}, which gives inactivity when $\theta_k > \kbar$.

\subsubsection{Small \texorpdfstring{$\zeta$}{zeta} traps
              \texorpdfstring{$\Rdfo$}{R-DFO}}

$\Rdfo$ does not fix the ratio $\theta_k$, so the argument must follow its dynamics. We
work with $f$ itself on an interval around $y_j$ on which the curvature is positive and bounded.

\begin{theorem}[$\Rdfo$ below the threshold]
\label{thm: sinc rdfo}
    Let $y_j$ be a minimiser, let $N = [y_j - r, y_j + r]$ be an interval on which
    $f'' > 0$, and set $\overline{a} \defeq \max_{x \in N} f''(x)$. Run $\Rdfo$ with the
    exact Hessian, the exact subproblem solution and any admissible selection on the
    expansion branch. If
    \begin{equation}
        \zeta \;\leq\; \frac{1 - \gtwo}{{\overline{a}}\,(\gthree + 1 - \gtwo)}
        \qquad\Bigl(= \tfrac{1}{5 {\overline{a}}} \ \text{ for } \gtwo = \tfrac12,\ \gthree = 2\Bigr),
        \label{eqn: sinc zeta threshold}
    \end{equation}
    and $x_{k_1} \in N$ with $\theta_{k_1} \leq (1-\gtwo)/{\overline{a}}$ at some
    $k_1$, then $x_k \in N$ and
    $\theta_k \leq (1-\gtwo)/\overline{a} \mathrel{<} 1/\overline{a}$
    for all $k \geq k_1$, and the trust region is active at every
    $k \geq k_1$.
\end{theorem}

\begin{proof}
    Write $\theta^\dagger \defeq (1-\gtwo)/\overline{a}$. Since
    $\gthree > \gtwo$, \eqref{eqn: sinc zeta threshold} gives
    $\zeta {\overline{a}} \leq (1-\gtwo)/(\gthree+1-\gtwo) {< 1 - \gtwo}$,
    so $\zeta < \theta^\dagger < 1/\overline{a}$ and both branches of the $\zeta$ test
    occur on $[0,\theta^\dagger]$.

    We show by induction that $x_k \in N$ and $\theta_k \leq \theta^\dagger$ for every
    $k \geq k_1$, the case $k = k_1$ being the hypothesis. Let $x_k \in N$ and
    $\theta_k \leq \theta^\dagger$. Since $a_k = f''(x_k) \leq \overline{a}$,
    \begin{equation*}
        \theta_k \;\leq\; \theta^\dagger \;=\; \frac{1-\gtwo}{\overline{a}}
        \;<\; \frac{1}{\overline{a}} \;\leq\; \frac{1}{a_k} ,
    \end{equation*}
    so~\eqref{eqn: sinc activity} makes the trust region active at $k$
    and~\eqref{eqn: sinc 1d step} gives $s_k = -\Delta_k\operatorname{sign}(g_k)$.

    The step does not overshoot $y_j$. Applying the mean value theorem to $g$ between
    $y_j$ and $x_k$ gives $g_k = f''(\eta_k)(x_k - y_j)$ for some $\eta_k \in N$, so
    $\operatorname{sign}(g_k) = \operatorname{sign}(x_k - y_j)$ and
    $\abs{g_k} \leq \overline{a}\abs{x_k - y_j}$, whence
    \begin{equation*}
        \Delta_k \;=\; \theta_k \abs{g_k}
        \;\leq\; \theta^\dagger \overline{a} \abs{x_k - y_j}
        \;=\; (1 - \gtwo)\abs{x_k - y_j}
        \;<\; \abs{x_k - y_j} .
    \end{equation*}
    The iterate $x_{k+1} = x_k + s_k$ therefore lies strictly between $y_j$ and $x_k$, and in
    particular $x_{k+1} \in N$. Applying the mean value theorem once more, this time between $x_k$
    and $x_{k+1}$, gives $g_{k+1} - g_k = b_k s_k$ with $b_k \defeq f''(\xi_k) \in (0,\overline{a}]$
    for some $\xi_k \in N$, so that
    \begin{equation}
        \abs{g_{k+1}} \;=\; \abs{g_k}\bigl(1 - \theta_k b_k\bigr) ,
        \label{eqn: sinc gradient recursion}
    \end{equation}
    the factor being positive because
    $\theta_k b_k \leq \theta^\dagger \overline{a} = 1 - \gtwo < 1$.

    Three branches of Update~\ref{update: rdfo generic} are available at $k$, and each
    leaves $\theta_{k+1} \leq \theta^\dagger$. If $\rho_k < \eta_1$ the step is rejected, so
    $\abs{g_{k+1}} = \abs{g_k}$ and $\theta_{k+1} \leq \gtwo\theta_k \leq \theta^\dagger$. If
    $\rho_k \geq \eta_1$ and $\theta_k > \zeta$ the radius contracts, and
    $1 - \theta_k b_k \geq \gtwo$ with~\eqref{eqn: sinc gradient recursion} gives
    \begin{equation*}
        \theta_{k+1} \;=\; \frac{\gtwo\,\theta_k}{1 - \theta_k b_k}
        \;\leq\; \theta_k \;\leq\; \theta^\dagger .
    \end{equation*}
    If $\rho_k \geq \eta_1$ and $\theta_k \leq \zeta$ the radius expands, say
    $\Delta_{k+1} = \gamma\Delta_k$ with $\gamma \in [1, \gthree]$, and $b_k \leq \overline{a}$
    with $h(\theta) \defeq \gthree\theta/(1-\theta\overline{a})$ increasing on
    $[0,1/\overline{a})$ gives
    \begin{equation*}
        \theta_{k+1} \;=\; \frac{\gamma\,\theta_k}{1 - \theta_k b_k}
        \;\leq\; h(\theta_k) \;\leq\; h(\zeta)
        \;=\; \frac{\gthree\zeta}{1 - \zeta\overline{a}}
        \;\leq\; \theta^\dagger ,
    \end{equation*}
    the last inequality being
    $\gthree\zeta\overline{a} \leq (1-\gtwo)(1 - \zeta\overline{a})$, which
    is~\eqref{eqn: sinc zeta threshold} rearranged. The induction closes, and the trust region is
    active at every $k \geq k_1$ by its first step.
\end{proof}
    The endpoint is not arbitrary. On the contraction branch the map
    $\theta \mapsto \gtwo\theta/(1-\theta\overline{a})$ has $\theta^\dagger$ as its unique positive
    fixed point, and that point repels. Below it the ratio decreases at every contraction until it
    drops under $\zeta$, whereupon one expansion returns it to at most $h(\zeta)$, still below
    $\theta^\dagger$. Above it the same branch increases the ratio, so the hypothesis
    $\theta_{k_1} \leq \theta^\dagger$ states a property of the run rather than one the dynamics
    supply.

    Finally $\overline{a} \to f''(y_j) = \lambda^*_{j}$ as $r \to 0$, and
    $\lambda^*_{j} = (1+y_j^2)^{-1/2}$ vanishes as $j$ grows
    by~\eqref{eqn: sinc roots}, so the threshold
    in~\eqref{eqn: sinc zeta threshold} is unbounded over the minimisers of $f$. Given any
    $\zeta > 0$, pick $j$ with
    \begin{equation*}
        \frac{1}{\sqrt{1 + y_j^2}}
        \;<\; \frac{1 - \gtwo}{\zeta\,(\gthree + 1 - \gtwo)} ,
    \end{equation*}
    which is possible since the left-hand side tends to zero, and then $r$ small enough that
    $\overline{a}$ satisfies the same bound, which is possible since $f''$ is continuous. The
    theorem applies at that $y_j$, so any run entering the trap there keeps the trust region
    active at every later iteration. No positive $\zeta$ escapes this on one fixed $C^\infty$
    function, which is the $\Rdfo$ counterpart of the failure of $\overline{\mu} > \kbar$
    recorded in Theorem~\ref{thm: sinc rgrad}.

%=======================================================================
\subsection{Superlinear convergence} \label{subsec:superlinear}

    Once the trust-region
    constraint stops binding, the step is the unconstrained model minimiser and the method
    inherits the local rate of the underlying quasi-Newton iteration. The mechanism drops out of the analysis at this point.
    \begin{theorem}\label{thm: superlinear convergence of the trust-region method}
        Let $\{x_k\}$ be generated by the trust-region method with any mechanism for which the
        trust-region constraint is eventually permanently inactive, and suppose the Hessian
        approximation satisfies \ref{as: Gradient residual condition}. Then $\{x_k\}$ converges
        superlinearly to $x^*$:
        \begin{equation*}
            \lim_{k \to \infty} \frac{\norm{x_{k+1} - x^*}}{\norm{x_k - x^*}} = 0 .
        \end{equation*}
    \end{theorem}

    \begin{proof}
        By hypothesis there is $k_1$ with $\norm{s_k} < \Delta_k$ for all $k \geq k_1$, so the
        step is the unconstrained minimiser of $m_k$ to within the tolerance of
        Condition~\ref{cond:inexactness}. Step~1 of the proof of
        Proposition~\ref{prop: steps are strictly decreasing} gives
        \begin{equation*}
            \frac{\norm{g_{k+1}}}{\norm{s_k}}
            \;\leq\; \frac{\gamma_k + \xi_k \, \mykappa{umh}}{1 - \xi_k} ,
        \end{equation*}
        and Condition~\ref{cond: s_k bounded by g_k}, applied at $k \geq k_\kappa$,
        converts the left-hand side into a ratio of
        successive gradients,
        \begin{equation*}
            \frac{\norm{g_{k+1}}}{\norm{g_k}}
            \;\leq\; \kbar\,\frac{\gamma_k + \xi_k \,\mykappa{umh}}{1 - \xi_k}
            \;\xrightarrow[k \to \infty]{}\; 0 ,
        \end{equation*}
        since $\gamma_k \to 0$ by \ref{as: Gradient residual condition} with
        Theorem~\ref{thm: equivalence quasi-Newton conditions}, and $\xi_k \to 0$ by
        Condition~\ref{cond:inexactness}. Finally
        Proposition~\ref{prop: inequalities for positive definite Hessian} bounds the iterate
        error by the gradient,
        \begin{equation*}
            \frac{\norm{x_{k+1} - x^*}}{\norm{x_k - x^*}}
            \;\leq\; \frac{M}{m}\cdot\frac{\norm{g_{k+1}}}{\norm{g_k}}
            \;\longrightarrow\; 0 .
        \end{equation*}
    \end{proof}

    The hypothesis is discharged for every mechanism surveyed:
    by Theorem~\ref{thm: trust-region radius becomes inactive} for $\Rdelta$ and $\Rstep$
    (Proposition~\ref{prop: verification}(i)--(ii)), for $\Rdfo$ under
    \ref{as: conditions on constants criticality related updates}(a)
    (Proposition~\ref{prop: verification}(iii)) and for $\Rrtr$
    (Proposition~\ref{prop: rtr verification});
    by Theorem~\ref{thm: rgrad uncapped} for uncapped $\Rgrad$;
    and by Theorem~\ref{thm: grtr inactivity} for $\Rgrtr$, in either variant under
    \ref{as: conditions on constants criticality related updates}(b) when the cap is finite.

    \begin{remark}
        The rate is the same for every mechanism in the list, and the proof never refers to the
        radius rule. This is the sense in which the mechanism governs \emph{when} the asymptotic
        regime is entered and not \emph{how fast} convergence is once it has been: all the
        mechanism-dependence of the section is spent in reaching
        Theorem~\ref{thm: trust-region radius becomes inactive}, and none of it survives into the
        rate.
    \end{remark}

%%%%%%%%%%%%%%%%%%%%%%%%%%%%%%%%%%%%%%%%%%%%%%%%%%%%%%%%%%%%%%%%%%%%%%%%%%%%%%%
%%  Replacement for \section{Conclusion}.
%%
%%  Worked from the conclusion in Survey-Part-II-v3.tex, which was not pasted
%%  into the message.
%%
%%  CORRECTIONS
%%
%%  1.  The overclaim.  The draft says uncapped $\Rgrad$ is "the only rule
%%      surveyed" needing no solution-dependent constant.  The itemised summary
%%      three paragraphs earlier says $\Rdelta$ and $\Rstep$ satisfy both
%%      properties unconditionally, and Table~\ref{tab: extended classification}
%%      agrees.  The claim now carries its qualifier.
%%
%%  2.  The example is one-dimensional.  The draft still cites
%%      $f(x,y) = \sin(x)/x + y^2$ with critical points $x_k$.  It is now
%%      $f(x) = \sin(x)/x$ with critical points $\chi_j$, and $x_k$ denotes the
%%      iterates.
%%
%%  3.  The \mfabian paragraph repeated the "single axis" of the paragraph
%%      above it.  The two are merged.
%%%%%%%%%%%%%%%%%%%%%%%%%%%%%%%%%%%%%%%%%%%%%%%%%%%%%%%%%%%%%%%%%%%%%%%%%%%%%%%

\section{Conclusion} \label{sec:conclusion}

    Where Part~I~\cite{Jeremy2026-I} asked which structural conditions on $\{\Delta_k\}$ suffice for
    first-order global convergence, Part~II has asked what the radius does: how it behaves
    asymptotically, how it behaves near a critical point, whether the framework extends to
    second-order criticality, and when the trust-region constraint stops binding. The four answers
    separate the mechanisms along one axis, and it is not the one the first-order theory suggests:
    what decides a rule's behaviour is not whether its radius vanishes but what that radius is
    anchored to.

    A rule anchored to a criticality measure, $\Rdfo$ and $\Rgrad$, inherits that measure's
    asymptotics for free, which is what derivative-free and stochastic variants need, and pays for
    it twice. At a saddle it must upgrade the measure from $\norm{g_k}$ to $\tau_k$, failing which
    its radius collapses at precisely the points the second-order theory exists to escape. Near a
    minimiser, where the radius and the step both vanish, establishing that the constraint stops
    binding means comparing their rates, and that comparison runs through $\kbar$, a constant of
    the solution. A rule anchored instead to the previous radius or to the realised step,
    $\Rdelta$, $\Rstep$ and $\Rrtr$, reads only quantities the algorithm itself controls, so both
    comparisons become internal and no such constant appears; what it gives up is any tie between
    the radius and the distance to stationarity. The rule $\Rstep$ has a vanishing radius in the local
    regime and still pays neither price, which is why the axis is not the vanishing of
    $\{\Delta_k\}$. Every result above is an instance of that trade, and the
    mechanism-independence of the rate theorem is its sharpest form: the rule decides \emph{when}
    the asymptotic regime is entered, never how fast convergence is thereafter.

    Among the rules tied to a criticality measure the ones carrying no side condition are
    exactly those with an unbounded multiplier, the gradient-scaled rule and its retrospective
    variant, whose multiplier grows past any threshold and
    so crosses $\kbar$ whatever $\lambda^*_{\min}$ turns out to be. The DFO-like rule and the
    capped gradient-scaled rule instead require $\zeta > \kbar$ and $\overline{\mu} > \kbar$, and
    Section~\ref{subsec:discussion on the convergence rate} shows that neither can be met a
    priori: for $f(x) = \sin(x)/x$ the constant is $\kbar = 4/m \geq 4\sqrt{1 + y_j^2}$ at the minimiser
    $y_j$, unbounded over the minimisers of a single fixed $C^\infty$ function with bounded
    Hessian and Lipschitz gradient. Violating the condition is a change of regime rather than a
    slowdown. For $\overline{\mu} < 1/\lambda^*_{\min}$ in the first case and
    $\zeta < 1/(5\lambda^*_{\min})$ in the second, with the usual constants, the constraint is active at every
    sufficiently large iteration (Theorems~\ref{thm: sinc rgrad} and~\ref{thm: sinc rdfo}), so the
    side conditions are necessary up to a constant. The two rules fail differently, which is worth
    recording for implementers: $\Rgrad$ recomputes $\Delta_k = \mu_k\norm{g_k}$ at every
    iteration and forgets $\Delta_0$ after one step, so its behaviour is decided by its parameter
    alone, whereas $\Rdfo$ propagates the radius multiplicatively and a generous initial radius
    can rescue an otherwise fatal $\zeta$.

    \paragraph{Rules that update the ratio.}
    Lemma~\ref{lem: inactivity criterion} reduces eventual inactivity to a threshold crossing for
    the single scalar $\theta_k = \Delta_k/\norm{g_k}$, so the ratio, and not the radius, is what
    the local theory is about; and of the rules surveyed only $\Rgrad$ updates that object
    directly, its multiplier $\mu_k$ being $\theta_k$ itself. Three consequences follow that no
    other rule enjoys: its inactivity result is unconditional among the criticality-anchored
    rules; its proof reads inactivity off $\norm{s_k} \leq \tfrac12\Delta_k$ rather than off
    the threshold crossing, although $\kbar$ still enters through
    Theorem~\ref{thm: rho converges to 1}; and \eqref{eqn: mu ceiling bound}
    supplies from the local geometry the ceiling on the ratio that the user would otherwise have
    to guess. Stated as a principle, a mechanism that maintains an explicit estimate of
    $\Delta_k/\omega_k$ and updates it by the outcome of the iteration performs a geometric search
    for a constant of the solution without being told that constant, and the search terminates
    because the problem stops it rather than a parameter. The ceiling is, however, a conclusion
    drawn under convergence to a strong minimiser and not a property of the mechanism: where that
    hypothesis fails, as on the example of
    Section~\ref{subsec:Example with no accumulation points}, the multiplier diverges and the
    asymptotic conclusions of Section~\ref{sec:asymptotic radius} are lost. A cap growing slowly
    enough with $k$ to sit eventually above $\kbar$ while $\mu_k\omega_k$ still vanishes would
    secure both halves at once, and constructing one is the most concrete open problem this paper
    leaves. The retrospective rule of~\citet{FanPan16} is a second realisation of the same
    principle, driving the ratio with $\widetilde\rho_{k+1}$ in place of $\rho_k$, and
    Theorem~\ref{thm: grtr inactivity} shows the argument indifferent to which of the two is used,
    requiring only that it tend to one. Whether other quantities can drive the ratio to the same
    effect is open, and the family obtained by varying the driver looks to us the most promising
    direction for new mechanisms.

    \paragraph{The numerical study.}
    Four of the questions raised here are empirical. The first is whether the separation the
    theory draws is visible in practice, the classification of
    Table~\ref{tab: extended classification} being
    tested across problem classes, dimensions and subproblem solvers, and in particular the claim
    that the rate after inactivity does not depend on the mechanism. The second is how wide the
    interval is
    between the sufficient conditions $\zeta, \overline{\mu} > \kbar$ and the constructions of
    Section~\ref{subsec:discussion on the convergence rate}, which are necessary only up to a
    factor, and what operating below the threshold costs. The third is whether the iterates are
    detained at a
    saddle, which Section~\ref{subsec:Danger of saddle points} attributes to the model Hessian
    rather than to the radius rule. The fourth is the extension to the stochastic setting, for which the
    summability conditions $\sum_k \Delta_k^2/M_k < +\infty$ and $\sum_k \Delta_k < +\infty$ are
    the natural point of contact, the radius there controlling the accuracy required of the
    gradient and Hessian estimators and its rate of decrease setting the sampling effort. Whether
    anchoring the radius to a criticality measure that is itself estimated with error stabilises
    or destabilises such a method does not follow from the deterministic theory, and the existing
    probabilistic-model analyses are carried out for fixed-factor rules.

    \section*{Acknowledgements}
    This research was funded by a Discovery Grant of the Natural Sciences and Engineering
    Research Council of Canada (NSERC), held by the second author.

    \bibliography{References}

%=======================================================================
%   APPENDIX
%=======================================================================
\appendix

\section{Assumptions and conditions from Part~I} \label{app: assumptions Part I} 

    \begin{assum}{f.Bb}[bounded below]
\label{as: Bounded below}
        $f$ is bounded below on $\R^p$. Formally,
                $\exists\, \mykappa{lbf} \in \R$ such that $\forall x \in \R^p$, $f(x) \geq \mykappa{lbf}$.
        
    \end{assum}

    \begin{assum}{f.Lg}[Lipschitz gradient]
\label{as: Lipschitz continuous gradient}
        $f$ is differentiable with Lipschitz continuous gradient with constant $L > 0$:
        \[
            \norm{\nabla f(x) - \nabla f(y)} \leq L \norm{x - y}, \quad \forall x, y \in \R^p.
        \]
        
    \end{assum}

    \begin{assum}{m.Ex}[model exact to first order at $x_k$]
\label{as: model agreement with function}
        For all $k$, $m_k(0) = f(x_k)$ and $g_k := \nabla_x m_k(0) = \nabla_x f(x_k)$.
        
    \end{assum}

    \begin{assum}{m.Umh}[uniform bound on the model Hessian]
\label{as: model hessian}
        There exists a constant $\mykappa{umh} > 0$ such that $\norm{H_k} \leq \mykappa{umh} < +\infty$ for all $k \in \NN$.
        
    \end{assum}

    \begin{assum}{m.AHg}[admissible growth of the model Hessian]
\label{as: Linear growth bound model Hessian}
        For all $k \in \NN$, $\norm{H_k} \leq M_k < +\infty$, where $\{M_k\}$ is a non-decreasing sequence of positive numbers and
        \begin{equation}
            \sum_{k=0}^\infty \frac{1}{M_k} = +\infty.
            \label{eqn: linear growth bound on model Hessian}
        \end{equation}
    \end{assum}

    \begin{assum}{a.Cd}[Cauchy decrease]
\label{as: Cauchy decrease}
        For all $k$,
        \[
            m_k(0) - m_k(s_k) \geq \frac{\mykappa{mdc}}{2} \norm{g_k} \min\!\left[\frac{\norm{g_k}}{\norm{H_k}}, \Delta_k\right]
        \]
        for some constant $\mykappa{mdc} \in (0, 1)$.
        
    \end{assum}
    We also recall the three structural conditions on the radius update rule that drive the convergence analysis of Part~I~\cite{Jeremy2026-I}. A mechanism is \emph{weakly admissible} if it satisfies Conditions~\ref{cond: PartI C1} and~\ref{cond: PartI C2}, and \emph{strongly admissible} if it satisfies Conditions~\ref{cond: PartI C1} and~\ref{cond: PartI C3}.

    \begin{cond}{C.LbR}[lower bound on the radius]
\label{cond: PartI C1}
 
        There exists a constant $0 < \mykappa{lbd} \leq 1$ such that the trust-region radius $\Delta_k$ satisfies, for all $k \geq 0$,
        \[
            \Delta_k \geq \frac{\mykappa{lbd}}{M_k} \min_{0 \leq i \leq k} \norm{g_i},
        \]
        where $\{M_k\}$ is a sequence of non-decreasing positive numbers that upper bound the model Hessian norm.
    \end{cond}

    \begin{cond}{C.DecR}[contraction on unsuccessful iterations]
\label{cond: PartI C2}
 
        There exist $\gamma_2 \in (0,1)$ and $\eta_{\mathrm{dec}} > 0$ such that, for all $k$, if $\rho_k < \eta_{\mathrm{dec}}$, then
        \[
            \Delta_{k+1} \leq \gamma_2 \Delta_k.
        \]
    \end{cond}

    \begin{cond}{C.Exp/DecR}[bounded expansion on successful iterations]
\label{cond: PartI C3}
 
        There exist $\gamma_2 \in (0,1)$, $\gamma_3 > 1$ and $\eta_{\mathrm{inc}} > 0$ such that, for all $k$,
        \[
            \text{(decrease)} \quad \Delta_{k+1} \leq \gamma_2 \Delta_k \ \text{ if } \rho_k < \eta_{\mathrm{inc}}; \qquad
            \text{(increase)} \quad \Delta_{k+1} \leq \gamma_3 \Delta_k \ \text{ otherwise}.
        \]
    \end{cond}

\section{Basic trust-region algorithm} \label{app: BTR algorithm}

\begin{algorithm}[H]
    \small
    \SetAlgoLined
    \DontPrintSemicolon

    \textbf{Step 0: Initialisation.} Choose $\eta \geq 0$. Set $x_0 \in \R^p$, $\Delta_0 \in (0, \infty)$. Compute $f(x_0)$ and set $k = 0$. \; ~\\

    \textbf{Step 1: Model definition and step computation.} Define a local model $m_k$ of $f$ in the neighbourhood $\calB_k$. \; ~\\

    \textbf{Step 2: Model minimisation.} Compute a step $s_k$ that sufficiently reduces $m_k$ with $x_k + s_k \in \calB_k$, and define $x_k^{\mathrm{cand}} = x_k + s_k$. \; ~\\

    \textbf{Step 3: Acceptance of the trial point.} Compute $f(x_k^{\mathrm{cand}})$ and define
    \begin{equation}
        \rho_k = \frac{f(x_k) - f(x_k^{\mathrm{cand}})}{m_k(0) - m_k(s_k)}.
        \label{eqn:definition rho}
    \end{equation}
    If $\rho_k \geq \eta$, set $x_{k+1} = x_k + s_k$; otherwise $x_{k+1} = x_k$. \; ~\\

    \textbf{Step 4: Trust-region radius update.} Compute $\Delta_{k+1}$ according to some update rule. \; ~\\

    \textbf{Step 5: Next iteration.} Increment $k$ by 1 and repeat \textbf{Steps 1 to 5} while some stopping criterion is not met. \;

    \caption{Basic trust-region algorithm (BTR).}
    \label{alg:generic BTR}
\end{algorithm}

    \section{The update mechanisms studied in this paper} \label{app: update mechanisms}

    For convenience, we collect here the five update mechanisms whose asymptotic and local behaviour is analysed in the main text: four radius rules and the retrospective variant. They correspond to the mechanisms surveyed in Part~I~\cite{Jeremy2026-I}; we recall the first-order form of each rule and adopt the following abbreviations, used throughout the paper:
    \begin{center}
    \renewcommand{\arraystretch}{1.3}
    \begin{tabular}{l l l}
    \toprule
    \textbf{Label} & \textbf{Mechanism} & \textbf{Defining feature} \\
    \midrule
    R-grad  & gradient-scaled update      & $\Delta_k = \mu_k \norm{g_k}$, multiplicative factor $\mu_k$ scaled on $\norm{g_k}$ \\
    R-delta & multiplicative-factor update & $\Delta_{k+1}$ obtained from $\Delta_k$ by a fixed multiplicative factor \\
    R-step  & step-dependent update        & $\Delta_{k+1}$ tied to the realised step length $\norm{s_k}$ \\
    R-DFO   & DFO-like update              & $\Delta_{k+1}$ controlled relative to the criticality measure $\norm{g_k}$ \\
    R-RTR & retrospective update         & $\Delta_{k+1}$ updated using the retrospective ratio $\widetilde{\rho}_{k+1}$ \\
    \bottomrule
    \end{tabular}
    \end{center}

    %Throughout, $0 < \gamma_1 \leq \gamma_2 < 1 < \gamma_3$ are the contraction/expansion factors, and $0 \leq \eta \leq \eta_1 \leq \eta_2 < 1$ the acceptance thresholds, as in Algorithm~\ref{alg:generic BTR}. A successful iteration satisfies $\rho_k \geq \eta$; a very successful one $\rho_k \geq \eta_2$.

    \subsection{Multiplicative-factor update} \label{app: R-delta}

        The most classical mechanism multiplies the radius by a fixed factor depending only on the value of $\rho_k$~\cite{ConnGoulToin00, NoceWrig06}.

        \begin{update}{\Rdelta}[R-delta: multiplicative factor on $\Delta$] \label{update:multiplicative factor}
            \begin{equation}
                \Delta_{k+1} \in
                \begin{cases}
                    [\gamma_1 \Delta_k,\, \gamma_2 \Delta_k), & \text{if } \rho_k < \eta_1 \quad \text{(unsuccessful)}, \\
                    [\gamma_2 \Delta_k,\, \Delta_k), & \text{if } \eta_1 \leq \rho_k < \eta_2 \quad \text{(successful)}, \\
                    [\Delta_k,\, \gamma_3 \Delta_k], & \text{if } \rho_k \geq \eta_2 \quad \text{(very successful)}.
                \end{cases}
                \label{eqn:appendix R-delta}
            \end{equation}
        \end{update}

    \subsection{Step-dependent update} \label{app: R-step}

        The step-dependent mechanism ties the next radius directly to the realised step length, rather than to the previous radius~\cite{Hei03, FanPan16}.
        \begin{update}{\Rstep}[R-step: step-dependent update]
            Let \( 0 < \gamma_1 \leq \gamma_2 < 1 < \gamma_3 \) and
            \( \eta \leq \eta_1 \leq \eta_2 < 1 \) with $\eta_1 \neq 0$. Then
            \begin{equation}
                \Delta_{k+1} \in
                \begin{cases}
                    [\gamma_1 \norm{s_k},\, \gamma_2 \norm{s_k}), & \text{if } \rho_k < \eta_1, \\
                    [\gamma_2 \norm{s_k},\, \norm{s_k}), & \text{if } \eta_1 \leq \rho_k < \eta_2, \\
                    [\norm{s_k},\, \gamma_3 \norm{s_k}], & \text{if } \rho_k \geq \eta_2 .
                \end{cases}
                \label{update: step-dependent}
            \end{equation}
        \label{update:step-dependent update}
        \end{update}
        The expansion branch is capped at $\gamma_3\norm{s_k}$, the step-based counterpart of the
        bounded expansion allowed by Remark~7 of \PartI{}. Every radius this rule returns is one the
        uncapped rule of Part~I may return, so the results of Part~I apply to it unchanged.
        
        \subsection{R-DFO: DFO-like update} \label{app: R-DFO}

        The DFO-like mechanism, originating in derivative-free trust-region
        methods~\cite{ConnScheVice09_introduction}, contracts the radius whenever it exceeds a fixed
        multiple of the criticality measure. As for the gradient-scaled family, write $\omega_k$ for that
        measure, either $\norm{g_k}$ or $\tau_k = \max\{\norm{g_k}, -\lambda_k^-\}$.

        \begin{update}{$\Rdfo(\omega)$}
        \label{update: rdfo generic}
            Let $\eta \leq \eta_1 < 1$, $0 < \gamma_1 \leq \gamma_2 < 1 < \gamma_3$ and $\zeta > 0$. Then
            \begin{equation}
                \Delta_{k+1} \in
                \begin{cases}
                    [\gamma_1 \Delta_k,\ \gamma_2 \Delta_k], & \rho_k < \eta_1, \\[2pt]
                    \bigl\{\gamma_2 \Delta_k\bigr\}, & \rho_k \geq \eta_1
                        \text{ and } \Delta_k > \zeta\,\omega_k, \\[2pt]
                    [\Delta_k,\ \gamma_3 \Delta_k], & \rho_k \geq \eta_1
                        \text{ and } \Delta_k \leq \zeta\,\omega_k.
                \end{cases}
                \label{eqn: rdfo rule}
            \end{equation}
        \end{update}
        The criticality step, the middle branch, is deterministic: a radius exceeding
        $\zeta\omega_k$ is contracted by the fixed factor $\gamma_2$. It differs from the rule of
        \PartI{} only in restricting that branch to the lower endpoint $\gamma_2\Delta_k$ already
        used there, and satisfies Conditions~\ref{cond: PartI C1}--\ref{cond: PartI C3} with the
        constants \PartI{} computes for it, so every result established there applies verbatim.

        A rejected step reports that the model failed at the current radius, whereas the
        middle branch fires at an accepted step and reports only that the radius exceeds
        $\zeta\,\omega_k$. Implementations therefore take $\gamma_1 < \gamma_2$,
        contracting harder after a rejection than after the criticality test. No result in the paper
        needs the strict inequality, the analysis using only $\gamma_1 \leq \gamma_2$.

        \begin{center}
        \begin{tabular}{@{}ll@{}}
            \toprule
            Name & Measure $\omega_k$ \\
            \midrule
            $\Rdfo(\norm{g})$ & $\norm{g_k}$ \\
            $\Rdfo(\tau)$     & $\tau_k$ \\
            \bottomrule
        \end{tabular}
        \end{center}

    The expansion branch may be bounded above, i.e., $\Delta_{k+1} \in \big[\Delta_k, \min\{\gamma_3 \Delta_k, \overline{\Delta}\}\big]$, but this is not needed for the convergence analysis.

    \subsection{R-grad: gradient-scaled update} \label{app: R-grad}

        The gradient-scaled mechanism sets the radius to an explicit multiple of a criticality measure,
        and admits a cap on that multiple~\cite{FanPan16, WangYuan2022, CurtSche20}. Two independent
        choices are therefore involved. Write $\omega_k$ for the measure, either $\norm{g_k}$ or
        $\tau_k = \max\{\norm{g_k}, -\lambda_k^-\}$, and $\overline{\mu} \in (0, +\infty]$ for the cap,
        the value $+\infty$ meaning none.

        \begin{update}{$\Rgrad(\omega, \overline{\mu})$}[Expansion capped at $\overline{\mu}$]
        \label{update: rgrad generic}
            Let $0 < \gamma_1 \leq \gamma_2 < 1 < \gamma_3$ and
            $\eta \leq \eta_1 \leq \eta_2 < 1$. The radius is $\Delta_k = \mu_k\,\omega_k$ with
            $\mu_0 \leq \overline{\mu}$ and
            \begin{equation}
                \mu_{k+1} \in
                \begin{cases}
                    [\gamma_1 \mu_k,\ \gamma_2 \mu_k], & \rho_k < \eta_1, \\[2pt]
                    [\gamma_2 \mu_k,\ \mu_k), & \eta_1 \leq \rho_k < \eta_2, \\[2pt]
                    \bigl\{\min\{\gamma_3 \mu_k,\ \overline{\mu}\}\bigr\},
                        & \rho_k \geq \eta_2 \text{ and } \norm{s_k} > \tfrac12 \Delta_k, \\[2pt]
                    {\mu_k}, & \text{otherwise}.
                \end{cases}
                \label{eqn: rgrad rule}
            \end{equation}
        \end{update}
        The cap is preserved, $\mu_k \leq \overline{\mu}$ for all $k$ by induction from $\mu_0 \leq \overline{\mu}$,
        and the two contraction branches are bounded below by $\gamma_1\mu_k$, which is all the
        convergence arguments require of them.
        The expansion branch is deterministic, and in particular
        $\mu_k \leq \mu_{k+1} \leq \gamma_3\mu_k$ there, so that
        Condition~\ref{cond: PartI C3} of \PartI{} holds for the radius as well.

        \begin{update}{$\Rgrad(\omega)$}[Uncapped variant]
        \label{update: rgrad uncapped rule}
            Update~\ref{update: rgrad generic} with $\overline{\mu} = +\infty$, so that the expansion
            branch reads $\mu_{k+1} = \gamma_3\mu_k$ and the last branch
            $\mu_{k+1} = \mu_k$.
        \end{update}

        The four instances are named as follows, the abbreviations $\Rgrad(\norm{g})$ and $\Rgrad(\tau)$
        being used where the cap is not at issue.

        \begin{center}
        \begin{tabular}{@{}lll@{}}
            \toprule
            Name & Measure $\omega_k$ & Cap $\overline{\mu}$ \\
            \midrule
            $\Rgrad(\norm{g})$, uncapped & $\norm{g_k}$ & $+\infty$ \\
            $\Rgrad(\norm{g})$, capped   & $\norm{g_k}$ & finite \\
            $\Rgrad(\tau)$, uncapped     & $\tau_k$     & $+\infty$ \\
            $\Rgrad(\tau)$, capped       & $\tau_k$     & finite \\
            \bottomrule
        \end{tabular}
        \end{center}

    \subsection{R-RTR: retrospective update} \label{app: R-RTR}

                The retrospective update~\cite{BastMalmMoufToinToma10} judges the quality of the
        step $s_k$ \emph{retrospectively}, from the model built at the point that step
        produced. Its ratio is
        \begin{equation}
            \widetilde{\rho}_{k+1}
              = \frac{f(x_k) - f(x_{k+1})}{m_{k+1}(-s_k) - m_{k+1}(0)},
            \qquad k \geq 0 ,
            \label{eqn:appendix retrospective ratio}
        \end{equation}
        with $x_{k+1} = x_k + s_k$. The standard ratio $\rho_k$ decides whether the step
        succeeds, while $\widetilde\rho_{k+1}$, read off the model $m_{k+1}$ that governs
        the next iteration, sets the new radius. The rule is indexed so that it produces
        $\Delta_{k+1}$ from the data of iteration $k$, as the four rules above do;
        \citet{BastMalmMoufToinToma10} write it with the index shifted by one.

        \begin{update}{\Rrtr}[R-RTR: retrospective update] \label{update:retrospective}
            At iteration $k \geq 0$, with $\widetilde{\rho}_{k+1}$ given
            by~\eqref{eqn:appendix retrospective ratio} and retrospective thresholds
            $0 < \widetilde\eta_1 \leq \widetilde\eta_2 < 1$:
            if $\rho_k < \eta_1$ (step rejected), set
            \[
                \Delta_{k+1} =
                \begin{cases}
                    \min\!\big\{\gamma_2 \norm{s_k},\, \max[\gamma_1, \vartheta_k]\,\Delta_k\big\},
                        & \rho_k < 0,\\[4pt]
                    \gamma_2 \norm{s_k}, & \text{otherwise},
                \end{cases}
            \]
            where $\vartheta_k$ is the safeguard factor of Part~I; if $\rho_k \geq \eta_1$
            (step accepted), set
            \[
                \Delta_{k+1} =
                \begin{cases}
                    \max\{\gamma_3 \norm{s_k},\; \Delta_k\},
                        & \widetilde{\rho}_{k+1} \geq \widetilde\eta_2,\\[4pt]
                    \Delta_k,
                        & \widetilde{\rho}_{k+1} \in [\widetilde\eta_1, \widetilde\eta_2),\\[4pt]
                    \gamma_2 \norm{s_k},
                        & \widetilde{\rho}_{k+1} \in [0, \widetilde\eta_1),\\[4pt]
                    \min\!\big\{\gamma_2\norm{s_k},\,
                        \max[\gamma_1, \widetilde\vartheta_{k+1}]\,\Delta_k\big\},
                        & \widetilde{\rho}_{k+1} < 0.
                \end{cases}
            \]
        \end{update}

\section{Equivalence of the quasi-Newton consistency conditions}
\label{app: quasi-Newton equivalence}

\begin{theorem}[Equivalence of quasi-Newton consistency conditions]
    \label{thm: equivalence quasi-Newton conditions}
    Suppose $f : \Rp \to \R$ is twice continuously differentiable, that the sequence $\{x_k\}$ converges to a point $x^*$ with $\nabla f(x^*) = 0$ and $\nabla^2 f(x^*)$ positive definite, that the model Hessian is uniformly bounded ($\norm{H_k} \leq \mykappa{umh}$), that the step $s_k$ satisfies the inexactness Condition~\ref{cond:inexactness}, and that $\norm{s_k} \to 0$. Then Assumptions~\ref{as: Gradient residual condition}, \ref{as: Hessian-step residual condition} and \ref{as: Finite-difference Hessian consistency} are equivalent.
\end{theorem}

\begin{proof}
    We prove \ref{as: Hessian-step residual condition} $\Leftrightarrow$ \ref{as: Finite-difference Hessian consistency} (Step~1, which does not use~\eqref{eqn:inexactness}), and then \ref{as: Hessian-step residual condition} $\Leftrightarrow$ \ref{as: Gradient residual condition} (Step~2, which does).

    \smallskip
    \noindent\textbf{Step 1.} \emph{\ref{as: Hessian-step residual condition} $\Leftrightarrow$ \ref{as: Finite-difference Hessian consistency}.}
    Fix $\varepsilon > 0$. Since $f$ is twice continuously differentiable, $x_k \to x^*$ and $\norm{s_k} \to 0$ by hypothesis, for all sufficiently large $k$ and all $\theta \in [0,1]$,
    \begin{equation}
        \|\nabla^2 f(x_k + \theta s_k) - \nabla^2 f(x_k)\| < \tfrac{\varepsilon}{2}.
        \label{eqn:eqv continuity hessian}
    \end{equation}
    The fundamental theorem of calculus gives the splitting
    \begin{equation}
        g(x_k + s_k) - g_k - H_k s_k = \int_0^1 \bigl(\nabla^2 f(x_k + \theta s_k) - \nabla^2 f(x_k)\bigr) s_k \, d\theta + \bigl(\nabla^2 f(x_k) - H_k\bigr) s_k.
        \label{eqn:eqv splitting B C}
    \end{equation}

    \emph{\ref{as: Hessian-step residual condition} $\Rightarrow$ \ref{as: Finite-difference Hessian consistency}.} For all $k$ sufficiently large, $\|(\nabla^2 f(x_k) - H_k) s_k\| \leq \frac{\varepsilon}{2}\|s_k\|$, and combining with~\eqref{eqn:eqv continuity hessian},
    \[
        \|g(x_k+s_k) - g_k - H_k s_k\| \leq \int_0^1 \tfrac{\varepsilon}{2}\,\norm{s_k}\, d\theta + \tfrac{\varepsilon}{2}\,\norm{s_k} = \varepsilon \norm{s_k}.
    \]

    \emph{\ref{as: Finite-difference Hessian consistency} $\Rightarrow$ \ref{as: Hessian-step residual condition}.} For all $k$ sufficiently large, $\|g(x_k+s_k) - g_k - H_k s_k\| \leq \frac{\varepsilon}{2}\|s_k\|$, and from~\eqref{eqn:eqv splitting B C} together with~\eqref{eqn:eqv continuity hessian},
    \begin{align*}
        \|(H_k - \nabla^2 f(x_k)) s_k\|
          &\leq \|g(x_k+s_k) - g_k - H_k s_k\| \\
          &\qquad + \int_0^1 \|\nabla^2 f(x_k+\theta s_k) - \nabla^2 f(x_k)\|\,\norm{s_k}\, d\theta
          \;\leq\; \varepsilon \norm{s_k}.
    \end{align*}
    Since $\varepsilon > 0$ is arbitrary, Step~1 is complete.

    \smallskip
    \noindent\textbf{Step 2.} \emph{\ref{as: Hessian-step residual condition} $\Leftrightarrow$ \ref{as: Gradient residual condition}.} We use the identity
    \begin{equation}
        g_k + \nabla^2 f(x_k)\, s_k = (\nabla^2 f(x_k) - H_k)\, s_k + (H_k s_k + g_k).
        \label{eqn:eqv identity A B}
    \end{equation}
    The first term is controlled by \ref{as: Hessian-step residual condition}; the second by~\eqref{eqn:inexactness}. We first show that the inexactness term is $o(\norm{s_k})$.

    Suppose for a contradiction that $\{\norm{g_k}/\norm{s_k}\}$ is unbounded. From~\eqref{eqn:inexactness}, $\norm{H_k s_k} \geq \norm{g_k} - \norm{H_k s_k + g_k} \geq (1 - \xi_k) \norm{g_k}$, so along a subsequence $\{k_j\}$, $\norm{H_{k_j}} \geq (1 - \xi_{k_j}) \norm{g_{k_j}}/\norm{s_{k_j}} \to +\infty$, contradicting $\norm{H_k} \leq \mykappa{umh}$. Hence there exists $C > 0$ such that $\norm{g_k}/\norm{s_k} \leq C$ for all $k$ large, and by~\eqref{eqn:inexactness},
    \begin{equation}
        \frac{\|H_k s_k + g_k\|}{\|s_k\|} \;\leq\; \xi_k \, \frac{\|g_k\|}{\|s_k\|} \;\leq\; C \xi_k \;\xrightarrow[k\to\infty]{}\; 0.
        \label{eqn:eqv inexactness vanishes}
    \end{equation}
    Applying~\eqref{eqn:eqv identity A B} and the triangle inequality in both directions,
    \begin{align*}
        \frac{\|g_k + \nabla^2 f(x_k) s_k\|}{\|s_k\|} - \frac{\|H_k s_k + g_k\|}{\|s_k\|}
          &\;\leq\; \frac{\|(\nabla^2 f(x_k) - H_k) s_k\|}{\|s_k\|} \\
          &\;\leq\; \frac{\|g_k + \nabla^2 f(x_k) s_k\|}{\|s_k\|} + \frac{\|H_k s_k + g_k\|}{\|s_k\|}.
    \end{align*}
    Combined with~\eqref{eqn:eqv inexactness vanishes}, this gives the equivalence \ref{as: Hessian-step residual condition} $\Leftrightarrow$ \ref{as: Gradient residual condition}.
\end{proof}

\end{document}

%% file: macros.tex
\newcommand{\R}{{\mathbb{R}}} % reals
\newcommand{\NN}{{\mathbb{N}}} % natural numbers {1, 2, ...}
\newcommand{\reals}{\mathbb{R}}

\DeclareSymbolFont{bbold}{U}{bbold}{m}{n}
\DeclareSymbolFontAlphabet{\mathbbold}{bbold}

\newcommand{\calA}{{\mathcal{A}}}
\newcommand{\calB}{{\mathcal{B}}}
\newcommand{\calC}{{\mathcal{C}}}

\newcommand{\calI}{{\mathcal{I}}}

\newcommand{\calO}{{\mathcal{O}}}

\newcommand{\calR}{{\mathcal{R}}}
\newcommand{\calS}{{\mathcal{S}}}

\newcommand{\defeq}{:=}

\newcommand{\abs}[1]{\left \vert #1 \right \vert}

\newcommand{\norm}[1]{\left \lVert #1 \right \rVert}